\documentclass[leqno]{article}

\title{Duality Theory For Markov Processes: Part I}
\author{Ronald Getoor \\ \medskip
University of California, San Diego\\
9500 Gilman Drive\\
La Jolla, CA 92093-0112\\
rgetoor@math.ucsd.edu\\
(858)534-2640 }
\date{January 19, 2010}

\usepackage{amssymb, amsmath, amsthm,paralist, stmaryrd, fancyhdr, calc}
\usepackage{setspace}

\onehalfspace
\usepackage[left=2.5cm,top=3cm,bottom=4.cm,right=2.5cm,nohead,nofoot]{geometry}

\swapnumbers

\newtheorem{theorem}{Theorem}

\newtheorem{definition}[theorem]{Definition}
\newtheorem{corollary}[theorem]{Corollary}
\newtheorem{lemma}[theorem]{Lemma}

\DeclareMathOperator{\Exc}{Exc}

\DeclareMathOperator{\ess}{ess \ \underset{s\downarrow 0}{lim} \ sup}
\DeclareMathOperator{\essv}{ess \ \underset{v\downarrow s}{lim} \ sup}

\DeclareMathOperator{\essu0}{ess \ \underset{u\downarrow 0}{lim} \ sup}

\numberwithin{equation}{section}
\renewcommand{\theequation}{A.\arabic{equation}}
\renewcommand{\thelemma}{A.\arabic{equation}}

\def\A{\mathcal{A}}
\def\B{\mathcal{B}}
\def\C{\mathcal{C}}
\def\D{\mathcal{D}}
\def\E{\mathcal{E}}
\def\F{\mathcal{F}}
\def\G{\mathcal{G}}
\def\H{\mathcal{H}}
\def\I{\mathcal{I}}
\def\k{\kappa}
\def\P{\mathcal{P}}
\def\O{\mathcal{O}}
\def\N{\mathfrak{N}}
\def\M{\mathfrak{M}}

\def\R{\mathbb{R}}
\def\oR{\overline{\mathbb{R}}}
\def\oA{\overline{A}}
\def\ofM{\overline{\M}}
\def\oM{\overline{M}}
\def\oE{\overline{E}}
\def\oY{\overline{Y}}
\def\ll{\lim\limits}
\def\sumlim{\sum\limits}
\def\wo{\widehat{\omega}}
\def\wO{\widehat{\Omega}}
\def\~O{\widetilde{\Omega}}
\def\wA{\widehat{A}}
\def\wP{\widehat{\P}}
\def\wp{\widehat{p}}
\def\wQ{\widehat{Q}}
\def\wT{\widehat{T}}
\def\ww{\widehat{w}}
\def\wW{\widehat{W}}
\def\wX{\widehat{X}}
\def\wG{\widehat{\G}}
\def\wZ{\widehat{Z}}
\def\tQ{\widetilde{Q}}
\def\l{\left}
\def\r{\right}

\begin{document}

\maketitle


\pagebreak

\section{Introduction}

\large
This is the first part of a proposed monograph on duality theory for an arbitrary Borel right process, $X$, relative to a fixed excessive measure $m$. It is perhaps surprising that one may develop such a rich duality theory under such minimal hypotheses. However, the theory differs from the more standard duality theory as presented in [BG68] or, more recently, in [CW05] in that the dual process, $\hat{X}$, is a left continuous moderate Markov process rather than a right continuous strong Markov process. My eventual aim is to demonstrate the power and usefulness of this more general duality theory. \\

But there is a price to be paid for working under these more general hypotheses. Namely there are some technical measure theoretic issues that must be faced. I have attempted to treat the facts that are required honestly and completely. My guiding principle has been to prove all results that are not available in the standard books on the subject. The first volume of Dellacherie and Meyer's treatise ``Probaliti\'es et Potentiel'', [DM] is the standard reference for the necessary measure theory and I have tried to give precise references to the results of this volume as needed. My one deviation from this principle is the last part of the proof of Theorem 5.28 where Dellacherie's  theorem [D88] characterizing $\mu$-semipolar sets is used. \\

In this first part we prove the existence of the moderate Markov left continuous dual process and develop its basic properties following Fitzsimmons [F87]. The final section of this part contains one of the most important applications of the preceding results: The correspondence between optional copredictable homogeneous random measures and a class of $\sigma$-finite measures on the state space of $X$, once again following [F87]. Although this does not involve the dual process $\hat{X}$ directly, it depends heavily on the results and techniques of the previous sections. This first part of the proposed monograph complements and extends the results of my previous monograph [6]. I give precise references to the results from [6] that are used herein. In Part II, I intend to develop the applications of the duality established in Part I to the potential theory of the underlying Markov process. Since it is not clear when, or even if, Part II will be completed, I have decided to make this first part available online. \\

I now shall describe briefly results in the sections to follow. Section 2 contains a summary of the by now standard facts that will be used in the sequel. It begins with precise statements of the hypotheses to be assumed throughout and of the notation to be used without special mention. Then basic definitions and results are recalled. Of special importance is the Kuznetsov process $(Y,Q_m)$ associated with a Borel right process, $X$, and an excessive measure $m$. We refer the reader to [DMM] for its existence, and to [DMM] and [G] for its properties. Theorem 2.13 gives the crucial maximal extension of the strong Markov property to the birth time $\alpha$. It uses the extended process defined in (2.10) which is the extended process $Y^*$ defined in [G] rather than $\overline{Y}$ used in Fitzsimmons' original paper [F87]. This change is important in later sections; $Y^*$ is, in some sense, the maximal possible extension to $\alpha$ maintaining the strong Markov property there. Although this result is proved in [G], we give an alternate proof which, perhaps, is simpler and more transparent following [FM86] in an appendix to Section 2. The final result in this section is a measure theoretic lemma due to Meyer [M71] that we prepare for use in Section 4. \\

Section 3 contains the required results from the ``general theory'' as extended to cover applications to the Kuznetsov process. The optional and copredictable $\sigma$-algebras are defined and their basic properties established. The key results are the existence of optional and copredictable projections. The results in the present setting are due to Fitzsimmons [F87]. However we give complete proofs as there are some technical difficulties because $Q_m$ is not a finite measure in general. In the copredictable situation, we first construct an appropriate predictable projection. The desired copredictable projection is then obtained by reversing time. The optional case also requires some care since we require an optional projection that is maximally extended to a portion of the birth time which does not follow by a straightforward extension of the classical situation. Section 4 is devoted to establishing the existence and uniqueness of the dual process $\hat{X}$. This, of course, is the crucial result for the entire theory and under these hypotheses is due to Fitzsimmons [F87]. For earlier work see also Azema [A73] and Jeulin [J78]. Our proof follows Fitzsimmons quite closely. The reader may want to omit the technical details on a first reading. The section closes with a version of the familiar commutation of projections. \\

The final section of this part is devoted to the proof of the correspondence between optional copredictable homogeneous random measures and $\sigma$-finite measures on the state space not charging $m$-exceptional sets. Theorem 5.28 is the statement of the fundamental result. It also is due to Fitzsimmons [F87] and our proof follows his original proof quite closely with one exception. Namely, in proving the existence of a ``very good'' version of an optional copredictable homogeneous random measure we make use of Meyer's master perfection theorem [M74] rather than a result from [DG85]. The version of Meyer's theorem needed for Theorem 5.28 is stated as Theorem 5.27. Meyer's theorem is proved in [G], but the proof there makes use of perfection theorems in [S]. In order to have a complete proof of the important Theorem 5.28, we give a detailed proof of Theorem 5.27 in an appendix. The proof is somewhat simpler under the hypotheses needed for Theorem 5.28 and, hopefully, more accessible than the proof in [G]. \\

We close this introduction with a few words on notation. We shall use $\mathcal{B}$ to denote the Borel subsets of the real line $\mathbb{R}$ and $\mathcal{B}^+$ for the Borel subsets of $\mathbb{R}^+ = \{x \in \mathbb{R}: x \ge 0\}$. If $(F, \mathcal{F}, \mu)$ is a measure space, then $b\mathcal{F}$ (resp. $p\mathcal{F}$) denotes the class of bounded real-valued (resp. $[0,\infty]$-valued) $\mathcal{F}$-measurable functions on $F$. For $f \in p\mathcal{F}$ we shall use $\mu(f)$ to denote the integral $\int_F f d\mu$; similarly if $D \in \mathcal{F}$ then $\mu(f;D)$ denotes $\int_D f d\mu$. We write $\mathcal{F}^*$ for the universal completion of $\mathcal{F}$; that is, $\mathcal{F}^* = \cap_\nu \mathcal{F}^\nu$, where $\mathcal{F}^\nu$ is the $\nu$-completion of $\mathcal{F}$ and the intersection runs over all finite measures on $(F,\mathcal{F})$. If $(E, \mathcal{E})$ is a second measurable space and $K=K(x,dy)$ is a kernel from $(F,\mathcal{F})$ to $(E,\mathcal{E})$ (\textit{i.e.,} $F \ni x \rightarrow K(x,A)$ is $\mathcal{F}$-measurable for each $A \in \mathcal{E}$ and $K(x,\cdot)$ is a measure on $(E,\mathcal{E})$ for each $x \in F)$, then we write $\mu K$ for the measure $A \rightarrow \int_F \mu(dx) K(x,A)$ and $Kf$ for the function $x \rightarrow \int_E K(x,dy) f(y)$. If $x \in E$, $\epsilon_x: \mathcal{E} \rightarrow [0,1]$ denotes the unit mass at $x$, sometimes called the Dirac measure at $x$.


\def \comp \ {{\leavevmode
     \raise.2ex\hbox{${\scriptstyle\mathchar"020E}$}}}
\noindent\textbf{2. Preliminaries}

\medskip\noindent
We fix a Borel right semigroup $P:= \{P_t; t \geq 0\}$ on a Lusin topological space $(E, \E)$.  We shall now explain this in detail.  
$E$ is (homeomorphic to) a Borel subset of a compact metrizable space and $\E$ is the Borel $\sigma$-algebra of $E$.  
For each $t \geq 0, \ P_t$ is a subMarkovian kernel on $(E,\E)$ and $P_0 (x,\cdot)= \varepsilon_x$.  In particular for $t \geq 0, \ x \in E, \ B \in \E, \ x \to P_t(x,B)$ is $\E$ measurable, $0 \leq P_t(x,B) \leq P_t (x,E) \leq 1$ and $\ll_{t \to 0} P_t(x,E) =1$ for all $ x \in E$.  Of course, $P_{t+s} =P_t P_s =P_sP_t$ for $s,t \geq 0$.  In order to account for the fact that $P_t(x,E)$ may be strictly less than one we adjoin a point $\Delta$ to $E$ as an isolated point and write $E_\Delta := E \cup \{ \Delta\}$ and $\E_\Delta:= \sigma( \E, \{\Delta\})$.  As usual $P_t$ is extended to $(E_\Delta, \E_\Delta)$ in the standard manner so that $\Delta$ is a trap and the extend semigroup again denoted by $P=\{P_t; t \geq 0\}$ is Markovian on $(E_\Delta, \E_\Delta)$.  In particular $P_t(x,\{\Delta\}) =1-P_t(x,E)$ for $x \in E$.  We adopt the standard convention that functions (measures) defined on $E \ (\E)$ are extended to $E_\Delta \ (\E_\Delta)$ by setting them equal to zero at $\Delta \ (\{\Delta\})$ unless explicitly stated otherwise.  Let $\widetilde{\Omega}$ denote the set of all right continuous functions $\omega: [0,\infty] \to E_\Delta$ with the properties 
\begin{inparaenum}[(i)]
\item
$\omega(\infty)=\Delta$ and
\item
if $\zeta(\omega):= \inf \{ t \geq 0: \omega (t)=\Delta\}$ 
\end{inparaenum}
then $\omega(t)=\Delta$ for all $t \geq \zeta(\omega)$ with the usual convention that the infimum of the empty set is $+\infty$.  Define $X_t(\omega):= \omega(t)$ and $\F^0 := \sigma(X_t:t \geq0), \ \F^0_t := \sigma (X_s; \ 0 \leq s \leq t)$.  In particular $X_\infty(\omega)=\Delta$.  The crucial assumption on $P$ is the following: 

\medskip
\noindent (2.1)
{\it For each probability measure $\mu$ on $(E, E_\Delta)$ there exists a unique probability measure $P^\mu$ on $(\widetilde{\Omega},\F^0)$ so that $X:= (X_t, t \geq0)$ is a strong Markov process under the law $P^\mu$ with respect to the filtration $(\F^0_t)$.} 

\medskip \noindent
$X$ is called the canonical realization of the Borel right semigroup $P$.  Because $(E, \E)$ is a Lusin topological space, this is equivalent to the conjunction of (HD1) and (HD2) in [DMM XVII; 1.3, 1.5], and (HD3) and (HD4) in [DM XVII; 1.5] are consequences of (2.1).  For notational simplicity we shall often write $X(t)$ for $X_t$ and $X(t,\omega)$ for $X_t(\omega)$.  The same convention will be used for the Kuznetsov process, $Y$, to be introduced shortly.

\medskip\noindent
Having spelled out in detail the definition of a Borel right semigroup we suppose the reader is familiar with the basic properties of such semigroups and its canonical realization $X=(\~O, \F^0, \F^0_t, X_t, \theta_t, P^x)$ where $P^x=P^{\varepsilon_x}$ and $\theta_t :\~O \to \~O$ is given by $(\theta_t \omega)(s) =\omega(s +t)$.  Chapter XVII of [DMM] contains an excellent summary of these properties.  (Note that the numbering system in [DMM] differs slightly from that used in preceding volumes of [DM].)  Other standard references are listed at the beginning of the bibliography.  We fix a Borel right semigroup $P=(P_t; t \geq0)$ and let $X$ be its canonical representation.  The resolvent of $P, \{U^q; \ q \geq0\}$ is given by

\begin{equation} \tag{2.2}
U^q f(x):= \int_0^\infty e^{-qt} P_t f(x) dt 
=P^x \int_0^\zeta e^{-qt} f(X_t) dt
\end{equation}
for $x \in E, \ f \in p\E$.  In view of our convention that $f(\Delta)=0$ one may replace $\int_0^\zeta$ by $\int_0^\infty$ in (2.2).  We adopt the standard convention here and in the sequel to omit the parameter $q$ when $q=0$.  Thus $U=U^0$.  For the most part we shall use the standard notation for Markov processes as in the reference books [BG], [DM], [G] or [S] without special mention.

\medskip\noindent
For ease of reference we recall some standard definitions.  See, for example, [G], [DM] or [DMM].  A $\sigma$-finite measure $\xi$ on $(E, \E)$ is $q$-excessive for $q \geq 0$ provided $e^{-qt} \xi P_t \leq \xi$ for $t \geq 0$.  It is convenient to define $P_t^q:= e^{-qt} P_t$.  Clearly $(P_t^q, t \geq 0)$ is again a semigroup.  It follows that $\xi P_t^q \uparrow \xi$ as $t \downarrow 0$ for a $q$-excessive $\xi$ [DMXII; 37.1].  The class of $q$-excessive measures is denoted by $Exc^q$.  It is standard fact that a $\sigma$-finite measure $\xi$ is in $Exc^q$ if and only if $r \xi U^{q+r} \leq \xi$ for $r >0$.  Then $r \xi U^{q+r} \uparrow \xi$ as $r \uparrow \infty$.  There are several important subclasses of excessive measures.  Here we give the basic definitions and refer to [G] for their properties.  We write them for $q=0$ only, the extensions of these definitions to $q>0$ being clear.  Let $\xi \in Exc$.  Then $\xi$ is {\it purely excessive} provided $\xi P_t \to 0$ as $ t \to \infty$ and $\xi$ is {\it invariant} provided $\xi P_t =\xi$ for all $t \geq 0$.  $\xi$ is a {\it potential} provided $\xi=\mu U$ for some necessarily unique and $\sigma$-finite measure $\mu$.  $\xi$ is {\it harmonic} provided it strongly dominates no non-zero potential; that is, if $\xi =\mu U+\eta$ with $\eta \in Exc$, and $\eta \neq 0$, then $\mu=0$.  $\xi$ is {\it dissipative} provided it is the supremum of all potentials $\mu U$ with $\mu U \leq \xi$.  Then $\xi$ is, in fact, an increasing limit of a sequence $(\mu_nU)$ of potentials.  Finally $\xi$ is {\it conservative} provided it dominates no non-zero potential in natural order of measures; that is, if $\mu U \leq \xi$ as measures, then $\mu=0$.  These classes of excessive measures are denoted by $Pur, \ Inv, \ Pot, \ Har, \ Dis$ and $Con$ respectively.  There correspond three unique Riesz type decompositions of $Exc$:

\medskip
\noindent
\begin{equation}\tag{2.3}
\begin{cases}
&(i) \ Exc= Pur +Inv,\\
&(ii) \ Exc= Pot + Har,\\
&(iii) \ Exc= Dis + Con.
\end{cases}
\end{equation}

\medskip\noindent
An {\it entrance rule} is a family $\nu= (\nu_t; t \in \R)$ of $\sigma$-finite measures on $(E,\E)$ such that $\nu_s P_{t-s} \leq \nu_t$ for $s<t$ and $\nu_sP_{t-s} \uparrow \nu_t$ as $s \uparrow t$.  Such an entrance rule is called a regular entrance rule in [DMM XIX; \S 2], but this is the only type we shall consider.  If $\xi \in Exc$, then $\nu_t:= \xi$ for $t \in \R$ clearly defines an entrance rule.  An {\it entrance law at} $s \in \R$ is an entrance rule $\nu$ such that $\nu_t=0$ for $t \leq s$ and $\nu_t P_r= \nu_{t+r}$ for $t>s$.  An {\it entrance law} is an entrance law at $s=0$.

\medskip\noindent
We are now ready to introduce the notation necessary to describe Kuznetsov measures.  Let $W$ be the space of paths $w: \R \to E \cup \{\Delta\}$ that are $E$-valued and right continuous on an open interval $] \alpha(w), \beta(w) [$ and take the value $\Delta$ off this interval.  The path $[\Delta]$ constantly equal to $\Delta$ corresponds to $]\alpha,\beta[$ being empty and we set $\alpha([\Delta]) =+\infty, \ \beta([\Delta])=-\infty$.  Define $w(-\infty)= w(\infty)=\Delta$ for $w \in W$ and let $(Y_t, t \in \R)$ denote the coordinate process on $W, \ Y_t(w)=w(t)$ and $Y_{\pm \infty}(w) =w(\pm \infty) =\Delta$.  Define $\G^0 = \sigma (Y_t; \ t \in \R)$ and $\G^0_t = \sigma (Y_s; \ s \leq t)$.  Note that $\alpha$ is a stopping time relative to the filtration $(\G^0_{t+}, \ t \in \R)$.  We can now state the basic existence theorem for Kuznetsov measures.  Section 9 of [DMM XIX] contains a nice proof.  Of course our assumption that $P= (P_t, t \geq 0)$ is a Borel right semigroup is still in force.

\noindent 
{\bf 2.4 Theorem}
{\it
Let $\nu=(\nu_t)$ be an entrance rule.  Then there exists a unique measure $Q_\nu$ on $(W,\G^0)$ not charging $\{[\Delta]\}$ such that if $t_1< \dots, t_n$,
\begin{equation}\tag{2.5}
\begin{split}
Q_\nu 
&(\alpha< t_1, Y_{t_1} \in dx_1, \dots, Y_{t_n} \in dx_n, t_n < \beta)\\
&= \nu_{t_1} (dx_1) P_{t_2-t_1}(x_1, dx_2)
\dots P_{t_n -t_{n-1}} (x_{n-1}, dx_n).
\end{split}
\end{equation}
Moreover $Q_\nu$ is $\sigma$-finite.
}

\medskip \noindent
If $m \in Exc$ and $\nu_t =m$ for all $t$ we write $Q_m$ in place of $Q_\nu$.  $Q_\nu$ or $Q_m$ is called the Kuznetsov measure of $\nu$ or $m$.  Let $\G^m$ (resp. $\G^\nu$) denote the completion of $\G^0$ with respect to $Q_m$ (resp. $Q_\nu$) and $\G_t^m$ (resp. $\G_t^\nu$) denote the $\sigma$-algebra generated by $\G_t^0$ and all $Q_m$ (resp. $Q_\nu$) null sets in $G^m$ (resp. $\G^\nu$).  The process $(Y_t, Q_m, t\in \R)$ (resp. $(Y_t, t \in \R, Q_\nu)$) is called the Kuznetsov process of $m$ (resp. $\nu$).  Define shift operators $\sigma_t, t \in \R$ on $W$ by
\[
(\sigma_t w)(s) =w(t+s), \ s\in\R
\]

\medskip\noindent
If $m \in Exc$, then only differences of $t$-values appear on the right side of (2.5) when $\nu_t =m$ for all $t$.  Therefore $Q_m$ is stationary; that is, $\sigma_t (Q_m) =Q_m$ for $t \in \R$.  But $\alpha \circ \sigma_t =\alpha -t$ and so $Q_m(\alpha=t)= Q_m (\alpha=0)$ for each $t$.  Since $Q_m$ is $\sigma$-finite this implies that $Q_m (\alpha=t)=0$ for each $t$.  Similarly $Q_m(\beta=t)=0$ for each $t$.  The following Markov property of the process $Y$ under $Q_\nu$ is a standard consequence of (2.5).

\medskip \noindent
{\bf 2.6 Proposition}
{\it
For each $t \in \R$ the process $(Y_{t+s}; s \geq0)$ under $Q_\nu$ restricted to $\{ \alpha < t< \beta\}$ has the same law as $X$ under $P^{\nu_t}$.  In particular $Q_\nu$ is $\sigma$-finite on $\G_t^t:= \G_t^0 \mid _{\{ \alpha < t< \beta\}}$ and if $H \in p \sigma (Y_s; s \geq t),$ then $Q_\nu (H \mid \G_t^t)$ is $\sigma (Y_t) \mid_{\{ \alpha< t< \beta\}}$ measurable.
}

\medskip\noindent
Next we introduce for $t \in \R$ the birthing $b_t$, killing $k_t$ and truncated shift $\theta_t$ operators on $W$ as follows:

\begin{equation}\tag{2.7}
\begin{cases}
&(b_tw)(s)=w(s) \text{ if } s>t, \ (b_t w) (s) =\Delta \text{ if } s \leq t;\\
&(k_tw)(s)=w(s) \text{ if } s<t, \ (k_t w) (s) =\Delta \text{ if } s \geq t;\\
&(\theta_tw)(s)=w(s+t) \text{ if } s>0, \ 
(\theta_t w) (s) =\Delta \text{ if } s \leq 0.\\
\end{cases}
\end{equation}

\noindent
Note that $\theta_t =b_0 \sigma_t =\sigma_t b_t$ and so $\theta_t \sigma_s =\theta_{t+s}$.  In order to more conveniently express the relationship between the Borel right process $X$ and the Kuznetsov process $Y$ define

\[
\Omega =\{ w \in W: \alpha(w) =0 \text{ and } w(0+) \text{ exists in } E \} \cup \{[\Delta]\}.
\]

\noindent
Clearly $\Omega$ may be identified with the canonical path space $\~O$ of $X$ and defining $X_t =Y_t \mid_\Omega$ if $t>0, \ X_0(\omega) =Y_{0+} (\omega)$ for $\omega \in \Omega$, then $(X,\Omega)$ is isomorphic to the canonical realization of $(P_t)$.  This is the realization that will be used henceforth.  
Observe that $\theta_t w \in \Omega$ if 
$ t \in \mathbb{R} \mbox{\ and\ } \alpha(w) \leq t < \infty$.
Also observe that $\F^0 := \sigma (X_t; t \geq 0) =\G^0 \mid_\Omega = \G_{>0}^0 \mid_\Omega$ and that $\F_t^0 := \sigma (X_s; s \leq t) = \G_t^0 \mid_\Omega$ if $t>0$ while $\F_0^0 =\G_{0+}^0 \mid_\Omega$.

\medskip\noindent
The strong Markov property of $(Y_t)$ under $Q_\nu$ is proved in [DDM XIX; \S7,8] although it is not explicitly stated there.  We state it here as a proposition.

\medskip\noindent
{\bf 2.8 Proposition}
{\it
Let $T:W \to \oR =[-\infty, \infty]$ be a $(\G_{t+}^0)$ stopping time and $F \in p \F^0$, then
\[
Q_\nu (F \circ \theta_T \mid \G_{T+}^0) =P^{Y(T)} (F) 
\text{ on } \{ \alpha < T < \beta \}
\]
and $Q_\nu$ is $\sigma$-finite on the trace, $\G_{T+}^0 \mid_{\{ \alpha<T<\beta\}}$ of $\G_{T+}^0$ on $\{ \alpha <T <\beta\}$.
}

\medskip\noindent
The most important case in the sequel is $\nu_t =m \in Exc$ for all $t$.  In particular the filtration $(\G_t^m)$ is right continuous.

\medskip\noindent
It is necessary to extend (2.8) to include $\{T=\alpha \}$ as far as possible.  For this we introduce a modified process $Y^\ast$ as in [G; 6.12].  This process will play a crucial role in the sequel.  To this end let $d$ be a totally bounded metric on the Lusin space $E$ compatible with the topology of $E$, and let $\D$ be a countable uniformly dense subset of the $d$-uniformly continuous bounded real-valued functions on $E$.  Given a strictly positive $b \in b\E$ with $m(b)<\infty$ defined $W(b) \subset W$ by the conditions:

\begin{equation}\tag{2.9}
\begin{cases}
& \text{(i) } \alpha \in \R,\\
& \text{(ii) } Y_{\alpha +} := \lim_{t \downarrow \alpha} Y_t
\text{ exists in } E,\\
& \text{(iii) } U^q f \circ Y_{\alpha +1/n} \to U^q f \circ Y_{\alpha +} 
\text{ as } n \to \infty, \text{ for all } f \in \D \text{ and rational } q>0,\\
& \text{(iv) }Ub \circ Y_{\alpha +1/n} \to Ub \circ Y_{\alpha +}
\text{ as } n \to \infty.
\end{cases}
\end{equation}

\noindent
It is evident that $\sigma_t^{-1} W(b) =W(b)$ for all $t \in \R$ and $W(b) \in \G_{\alpha +}^0$.  Define

\begin{equation}\tag{2.10}
Y_t^\ast(w) =
\begin{cases}
Y_{t+}(w) &\text{ if } t= \alpha(w) \text{ and } w \in W(b),\\
Y_t(w) &\text{otherwise.}
\end{cases}
\end{equation}

\noindent
Note that $Y_t^\ast (w) =Y_t(w)$ if $w \not \in W(b)$ and that $Y_t^\ast (w) =Y_{t+} (w)$ if $w \in W(b)$.  Since $W(b) \in \G_{\alpha +}^0$ it follows that $Y^\ast$ is $(\G_{t+}^0)$ adapted.  Hence $Y^\ast$ is optional relative to the filtration $(\G_{t+}^0)$.  In particular $Y^\ast$ is $(\G_t^m)$ optional.  Moreover $Y_t^\ast(w)=\Delta$ if $-\infty <t < \alpha(w)$.  The set

\begin{equation}\tag{2.11}
\Lambda^\ast:= \{ (t,w) : Y_t^\ast (w) \in E \}
\end{equation}
will play an important role if the sequel.  It is evident that $\theta_t w \in \Omega$ if $(t,w) \in \Lambda^\ast$ and that $\rrbracket \alpha,\beta \llbracket  \subset \Lambda^\ast \subset \llbracket  \alpha, \beta \llbracket$.  Observe that $W(b) \cap \{ \alpha =0 \} \subset \Omega \smallsetminus \{[\Delta ]\}$ and since $X$ is a Borel right process $P^x(\Omega \smallsetminus W(b))=0$ if $x \in E$.  Note that $W(b)$ depends on $m$ and $b$ only through (2.9-- iv).  Fix $m \in Exc$ and $b$ as above.  If $m = \eta + \pi = \eta + \rho U$ is the decomposition of $m$ into its harmonic and potential parts (2.3), then $Q_m = Q_\eta + Q_\pi$ and according to [G; 6.19], $Q_\eta (W(b))=0$ and $Q_\pi =Q_m (\cdot; W(b))$.  In particular, if $b'$ is another function with the properties of $b$, then $Q_m (W(b) \Delta W(b'))=0$.

\medskip
\noindent
{\bf 2.12 Remark }
{\it
Let $\oE$ be a Ray-Knight compactification of $E_\Delta$.  See [S; \S 17-18], [DM XII; \S 5] or [G75].  Since $X$ is a Borel right process in its Ray topology, one may define $W^r (b)$ analogously to $W(b)$ when $E$ is given the Ray topology.  In this case (iii) of (2.9) is automatically satisfied by the very definition of the Ray-Knight compactification.  Therefore $W^r(b)$ is characterized by (i), (ii) and (iv) of (2.9) with $Y_{\alpha +}$ replaced by $Y_{\alpha +}^r := \lim_{t \downarrow \alpha} Y_t$ where the limit is taken in the Ray topology.  The decomposition $m=\eta + \pi$ into harmonic and potential parts depends only on the resolvent and so $Q_m (W(b) \Delta W^r (b))=0$.
}

\medskip\noindent
We are now able to state the important extension of (2.8) mentioned earlier.  It is Proposition 6.15 in [G] and is proved there.  However for completeness we shall give an alternate proof which, perhaps, is simpler and more transparent.  but in order not to interrupt the flow of this section we defer it to an appendix for this section.

\medskip \noindent
{\bf 2.13 Theorem}
{\it 
Let $T$ be a $(\G_t^m)$ stopping time.  Then $Q_m$ is $\sigma$-finite on $\G_T^m \mid_{\{Y^\ast(T) \in E \}}$ and 
\[
Q_m (F \circ \theta_T \mid \G_T^m) = P^{Y^\ast (T)}(F) \text{ on }
\{ Y^\ast_T \in E \}
\]
for $F \in p \F^\ast$.
}

\noindent
Note that $\{ Y_T^\ast \in E \} = \{ \alpha \leq T<\beta\} \cap W(b)$.

\medskip\noindent
We remind the reader that a subset $A \subset E$ is $m$-polar (resp. $m$-semipolar) provided $A \subset B, \ B \in \B$ such that a.s. $Q_m, \ \{ t: Y_t \in B\}$ is empty (resp. countable).

\medskip\noindent
For $m \in Exc$, let $m = \eta + \rho U$ be its decomposition into its harmonic and potential parts.  An arbitrary subset $A \subset E$ is $m$-{\it exceptional} provided $A$ is contained is a Borel $m$-polar set $B$ with $\rho(B)=0$.  Let $\N(m)$denote the class of $m$-exceptional sets.  It is easily checked that $B \in \E$ is $m$-polar if and only if $\{(t,w) \in \R \times W: Y_t (w) \in B\}$ is $Q_m$ evanescent.  There is a similar characterization of $m$-exceptional sets.

\medskip\noindent
{\bf 2.14 Proposition}
{\it
$B \in \N(m) \cap \E$ if and only if 
\[
\{(t,w) \in \R \times W: Y_t^\ast(w) \in B\}
\]
is $Q_m$ evanescent.
}

\begin{proof}
Since $Y$ and $Y^\ast$ may differ only at $\alpha \in \R$ on $W(b)$, it suffices to show that if $B \in \E$ is $m$-polar, then $\rho(B)=0$ is equivalent to $Q_m (Y^\ast_\alpha \in B)=0$.  But this is clear from [G; 6.20] which states in our current notation that $Q_m(Y_\alpha^\ast \in B; 0< \alpha < 1) = \rho (B)$ and from the stationarity of  $Q_m$.
\end{proof}

\medskip\noindent
We conclude this section with a lemma that will be used later.  The reader may want to skip it now and just refer it later as needed.  It is due to Meyer [M71].

\medskip\noindent
{\bf 2.15 Lemma}
{\it 
Let $(E,\E)$ be an arbitrary measurable space, $M$ a separable metrizable space and $\M$ its Borel $\sigma$-algebra. 
Let $N$ be a kernel from $(E, \E)$ to $(M,\M)$ such that
$N(x,M)< \infty$.  If $A$ is a Souslin subset of $M$, then $\{x: N(x,A)>0\}$ is a Souslin subset of $E$.
}

\begin{proof}
Since $\{ x: N(x,M) =0 \} \in \E$ we may delete it from $E$ and suppose that $N(x,M)>0$ for all $x \in E$.  Then replacing $N(x,\cdot)$ by $N(x, \cdot)/ N(x,M)$ we may suppose without loss of generality that $N(x, \cdot)$ is a probability for each $x \in M$.  But $M$ is homeomorphic to a subspace of a compact metric space $\oM$ and so we may suppose without loss of generality that $M \subset \oM$. Then $\M= \ofM \mid_M$ where $\ofM$ is the Borel $\sigma$-algebra of $\oM$.  Extending $N$ to $\ofM$ by $N(x,B) =N(x, B \cap M)$ for $B \in \ofM, N$ becomes a kernel from $(E, \E)$ to $(\oM, \ofM)$ that is carried by $M$.  Since $A$ is Souslin in $M$, it is analytic in $M$ [DM III; 19(2)] and therefore $A=\oA \cap M$ where $\oA$ is analytic in $\oM$ [DM III; 15].  Let $\P(\oM)$ be the set of all probability measures on $\oM$ equipped with the weak topology generated by the continuous functions from $\oM$ to $\R$.  Now $1_{\oA}$ is an analytic function on $\oM$ as defined in [DM III;61].  Hence by [DM III;62] the function $\mu \to \mu(\oA)$ is an analytic function on $\P(\oM)$.  Therefore $J:= \{ \mu \in \P(\oM): \mu (\oA)>0 \}$ is analytic, and hence Souslin in $\P(\oM)$.  Since the map $h:x \to N(x,\cdot)$ is measurable from $E$ to $\P(\oM), \ h^{-1}(J)$ is  Souslin in $E$.  The fact that the inverse image of a Souslin set under a measurable map is Souslin follows easily from the characterization of Souslin sets in terms of Souslin schemes.  See [DM III; 75-77].  But $\{x:N(x,A)>0\} =\{ x: N(x,\oA)>0 \}=h^{-1}(J)$ establishing (2.15).
\end{proof}

\medskip\noindent
{\bf 2.16 Remark}
{\it 
It is important for later applications that the spaces $\Omega$ and $W$ are co-Souslin.  For $\Omega$ this is [DM III; 19b].  A similar argument works for $W$ also.
}

\medskip\noindent
{\bf 2.17 Remark}
{\it 
Since $\theta_T w \in \Omega$ if $Y_T^\ast (w) \in E, \ F 1_\Omega \in p \F^0$ if $F \in p \G^0$ and $P^x$ is carried by $\Omega$, it follows that (2.13) is valid for $F \in p\G^0$.  We shall use (2.13) for $F \in p\G^0$ without special mention in the sequel.
}


Finally, for ease of reference, we record the monotone class theorem [DMI; 21] as improved by Bruce Atkinson---see the additions and corrections on page 231 of the volume containing Chapters IX to XI of [DM].

\medskip\noindent
{\bf (2.18) Theorem}
{\it
Let $E$ be a non-empty set and let $\H$ be a vector space of bounded
real-valued functions on $E$.  Suppose that $\H$ contains the constants and that
$\H$ is closed under monotone increasing limits
of uniformly bounded sequences of positive functions in $\H$; that is, 
if $(f_n) \subset p\H, \mbox{sup}\{f_n(x): x \in E\} \leq M < \infty$ for
$ n \geq 1$ and $f_n \uparrow f$, then $f \in \H$.  Let $\H_0$ be a subset of $\H$ closed under finite products.  Then $b \sigma(\H_0) \subset H$.
}

\vspace{.42in}
\noindent
{\bf Appendix}

\medskip
This appendix contains an alternate proof of Theorem 2.13.  In essence it has the same structure as the proof in [G] but the details are somewhat different.  In view of (2.8) and since 
$\{Y^\ast_T \in E\}= \{ \alpha <T< \beta\} \cup 
\{ T = \alpha \} \cap W(b)$ 
it suffices to show that 
$Q_m \l( F \circ \theta_T \mid \G_{T+}^0 \r) = P^{Y^\ast(T)} (F)$
on $\{ T = \alpha \} \cap W(b)$ and that $Q_m$ is $\sigma$-finite of 
$\G_{T+}^0 \mid_{\{Y^\ast (T) \in E\}}$.  The equality is equivalent to
\[
Q_m \l( F \circ \theta_T; \ T=\alpha, \ \Lambda, \ W(b) \r)
=Q_m \l[ P^{Y^\ast(T)} (F); \ T= \alpha, \ \Lambda \cap W(b) \r]
\]
for all $\Lambda \in \G_{T+}^0$.  
Replacing $T$ by $T_\Lambda := T$ on $\Lambda, T_\Lambda := \infty$ on $\Lambda^c$ this is reduced to showing 
\begin{equation}\tag{2.19}
Q_m 
\l[ F \circ \theta_T; \ T=\alpha, W(b) \r]
= Q_m \l[ P^{Y^\ast (T)} (F); \  T = \alpha, \ W(b) \r],
\end{equation}
for all $(\G_{t+}^0)$ stopping times $T$, and $F \in p\F^0$.  Let $m=m_i +m_p$ be the decomposition of $m$ into its invariant, $m_i$, and purely excessive, $m_p$, parts.  Then $Q_m=Q_{m_i} +Q_{m_p}$ and according to [G; 6.7], $Q_{m_i} ( \alpha> -\infty) =0$.  But $W(b) \subset \{ \alpha \in \R \}$ and so in proving (2.17) we may suppose that $m \in Pur$ without loss of generality.  Therefore as is well-known $m=\int_0^\infty \nu_t dt$ where $\nu =(\nu_t; \ t>0)$ is an entrance law (at zero).  This also appears in [G; 6.7].  Let $Q_\nu$ denote the Kuznetsov measure of $\nu$.  $Q_m$ and $Q_\nu$ are related by (see, [G; 6.11]) $Q_m = \int_\R \sigma_t (Q_\nu) dt$.  Also $Q_\nu (\alpha \neq 0)=0$, see page 54 of [G].  Set $\Omega(b) =W(b) \cap \Omega = W(b) \cap \{ \alpha=0 \}$.

\medskip
We begin by showing that $Q_m$ is $\sigma$-finite on $\G_{T+}^0 \mid_{\{ Y^\ast(T) \in E\}}$ and that $Q_\nu$ is $\sigma$-finite on $\G_{0+}^0 \mid_{\Omega (b)}$.  Since $Y_{0+}= Y_0^\ast$ exists in $E$ on $\Omega (b)$
\begin{equation}\tag{2.20}
\begin{split}
Q_\nu 
&[ Ub \circ Y_{0+}; \ \Omega(b) ] 
\leq \lim\inf Q_\nu [Ub \circ Y_{1/n}]\\
&= \lim\inf \nu_{1/n} (Ub) 
\leq \int_{1/n} \nu_t (b) dt 
\leq m(b) < \infty,
\end{split}
\end{equation}
where $b$ comes from (2.9).  Now $Ub \circ Y_{0+} \in \G_{0+}^0$ and $Ub \circ Y_{0+} >0$ on $\Omega(b)$; hence $Q_\nu$ is 
$\sigma$-finite on $\G_{0+}^0 \mid_{\Omega(b)}$. From (2.8), $Q_m$ is $\sigma$-finite on $\G_{T+}^0 \mid_{\{ \alpha <T<\beta\}}$.  But $\{ Y_T^\ast \in E \} = \{ \alpha < T<\beta\} \cup \{ T= \alpha, Y_\alpha^\ast \in E\}$ and $\{ T= \alpha \} \cap \{ Y_\alpha^\ast \in E\} \in \G_{T+}^0 \mid_{\{T=\alpha, \ Y_T^\ast \in E\}}$.  
Recall that $\alpha$ is a $(\G_{t+}^0)$ stopping time.  Consequently it suffices to show that $Q_m$ is $\sigma$-finite on $\G_{T+}^0 \mid_{\{T=\alpha, \ Y_\alpha^\ast \in E\}}$.  Let $g>0$ be a Borel function on $\R$ with $\int g(t) dt < \infty$.  Then using the relationship between $Q_m$ and $Q_\nu$ we find
\begin{align*}
&Q_m 
[g(\alpha) Ub \circ Y_\alpha^\ast; \ T= \alpha, \ Y_\alpha^\ast \in E]
 \leq Q_m [g(\alpha) Ub \circ Y_\alpha^\ast; W(b)]\\
& = \int Q_\nu [g(\alpha-t) Ub \circ Y_\alpha^\ast; \ \Omega(b)]dt
= \int g(t) dt \cdot Q_\nu [Ub \circ Y_{0+}; \ \Omega(b)] < \infty
\end{align*}
in view of (2.20).  The first equality above uses the facts that 
$Y_\alpha^\ast \circ \sigma_t =Y_\alpha^\ast, \ \sigma_t^{-1} W(b) =W(b)$ and that $Q_\nu$ is carried by $\{ \alpha =0\}$.  Hence $Q_m$ is $\sigma$-finite on $\G_{T+}^0 \mid_{\{Y^\ast (T) \in E\}}$.

\medskip
For the next step we require a lemma.

\medskip\noindent
{\bf (2.21) Lemma.}  
{\it Let $F \in p\F^0$ and $\Lambda \in \G_{0+}^0$.  Then for $t \geq 0$,}
\[
Q_\nu [F \circ \theta_t; \ \Lambda \cap \Omega(b)] 
= Q_\nu [P^{Y^\ast(t)} (F); \ \Lambda \cap \Omega(b)].
\]

\begin{proof}
Since $Q_\nu$ is $\sigma$-finite on $\G_{0+}^0 \mid_{\Omega(b)}$ for each $n \geq 1$ there exists $h_n \in \G_{0+}^0 \mid_{\Omega(b)} $ with $0 \leq h_n \leq 1$ and $Q_\nu (h_n; \Omega(b))<\infty$ such that $h_n \uparrow 1$ on $\Omega(b)$ as $n \uparrow  \infty$.  Clearly on $\Omega(b), \ Y_t^\ast = Y_t$ if $t>0$ and $Y_0^\ast= Y_{0+} =X_0$.  Fix $n \geq 1$ for the moment and write $h_n=h$.

\medskip\noindent
Define a finite measure, $\tQ$, on $\F^0$ by $\tQ(F) =Q_\nu (F h; \ \Lambda, \ \Omega(b))$ for $F \in p \F^0$.  If $q>0$ is rational and $f \in \D$ where $\D$ is defined above (2.9), then
\begin{align*}
\int_0^\infty e^{-qt} \tQ [f \circ Y_t] dt
&= \ll_{k \to \infty} 
\int_{1/k}^\infty e^{-qt} \tQ [ f \circ Y_{t-1/k} \circ \theta_{1/k} ]dt\\
&= \ll_{k\to\infty} e^{q/k} \tQ [U^q f \circ Y_{1/k}]\\
&= \tQ [U^q f \circ Y_{0+}] 
= \int_0^\infty e^{-qt} \tQ [P_t f \circ Y_{0+}] dt
\end{align*}
where the second equality follows from the Markov property of $(Y_t, \ t>0)$ under $Q_\nu$ and the third equality from the definition of $W(b)$.  Then by continuity the extreme terms in this last display are  equal for all $q>0$.  
Both $t \to \tQ [f \circ Y_t]$ and $t \to \tQ [P_t f \circ Y_{0+}]$ are right continuous in $t$ and so the uniqueness of Laplace transforms implies that $\tQ [f \circ Y_t] =\tQ [P_t f \circ Y_{0+}]$ for $t>0$ and $f \in \D$.  A monotone class argument using the definition of $\D$ and the finiteness of $\tQ$ yields that this last equality holds, in fact, for all $f \in b \E$.  Now recall the definition of $\tQ$ and that $h=h_n$.  Then letting $n \to \infty$ we obtain
\begin{equation}\tag{2.22}
Q_\nu[ f \circ Y_t; \ \Lambda \cap \Omega(b)]
= Q_\nu [P_t f \circ Y_{0+}; \ \Lambda \cap \Omega(b)].
\end{equation}

\medskip
If $t>0$, the conclusion of (2.21) is an immediate consequence of the Markov property of $(Y_t; \ t >0)$ under $Q_\nu$.  If $t=0$, it suffices to verify (2.21) for $F$ of the form $F= \Pi_{j=0}^k f_j \circ X_{tj}$ with $f_j \in p \E, \ 0 \leq j \leq k$ and $0= t_0< t_1< \dots< t_k$.  
But then $F$ takes the form $F=f_0 \circ Y_{0+} G \circ \theta_{t_1}$ on $\Omega(b)$ where $G= \Pi_{j=1}^k f_j \circ X_{t_j-t_1}$.  
Define $\varphi(x)= P^x(G)$.  Then using the Markov property of $(Y_t; t>0)$ under $Q_\nu$ for second quality below,

\begin{align*}
Q_\nu [F \circ \theta_0; \ \Lambda \cap \Omega(b)]
&= Q_\nu [f_0 \circ Y_{0+} G \circ \theta_{t_1}; \ 
\Lambda \cap \Omega(b)]\\ 
&= Q_\nu [f_0 \circ Y_{0+} P^{Y(t_1)} (G); \ \Lambda \cap \Omega(b)]\\
&=Q_\nu [f_0 \circ Y_{0+} P_{t_1} \varphi (Y_{0+}); \ 
\Lambda \cap \Omega(b)]\\
&=Q_\nu [P^{Y^\ast(0)}(F); \ \Lambda \cap \Omega(b)]
\end{align*}
where the third equality comes from (2.21) and the last equality holds because of the Markov property of $(X_t; \ t \geq 0)$ under $P^x$ and $P^x(X_0=x)=1$ for $x \in E$.  This completes the proof of (2.21).
\end{proof}

\medskip
Returning to the proof of Theorem 2.13, it remains to verify (2.19) for $m \in Pur$.  Let $\nu =(\nu_t; t>0)$ be the entrance law corresponding to $m$ so that $Q_m = \int \sigma_t (Q_\nu) dt$.  Then the left hand side of (2.19) equals
\[
Q_m [F \circ \theta_\alpha; \ T=\alpha, \ W(b)] 
= \int Q_\nu [F \circ \theta_\alpha; \ T \circ \sigma_t = \alpha-t ; \
\Omega(b)] dt
\]
since $\theta_\alpha \circ \sigma_t = \theta_\alpha, \ \sigma_t^{-1} W(b)=W(b)$ and $Q_\nu$ is carried by $\{ \alpha=0 \}$.  But $T_t := T \circ \sigma_t + t$ is easily seen to be a $(\G_{t+}^0)$ time for each $t$.  Hence using (2.21) with $t=0$, the left hand side of (2.19) becomes since $Q_\nu$ is carried by $\{ \alpha =0\}$ and $\{ T_t=0 \} \in \G_{0+}^0$,
\[
\int Q_\nu [F \circ \theta_0; \ T_t =0, \ \Omega(b)] dt
= \int Q_\nu [P^{Y^\ast (0)}(F); \ T_t=0, \ \Omega(b)]dt.
\]
A similar argument shows that the right hand side of (2.19) reduces to the same expression completing the proof of Theorem 2.13.

\newpage
\noindent
{\bf 3. Some General Theory}

\medskip
In this section we shall establish some general theory for the system $(W,\ \G^0, \ \G_t^0, \ Y_t, \ Q_m)$ with $m \in Exc$ that we shall need in later sections.  
Most of these results come from section 3 of [F87].  
However we use the process $Y^\ast$ defined in (2.10) systematically rather the process $\oY$ used in [F87].  
Define $\G_{\geq t}^0$, resp. $\G_{>t}^0$, to be the $\sigma$-algebra generated by $\{Y_s; s \geq t\}$, resp $\{Y_s; s>t\}$.  
Recall that $\G^m$ is the completion of $\G^0$ and that $\G_t^m$ is the $\sigma$-algebra generated by $\G_t^0$ and the $Q_m$ null sets in $\G^m$.  
Similarly $\G_{\geq t}^m$ (resp. $\G_{>t}^m$) is the $\sigma$-algebra generated by $\G_{\geq t}^0$ (resp. $\G_{\geq t}^0$) and the $Q_m$ null sets.
The basic reference for general theory is [DM].  In [DM] the results of interest deal with processes with parameter set $[0,\infty[$.  
But since these properties depend only on the order structure of the parameter set they are valid also when the parameter set is $[ -\infty,+\infty [$ or $\R=]-\infty, + \infty[$ and will be used in these cases without special mention. 
The following notation will be used systematically henceforth.  The Borel $\sigma$-algebra of a topological space $E$ is denoted by $\B(E)$.  
In particular $\B:= \B (\R)$.  
Let $\M^m$ denote the $\sigma$-algebra on $\R \times W$ generated by processes which differ from a $\B \times \G^0$ measurable process by a $Q_m$-evanescent process.  
Recall that a process $Z= (Z_t (w))$ is $Q_m$-evanescent provided there exists $N \in \G^m$ with $Q_m(N)=0$ such that $Z_t (w)=0$ for all $t \in \R$ and $w \not \in N$.  
Let $\mathcal{I}^m \subset \M^m$ be the ideal of $Q_m$-evanescent processes.  In situations where $m \in Exc$ is fixed we often shall omit the $m$ and just write $\M$ and $\mathcal{I}$.

\medskip
The {\it optional} $\sigma$-algebra, $\O^0$, and the {\it copredictable} $\sigma$-algebra $\wP^0$ on $\R \times W$ are defined as follows:
\begin{equation}\tag{3.1}
\O^0 := \sigma \{ Z \in \B \times \G^0: \ Z_t \in \G_{t+}^0, \ 
t \to Z_t \text{ is right continuous and }Z=0 \text{ on } 
\rrbracket -\infty,\alpha\llbracket \},
\end{equation}
\begin{equation}\tag{3.2}
\wP^0 := \sigma \{ Z \in \B \times \G^0: \ Z_t \in \G_{> t}^0, \ 
t \to Z_t \text{ is right continuous and }Z=0 \text{ on } 
\llbracket \beta,\infty\llbracket \}.
\end{equation}
Note that $\wP^0$ is unchanged if $\G_{> t}^0$ is replaced by $\G_{\geq t}^0$ in (3.2).  
Similarly $\O^m$ (resp. $\wP^m$) is defined by replacing $\G^0$ by $\G^m, \ \G_{t+}^0$ (resp. $\G_{> t}^0$) by $\G_t^m$ (resp. $\G_{> t}^m$) and right continuous by $Q_m$ almost surely right continuous in (3.1) (resp. (3.2).  
Recall that $(\G_t^m)$ is right continuous.  It will be convenient to abbreviate right continuous by rc in the sequel. 
A random variable $T:W \to [-\infty, \infty]$ is a {\it costopping time} (resp. $m$-{\it costopping time}) provided $\{ T \geq t \} \in \G_{\geq t}^0$ (resp. $\G_{\geq t}^m$) for each $t \in \R$.  
A random variable $T: W \to [-\infty,\infty]$ is an optional (resp. copredictable) time provided $\llbracket T, \infty \llbracket \in \O^0$ (resp. $\rrbracket -\infty, T \rrbracket \in \wP^0$).
Replacing $\O^0$  (resp. $\wP^0$) by $\O^m$ (resp. $\wP^m$) defines a $Q_m$-optional (resp. $Q_m$-copredictable) time.  
Note $T$ is optional if and only if it is a stopping time for the filtration $(\G_{t+}^0)$ and $T \geq \alpha$.
Clearly a copredictable time is a costopping time but not conversely.  Associated with a costopping time $T$ are two $\sigma$-algebras analogous to $\G_S^0$ and $\G_{S-}^0$ for stopping times $S$. 
The $\sigma$-algebra $\G_{\geq T}^0$ of events after $T$ and $\G_{> T}^0$ of events strictly after $T$ are defined as follows: $\G_{\geq T}^0$ consists of those $\Lambda \in \G^0$ such that $\Lambda \cap \{ T \geq t \} \in \G_{\geq t}^0$ for all $t \in \R$ and $\G_{> T}^0$ is generated by sets of the form $\Lambda \cap \{ T<t\}$ with $\Lambda \in \G_{\geq t}^0$.  
In addition $\G_{> T-}^0 := \{ \Lambda \in \G^0: \ \Lambda \cap \{ T>t \} \in \G_{\geq t}^0$ for $t \in \R\}$ is the analog of $\G_{S+}^0$, $S$ a stopping time.  
If $T$ is a $Q_m$-costopping time $\G_{>T-}^m, \ \G_{\geq T}^m$ and $\G_{>T}^m$ are defined in the obvious manner.
Of course, stopping times or predictable times are relative to the filtration $(\G_t^0)$ while $Q_m$-stopping or $Q_m$-predictable times are relative to $(\G_t^m)$.  It is not hard to check that
$\alpha$ is a co-stopping time and optional, while $\beta$ is a stopping time and 
$\{ \beta \ge t \} \in \G_{>t-}^0 := \underset{s < t} \cap \G_{\ge s}^0$.
On the other hand, in general, neither $\alpha$ nor $\beta$ is copredictable (or predictable).
From the discussion following (2.10), $Y^\ast \in \O^0$ and $\Lambda^\ast \in \O^0$.

\medskip
Of course co-predictable means predictable ``with time reversed".  
In order to make this precise define $\ww(t):= w(-t)$ for $w \in W$ and $\wW := \{ \ww: w \in W\}$.  
Define, for $t \in \mathbf{R}$, the coordinate maps $\widehat {Y}_t (\ww) := \ww(t)$ and the $\sigma$-algebras $\wG^0:= \sigma \{ \widehat{Y}_t; \ t \in \mathbf{R}\}, \ \wG_t^0 := \sigma \{ \widehat{Y}_s; \ s \leq t\}$.  
Define the {\it reversal operator} $r: W \to \wW$ by $rw(t):= \ww(t)$.  Then $r^{-1} :\wW \to W$ is given by $r^{-1} \ww(t) =\ww (-t)$. 
Observe that $r$ is an isomorphism between $(W, \G^0)$ and $(\wW, \widehat{\G}^0)$ with the property that $r(\G_{\geq t}^0) = \wG_{-t}^0$ and $r^{-1} \wG_t^0 = \G_{\geq -t}^0$.
Obviously $(\wG_t^0)$ is an {\it increasing} filtration on $(\wW, \wG^0)$.  
The following lemma will enable us to translate standard results about predictable processes to the corresponding result about copredictable processes.

\medskip\noindent
{\bf (3.3) Lemma}
\begin{inparaenum}[(a)]
\item 
{\it 
If $Z$ is copredictable, then $\wZ_t:= Z_{-t} \circ r^{-1}$ is $(\wG_t^0)$ predictable on $\wW$.  
Conversely if $\wZ$ is any $(\wG_t^0)$ predictable process, then $Z_t := \wZ_{-t} \circ r \cdot 1_{ ]-\infty, \beta[ }(t)$ is copredictable.  
Here the reversal operator $r$ and the filtered space $(\wW, \wG_t^0)$ are defined above.
}

\item
{\it  
If $T$ is a costopping, resp. corepredictable time, then $\wT := -T \circ r^{-1}$ is a $(\wG_t^0)$ stopping, resp. predictable, time.  
Moreover $\wG_{\wT}^0 =r \G_{\geq T}^0$ and $\wG_{\wT-}^0 = r (\wG_{>T}^0)$.  
Conversely if $\wT$ is a $(\wG_t)$ stopping, resp. predictable, time, then $T:=-\wT \circ r$ is a costopping, resp. copredictable, time, and $\G_{\geq T}^0 = r^{-1} \wG_{\wT}^0$ and $\G_{>T}^0 = r^{-1} \wG_{\wT-}^0$.
}
\end{inparaenum}

\begin{proof}
(a) If $Z \in \wP^0$ is right continuous, then $\wZ$ is left continuous and adapted to $\wG_t^0= r \wG_{\geq t}^0$; hence $\wZ$ is $(\wG_t^0)$ predictable.
Such $Z$ generate $\wP^0$ establishing the first assertion in (a).  The second follows similarly since left continuous processes adapted to $(\wG_t^0)$ generate the class of  $(\wG_t^0)$ predictable processes and $1_{\rrbracket -\infty,\beta \llbracket}$ is copredictable.

(b) $T$ is copredictable if and only if $Z:= 1_{]-\infty, T]} \in \wP^0$.  But then $\wZ =1_{[\wT, \infty[}$ and so (a) implies that $\wT$ is $(\wG_t^0)$ predictable.  
Let $\Lambda \cap \{ T<t \}$ with $\Lambda \in \G_{\geq t}^0$ be a generator of $\G_{> T}^0$.  Then
\[
r (\Lambda \cap \{T<t \}) = r \Lambda \cap \{ T \circ r^{-1} <t \}
= r \Lambda \cap \{\wT >-t\}
\]
and since $r \Lambda \in \wG_{-t}^0$, this is a generator of $\wG_{\wT-}^0$.  Thus $r \G_{>T}^0 \subset \wG_{T-}^0$ and reversing the argument we obtain the desired equality.  The costopping case is simpler and left to the reader.  The converse is proved similarly.
\end{proof}

\medskip\noindent
{\bf Remark}
Since 
$\G_{\geq t}^0
= \sigma \{Y_s : s \geq t\} 
= \sigma \{ Y_{-s}; \ -s \leq -t \} = \wG_{-t}^0, \
\underset{s<t} \cap \G_{\geq s}^0 
= \wG_{(-t)-}^0$.
Therefore the right continuity of $(\G_t^m)$ implies that $\wG_{\geq t}^m$ is left continuous; that is $\wG_{\geq t}^m = \underset{s < t} \cap \wG_{\geq s}^m$.

\medskip
Using (3.3) to translate known results in the predictable case leads to the next proposition.

\medskip\noindent
{\bf (3.4) Proposition}
\begin{inparaenum}[(i)]
\item
{\it
If $Z \in \wP^m$, then $Z$ is $Q_m$-indistinguishable from a process in $\wP^0$.
}

\item
{\it 
If $T$ is a $Q_m$-copredictable time, then there exists a copredictable time $S$ such that $Q_m(S\neq T)=0$.  Further for each $\Lambda \in \G_{>T}^m$ there exists $\Lambda' \in \G_{>S}^0$ with $Q_m (\Lambda \Delta \Lambda')=0$.
}
\end{inparaenum}

\begin{proof}
The predictable version of (i) is lemma 7 on page 413 of Appendix 1 in [DM 1980].  Moreover (ii) is just [DM IV; 78] in the predictable case.  Note that since $(\G_t^m)$ is right continuous it is the ``usual augmentation" of $(\G_t^0)$ as defined in [DM IV; 48].
\end{proof}

\medskip\noindent
Because of (3.4) many assertions dealing with $\wP^m$ or $Q_m$-copredictable times will follow from the corresponding assertions about $\wP^0$ or copredictable times.

\medskip\noindent
{\bf (3.5) Proposition}
\begin{compactenum}[(i)]
\item
{\it 
If $Z$ in $\wP^0$, then $Z_t= Z_t \circ b_t$ for all $t \in \R$.
}

\item
{\it 
If $F \in \G_{> \alpha}^0$, then $t \to F \circ b_t$ is copredictable.
}
\end{compactenum}

\begin{proof}
Point (i) is clear since $b_t^{-1} \G_{> t}^0 = \G_{> t}^0$.  
For point (ii) note that $\G_{> \alpha}^0$ is generated by $F$ of the form
$F= \prod\limits_{j=1}^n f_j \circ Y_{s_j} 1_{\{ \alpha <s\}}$
where $s< s_1< \cdots<s_n$ and $f_j \in p \E, \ 1 \leq j \leq n$.  
But for such $F, \ F \circ b_t =F 1_{]-\infty, s[}(t)$ and so $t \to F \circ b_t$ is copredictable.
\end{proof}

\medskip\noindent
{\bf (3.6) Corollary}
$Y^\ast \in \wP^0$.

\begin{proof}
From the definition (2.10), it is evident that 
$Y_{\alpha +}^\ast \in \G_{>\alpha}^0$ and that 
$Y_t^\ast = Y_{\alpha +}^\ast \circ b_t$.
\end{proof}

\medskip\noindent
{\bf (3.7) Lemma}
{\it 
Let $T$ be a stopping (resp. costopping) time. 
Then $Q_m$ is $\sigma$-finite on $\G_{T-}^0 |_{\{ \alpha <T \}}$ (resp.
$\G_{>T}^0 |_{\{T < \beta\}}$).  Also $\{ \alpha <T \} \in \G_{T-}^0$ (resp. 
$\{ T < \beta \} \in \G_{>T}^0$).
}

\begin{proof}
Let $T$ be a costopping time.  Then both assertions are clear since
\[
\{ T<\beta\} 
= \underset{q \in \mathbb{Q}}\cup \underset{k} \cup
\{ T< q <\beta \} \cap \{ Y_q \in E_k \}
\]
where each $E_k \in \E$ with $m(E_k) <\infty$ and $E= \cup E_k$.  Recall that $\beta$ is a costopping time (actually copredictable).  The stopping time case is similar and well-known.
\end{proof}

\medskip
We are going to need the existence of copredictable projections.  Of course the existence of predictable (and optional) projections are standard facts on a filtered probability space [DM IV; 43].  But we shall need similar results for a $\sigma$-finite measure space in which the relevant $\sigma$-algebras do not form a filtration in the usual sense.  Therefore I have decided to include detailed proofs for the predictable case in enough generality for later needs.  Then using (3.3) it will be easy to translate these results from the predictable case to the copredictable case which is the main interest here.  The most complete reference for martingale theory relative to a $\sigma$-finite filtered space is the little book [H66].  Some results are also contained in [DM V; 39-43].

\medskip
The next lemma contains the key technical result that we shall need.  We begin by establishing the notation that will be used in its statement and proof.  Let $(\Omega, \F, Q)$ be a $\sigma$-finite measure space and let $(\F_t)_{t\in \R}$ be a filtration on $(\Omega, \F)$.  We suppose that $\F$ is $Q$-complete and that $(\F_t)$ is right continuous with 
$\N := \{ \Lambda \in \F; \ Q(\Lambda)=0 \} \subset \F_t$ for each $t \in \R$.  Thus $(\Omega, \F, \F_t, Q)$ satisfies the ``usual conditions" [DM IV; 48] except that $Q$ is not a probability and $\R$ is the index set.  We assume further that there exists a stopping time, $\gamma$, such that for each $t, \ Q$ is $\sigma$-finite on $\F_t |_{\{\gamma <t \}}$.  
Stopping times, predictable times etc. are relative to the filtration $(\F_t)$.

\medskip\noindent
{\bf (3.8) Lemma}
{\it 
Let $H \in b(\F)$ with $Q(\lvert H \rvert)<\infty$.  
Then under the above conditions there exists an adapted process, $Z$, which vanishes on $\rrbracket -\infty, \gamma \rrbracket$ and is right continuous with left limits $(rcll)$ on $\rrbracket \gamma,\infty \llbracket$ such that if $T$ is a stopping time, then $Q$ is $\sigma$-finite on $\F_{T-} |_{\{ \gamma< T<\infty \}}$ and $Z_T= Q(H \ \lvert \  \F_T  \rvert_{\{\gamma<T<\infty\}})$ on 
$\{\gamma<T<\infty\}$.  
Moreover if $S$ and $T$ are stopping times with $S \leq T, \ Z_S = Q(Z_T \ \lvert \ \F_S \rvert_{\{ \gamma <S<\infty\}})$ on $\{ \gamma <S<\infty\}$.  
If $T$ is predictable, $Z_{T-} = Q(H \ \lvert \ \F_{T-} \rvert _{\{\gamma<T<\infty \}})$ on $\{ \gamma < T< \infty\}$.
}

\begin{proof}
Since $\{ \gamma <r <T<\infty \} \in \F_{T-}$ for each $r, \ Q$ is $\sigma$-finite of $\F_r \mid_{\{\gamma<r\}}$ and $\{ \gamma<T<\infty \}$ is the union over all rational $r$ of $\{ \gamma <r<T<\infty\}$, it follows that $Q$ is $\sigma$-finite on $\F_{T-} \mid_{\{\gamma<T<\infty\}}$.  
In particular $Q$ is $\sigma$-finite on $\F_T \mid_{\{ \gamma<T<\infty\}}$.  Suppose $\G$ is a sub-$\sigma$-algebra of $\F, \ \Lambda \in \F$ and $Q$ is $\sigma$-finite on $\G\mid_\Lambda$.  
If $H$ is $Q$ integrable, then $Q(H \ \lvert \ \G \rvert_\Lambda)$ 
--- the ``conditional expectation'' of $H$ given $\G \rvert_\Lambda$ ---
exists on $\Lambda$.  
It will be convenient to write $G=Q (H \mid \G)$ on $\Lambda$ to mean $G= Q(H \ \lvert \ \G \rvert_\Lambda)$.  For each $t$, define
$H_t= Q (H \mid \F_t)$ on $\{ \gamma <t \}$.  
Now fix $a \in \R$ and set $\Omega_a = \{\gamma<a\}$.  
Then if $t \geq a, \ H_t$ is defined on $\Omega_a \subset \{\gamma<t\}$ and $H_t \mid_{\Omega_a}$ is $\F_t \mid_{\Omega_a}$ measurable.  Define $H_t^a = H_t \mid_{\Omega_a}, \ \F_t^a =\F_t \mid_{\Omega_a}, \ \F^a =\F \mid_{\Omega_a}, \ t \geq a$ so that $H_t^a =Q( H \mid \F_t^a)$ on $\Omega_a$ for $t\geq a$.  
Then $(\F_t^a)_{t \geq a}$ is a filtration on $(\Omega_a, \F^a)$ and $(H_t^a)_{t \geq a}$ is a martingale on $(\Omega_a, \F^a, \F_t^a, Q)$.  
We may suppose that $\lvert H_t^a \rvert \leq \sup \lvert H\rvert$.  
We don't bother to write $Q^a$ for $Q$.  Of course $Q$ is $\sigma$-finite on $\F_t^a$ since $\F_a \mid_{\Omega_a} \subset \F_t \mid_{\Omega_a}$ for $t \geq a$.  
Now using Theorem 5.4 of [H66] which guarantees limits along countable sequences in the $\sigma$-finite case, the standard arguments yield the existence of an $rcll$ on $[a,\infty[$ version $Z^a$ of $H^a$.  
If $T$ is a stopping (resp. predictable) time, then $T^a := T \mid_{\Omega_a}$ is a stopping (resp. predictable) time relative to $(\F_t^a)$.  
Theorem 5.1 in [H66] is the optional sampling theorem for supermartingales indexed by a countable subset of $\R$.  So if $S \leq T$ are stopping times, the standard technique of approximating $S^a$ and $T^a$ by their dyadic approximations from above and using (vii) and (viii) of \S 7 of [H66], we obtain on $\{T_a <\infty\}$ (see the proof of Theorem 6.1 in [H66] for the proof of the required uniform integrability of supermartingales relative to a decreasing family of $\sigma$-algebras):
\begin{equation}\tag{3.9}
\begin{cases}
(i) & Z_{T^a}^a = Q(H \mid \F_{T^a}^a), \\
(ii) & Z_{S^a}^a = Q (Z_{T^a}^a \mid \F_{S^a}^a ),\\
(iii) & Z_{T^a-}^a = Q(H \mid \F_{T^a-}^a) \text{ if } T \text{ predictable.}
\end{cases}
\end{equation}

\medskip
Next suppose $a,b \in \R$ with $a<b$.  If $t \geq b, \ \F_t^a =\F_t^b \mid_{\Omega_a} \subset \F_t^b$ as a subset since $\Omega_a \in \F_t^b$.  Consequently if $\Lambda \subset \F_t^a$,
\begin{equation}\tag{3.10}
\int_\Lambda H_t^a dQ =\int_\Lambda H dQ
=\int_\Lambda H_t^b dQ =\int_\Lambda H_t^b \mid_{\Omega_a} dQ.
\end{equation}
But $H_t^b \mid_{\Omega_a} \in \F_t^a$ and since (3.10) holds for each $\Lambda \in \F_t^a$, it follows that for each fixed $t \geq b, \ H_t^a =H_t^b \mid_{\Omega_a}$ a.s. $Q$.  
Now $Z^a$ is $rcll$ version of $H^a$ and so $Z_t^a = Z_t^b$ on $\Omega_a$ for all $t \geq b$ a.s. $Q$.  
Therefore a.s. $Q$ on $\{ \gamma <\infty\}, \ Z_t := \ll_{q \in \mathbb{Q}, q \to \infty} Z_t^q$ exists
for all $t > \gamma$ and is $rcll$, and for $a \in\R, \ Z_t= Z_t^a$ on $\Omega_a$ for $t \geq a$.  
Define $Z_t=0$ if $t \leq \gamma$ or if the conditions in the preceding sentence fail. 
Then $Z=(Z_t)$ is adapted, $rcll$ on $\rrbracket \gamma, \infty \llbracket$, vanishes on $\rrbracket - \infty, \gamma \rrbracket$ and $Z_t=Z_t^a$ on $\{ \gamma < a\}, \ t \geq a$.  
For each $t$ and $\Lambda \in \F_t^t =\F_t \mid_{\{ \gamma <t \}}$, (3.10) implies that $\int_\Lambda Z_t dQ =\int_\Lambda H dQ$.  
Hence $Z_t =Q(H \mid \F_t)$ on $\{ \gamma <t\}$.  Noting that $\F_{T^a}^a =\F_T \mid_{\Omega_a}$ and letting $a \to \infty$ through a sequence in (i) and (ii) of (3.9) establishes all except the last assertion in (3.8).

\medskip
Finally suppose $T$ is $(\F_t)$ predictable.  Let $(T_n)$ announce $T$.  Then from what has been established $Z_{T_n} = Q( H \mid \F_{T_n})$ on $\{ \gamma <T_n<\infty\}$.  
Now let $n \to \infty$ and use [H66; \S 7 (vii)] to obtain $Z_{T-} =Q(H \mid \F_{T-})$ on $\{ \gamma <T<\infty\}$ completing the proof of (3.8).
\end{proof}

\medskip\noindent
{\bf Remark}
As noted during the proof of (3.8), $Z_t =Q(H \ \lvert \ \F_t \rvert_{\{\gamma< t\}})$ on $\{ \gamma <t\}$ for fixed $t$.  
However $Z$ is {\it not} a martingale, at least in the usual sense, since $\{ \F_t \mid_{\{\gamma <t\}} \}$ is not a filtration on $\Omega$.  
Therefore some caution is necessary in applying the results in [H66].

\medskip
We are now prepared to formulate our basic predictable projection theorem.  However before coming to it we shall establish the required uniqueness theorem.  For the next two theorems the hypotheses and notation are those of (3.8).  As in (3.8) predictable means $(\F_t)$ predictable, and evanescent will mean $Q$-evanescent.

\medskip\noindent
{\bf (3.11) Theorem}
{\it
Let $Z^1$ and $Z^2$ be bounded predictable processes such that for each predictable time $T$,
\[
Q(Z_T^1; \ \gamma<T<\infty)
= Q(Z_T^2; \ \gamma<T<\infty).
\]
Then $\{ Z^1 \neq Z^2 \} \cap \rrbracket \gamma,\infty \llbracket$ is evanescent.
}

\medskip
\begin{proof}
Since $\rrbracket \gamma, \infty \llbracket$ is predictable we may suppose that $Z^i$ vanishes on $\rrbracket -\infty, \gamma \rrbracket, \ i=1,2$.  
In which case the conclusion of (3.11) becomes $\{Z^1 \neq Z^2 \}$ evanescent.  Since $Q$ is $\sigma$-finite there exists a probability, $P$, on $\F$ equivalent to $Q$.  
Clearly a set is evanescent if an only if it is $P$-evanescent.  By symmetry it suffices to show that $\{Z^1 <Z^2 \}$ is evanescent. 
If not, for some $\varepsilon>0, \ \{ Z^1 <Z^2 + \varepsilon \}$ is not evanescent.  
Therefore by the predictable section theorem for probability measures [DM IV; 85] there exists a predictable time $T$ with  $\llbracket T \rrbracket \subset \{Z^1 < Z^2 + \varepsilon \}$ and $P(\gamma < T< \infty) >0$ since $\{ Z^1< Z^2 + \varepsilon \} \subset \rrbracket \gamma, \infty \llbracket$.
But $Q \sim P$ and so $Q(\gamma <T<\infty)>0$.
Also $Q$ is $\sigma$-finite on $\F_{T-} \mid_{\{\gamma<T<\infty\}}$.
Consequently there exists $\Lambda \in \F_{T-}$ with $\Lambda \subset \{ \gamma < T <\infty \}$ and $0 <Q(\Lambda) <\infty$.  
Define $S =T$ on $\Lambda$ and $S=+\infty$ off $\Lambda$.  Then $S$ is predictable [DM IV; 73c] and 
\[
Q\{Z_S^1; \ \gamma< S < \infty \}
= Q(Z_T^1; \Lambda) 
< Q( Z_T^2; \Lambda)
=Q\{Z_S^2; \gamma<S<\infty\}
\]
contradicting the hypothesis.
\end{proof}

\medskip
We come now to the basic predictable projection theorem that we need.  Armed with (3.8) and (3.11) its proof is a direct extension of the proof in [DM VI; 43, 44].  
The hypotheses of (3.8) are still in force and the notation is that set forth above (3.8).

\medskip\noindent
{\bf (3.12) Theorem}
{\it 
Let $X \in p (\B \times \F)$.  Then there exists a predictable process, $^pX$, vanishing on $\rrbracket -\infty, \gamma \rrbracket$ such that each predictable time $T$
\[
Q(X_T; \gamma<T<\infty)= Q({^p X_T}; \gamma <T<\infty).
\]
Moreover $^pX$ is unique up to $Q$-evanescence.
}

\begin{proof}
The class of processes $X \in b(\B \times \F)$ for which $^pX$ exists is a vector space closed under increasing limits of uniformly bounded sequences.  Processes of the form $X_t(\omega)= 1_{]a,b]} (t) H(\omega)$ with $a,b \in \R$ and $H \in b\F$ generate $\B \times \F$ and are closed under products.  
Hence by the monotone class theorem 
(2.18)
it suffices to verify the existence for the above class of processes. 
In view of the vector space property it suffices to consider $H \in pb \F$.  But any such $H$ is the increasing limit of $Q$-integrable elements of $pb\F$ and so finally it suffices to suppose $H \in pb\F$ with $Q(H)<\infty$.  Let $Z$ be the process in (3.8) corresponding to $H$ and define 
\[
{^p X_t} (\omega) =1_{]a,b]}(t) Z_{t-} (\omega)
1_{]\gamma(\omega),\infty[} (t).
\]
Then $^pX$ is left continuous and adapted, hence predictable.  From the last assertion in (3.8) for each predictable time $T$ we find
\[
{^pX_T} =1_{]a, b]}(T) Q(H \mid \F_{T-}) \text{ on }
\{ \gamma <T<\infty\}.
\]
Integrating with respect to $Q$ establishes the existence assertion.  For the uniqueness first note that if $X$ is bounded by $k$, then $\{{^pX}>k \}$ is evanescent and the uniqueness of $^pX$ follows from (3.11).  In the general case applying this to $X \wedge n$ and then letting $n \to \infty$ we obtain the uniqueness assertion of the (3.12).
\end{proof}

\medskip\noindent
{\bf Remark}
Theorem 3.12 obviously extends to processes which differ from a $\B \times \F$ measurable process by a $Q$-evanescent process and we shall use it for such processes without special mention.

\medskip
We now turn our attention to the copredictable situation which is our main interest here.  Recall the dual filtration $(\wW, \wG^0, \wG_t^0)$ and reversal operator $r:W \to \wW$ defined earlier in this section.  
Define $\wQ_m=rQ_m$ and then $\wG^m$ and $\wG_t^m$ similarly to $\G^m$ and $\G_t^m$.  
Also set $\widehat{Y}_t (\ww)= \ww(t)$.  Note that $\widehat{\alpha} := \inf \{t: \widehat{Y}_t \in E\}= -\beta \circ r^{-1}$.  It will be convenient to set $\widehat{r} = r^{-1}: \widehat{W} \to W$.  Both $(W, \G^m, \G_t^m, Q_m)$ and $(\wW, \wG^m, \wG_t^m, \wQ_m)$ satisfy the hypotheses of (3.8), (3.11) and (3.12) with $\gamma =\alpha$ and $\gamma =\widehat{\alpha}$ respectively since $\wG_t^0 =r ( \G_{-t}^0)$.

\medskip
We first translate (3.11) to the copredictable case.

\medskip\noindent
{\bf (3.13) Theorem}
{\it
Let $Z^1, Z^2 \in b\wP^m$ and suppose that for each copredictable time $T$,
\begin{equation}\tag{3.14}
Q_m (Z_T^1; - \infty <T<\beta) =Q_m (Z_T^2; - \infty <T<\beta).
\end{equation}
Then $Z^1$ and $Z^2$ are $Q_m$-indistinguishable.
}

\begin{proof}
By (3.4) we may suppose that $Z^i \in b \wP^0, \ i=1,2$.  Then from (3.3), $\widehat{Z}_t^i := Z_{-t}^i \circ \widehat{r}$ is $(\wG_t^0)$ predictable for $i=1,2$ and vanishes on $\rrbracket -\infty, \widehat{\alpha} \rrbracket$.  If $\widehat{T}$ is a $(\wG_t^0)$ predictable time, $T:= -\widehat{T} \circ r$ is copredictable.  Hence for $i=1,2,$
\[
\widehat{Q}_m (\widehat{Z}_{\widehat{T}}^i; 
\widehat{\alpha} <\widehat{T}<\infty) 
= Q_m (Z_T^i; -\infty <T<\beta).
\]
It now follows from (3.11) that $\widehat{Z}^1$ and $\widehat{Z}^2$ are $\widehat{Q}_m$ indistinguishable which clearly is equivalent to the assertion in (3.13).
\end{proof}

\medskip
Finally we arrive at the copredictable projection theorem.

\medskip\noindent
{\bf (3.15) Theorem} 
{\it
Supose $Z \in p\M^m$.  Then there exists a copredictable process, $^{\widehat{p}}Z$, unique up to $Q_m$ evanescence such that for each $Q_m$ copredictable time $T$,
}
\begin{equation}\tag{3.16}
Q_m (Z_T; -\infty <T<\beta)
= Q_m ({^{\widehat{p}}Z_T} ; -\infty< T<\beta).
\end{equation}

\begin{proof}
It suffices to verify (3.16) for $Z \in p(\B \times \G^0)$ and copredictable $T$.  Then $\widehat{Z}_t :=Z_{-t} \circ \widehat{r} \in p (\B \times \G^0)$.
Define $^{\wp} Z_t:=( {^p \wZ_{-t}}) \circ r$.  From (3.3), $^{\wp}Z$ is copredictable.  Let $T$ be copredictable and $\wT:=-T \circ r$.  
Then since $\wT$ is a $(\wG_t^0)$ predictable time,
\begin{align*}
Q_m &(Z_T;-\infty<T<\beta)
=\wQ_m (\wZ_{\wT}; \widehat{\alpha}< \wT<\infty)\\
&=\wQ_m ({^p} \wZ_{\wT}; \ \widehat{\alpha} <\wT <\infty)
=Q_m ( {^{\widehat{p}}} Z_T; \ - \infty < T<\beta).
\end{align*}
The uniqueness of ${^{\wp}} Z$ comes from (3.13) or the uniqueness assertion in (3.12).
\end{proof}

\medskip\noindent
{\bf (3.17) Remark}
It follows from (3.3) and the corresponding result in the predictable case [DM IV; 73c],
that if $T$ is copredictable and 
$\Lambda \in \G_{>T}^0$, then 
$\wT := T$ on $\Lambda$ and
$\wT := -\infty$ off $\Lambda$ is copredictable.  Consequently a
standard argument shows that
(3.16) holds for all $Q_m$-copredictable $T$ if and only if
\[
{^{\wp}} Z_T= Q_m (Z_T \mid \G_{T-}^m) \text{ on }
\{-\infty <T<\beta\}
\]
for all such $T$.

\medskip\noindent
{\bf (3.18) Remark}
The technique use in the proof of (3.15) enables us to translate standard facts about predictable projections to the corresponding statements about copredictable projections.  In particular, if $Z \in p\M^m$ is $Q_m$ a.s. right continuous, then $^{\wp} Z$ is $Q_m$ a.s. right continuous [DM VI; 47].

\medskip
We shall also need an optional projection.  An optional projection may be constructed by an argument analogous to that used in the proof of Theorem 3.12.  But such an argument leads to a projection that is carried by $\rrbracket \alpha, \infty \llbracket$ which does suffice for our later needs.  
We shall need a projection that is defined at $\alpha$ at least on $W(b)$.  Following Fitzsimmons [F87] we shall construct an optional projection that is defined on $\Lambda^\ast$ which will be adequate for us.  It is clear from the definition (2.11) that 
$\rrbracket \alpha, \beta \llbracket \subset \Lambda^\ast \subset \llbracket \alpha, \beta \llbracket$.  Our method is a variant of Dawson's formula [$S$; 22:7].  If $\omega \in \Omega$ define $\omega^+ (t) =\omega(t+)$ for $t \geq 0$.  Then $\omega^+(t)=\omega(t)$ if $t>0, \ \omega^+(0)=\omega(0+)$ which is in $E$ by definition, and $X_t(\omega) =\omega^+(t)$ for $t \geq 0$.  Define a ``splicing" map from $W \times \R \times \Omega \to W$ by
\begin{equation}\tag{3.19}
\begin{split}
(w \lvert t \rvert \omega)(s) 
&:= w(s) \text{ if }s<t\\
&:= \omega^+(s-t) \text{ if } s \geq t.
\end{split}
\end{equation}
If $(t,w) \in \Lambda^\ast$, then $\theta_t w \in \Omega$ and $w= (w \lvert t \rvert \theta_t w)$.  For $Z \in p (\B \times \G^0)$ define
\begin{equation}\tag{3.20}
^0 Z_t (w):= 1_{\Lambda^\ast} (t,w) 
\int_{\Omega} Z_t (w \lvert t\rvert \omega)
P^{Y^\ast (t,w)} (d\omega).
\end{equation}

\medskip\noindent
{\bf (3.21) Theorem}
{\it Let $Z \in p \M^m$.  Then there exists a unique up to $Q_m$ evanescence process $^0 Z \in p(\O^m)$ such that for each $Q_m$ stopping time $T$,
\[
Q_m (Z_T; \ Y_T^\ast \in E) =Q_m ( ^0 Z_T; \ Y_T^\ast \in E).
\]
}

\begin{proof}
It suffices to prove this for $Z \in bp (\B \times \G^0)$ and stopping times $T$.  For such $Z$ and $T$ define $^0Z$ by (3.20) and
\[
H(w,\omega):= 1_{\Lambda^\ast} (T(w),w) Z_{T(w)}
(w \lvert T(w) \rvert \omega).
\]
We claim that $H \in bp (\G^0_{T+} \times \F^0)$.   It suffices to verify this for $Z$ of the form $Z_t (w) =g(t) G(w)$ with $g \in bp\B$ and $G=\prod\limits_{j=1}^n f_j \circ Y_{t_j}$ with $t_1 < \dots< t_n$ and $f_j \in bp \E$.  For $Z$ of this, form

\begin{equation}\tag{3.22}
H(w,\omega) = 1_{\Lambda^\ast} (T(w),w) g \circ T(w)
\prod\limits_{j=1}^{k(w)} f_j \circ Y_{t_j}(w)
\prod\limits_{j=k(w)+1}^n f_j \circ X_{t_j -T(w)} (\omega),
\end{equation}
where, setting $t_0=-\infty$ and $t_{n+1}=\infty$, for $w$ such that $(T(w),w) \in \Lambda^\ast$ or equivalently $Y_T^\ast (w) \in E$ and $k(w)$ is the unique value of $k$ with $t_k <T(w) \leq t_{k+1}, 0 \leq k \leq n$.  Of course the products in (3.22) corresponding to $k(w)=0$ or $=n$ do not appear.  
Because $\Lambda^\ast \in \O^0$, 
$H \in \G_{T+}^0 \times \F^0$ for $Z$ of the above form and, hence for all $Z \in bp(\B \times \G^0)$ by a monotone class argument.  Now suppose that $Z \in bp(\B \times \G^0)$ is such that
\[
1_E (Y^\ast_T (w)) Z_{T(w)} (w \lvert T(w) \rvert \omega)
= H(w,\omega)
=1_E (Y^\ast_T(w)) K(w) F(\omega)
\]
where $K \in pb \G^0_{T+}$ and $F \in pb \F^0$.  But $w=(w \lvert t \rvert \theta_t w)$ if $Y_t^\ast (w) \in E$ and so
\[
1_E (Y^\ast_T (w)) Z_T (w) = 1_E (Y^\ast_T (w)) K(w) F(\theta_T w).
\]
From (2.13), $Q_m$ is $\sigma$-finite on $\G_T^m \mid_{\{Y^\ast(T) \in E\}}$ and if $\Lambda \in \G_T^m \mid_{\{Y^\ast(T) \in E\}}$ with $Q_m(\Lambda) <\infty$, then for $Z$ of the above form one checks using (2.13) that $Q_m (Z_T ;  Y_T^\ast \in E, \Lambda)= Q_m (^0Z_T; Y_T^\ast \in E, \Lambda)$ where $^0Z$ is defined in (3.20).  
Since both sides of this last equality are finite measures in $Z$, the equality holds for all $Z \in b(\B \times \G^0)$ by another monotone class argument.  Letting $\Lambda \uparrow \{ Y^\ast_T \in E\}$ through a sequence of such sets we obtain (3.21) for $Z \in bp(\B \times \G^0)$.

\medskip
It remains to check $^0Z \in \O^m$ and to verify the claimed uniqueness.  Using (3.22) wih $T(w) \equiv t \in \R$,
\[
H(\cdot, \omega)= 1_E (Y_t^\ast)g(t) 
\sumlim_{j=0}^k f_j \circ Y_{t_j}
\sumlim_{j=k+1}^{n+1} f_j \circ X_{t_j-t} (\omega) 
1_{] t_k, t_{k+1}]} (t)
\]
for $Z$ of the form above (3.22).  Define
\[
h_k(t, x) =P^x \prod\limits_{j=k+1}^{n+1}
f_j \circ X_{t_j-t} 1_{] t_k; t_{k+1}]} (t).
\]
Clearly $h_k \in \B \times \E$ and so using (2.13) again
\[
^0Z_t =1_E (Y^\ast_t) g(t) 
\sumlim_{j=0}^k f_j \circ Y_{t_j} h_k (t, Y_t^\ast) 1_{]t_k, t_{k+1}]}(t)
\]
for $Z$ as above.  Consequently $^0Z \in \O^0$ for $Z$ as above, and, hence, by yet another monotone class argument for $Z \in bp (\B \times \G^0)$.  Using the fact that $Q_m$ is $\sigma$-finite on $\G_T^m \mid_{\{Y_T^\ast \in E\}}$,  the uniqueness follows from the optional section theorem as in [DM IV; 87b].
\end{proof}

\medskip
We conclude this section with several properties of processes which are both optional and copredictable.  We begin with a lemma.

\medskip\noindent
{\bf (3.23) Lemma}
{\it 
Let $(W, \G,Q)$ be a $\sigma$-finite measure space and $\H$ a sub-$\sigma$-algebra of $\G$.  Let $J \in pb \G$ and suppose that for $\Lambda \in \G$ with $Q(\Lambda) < \infty$,
\[
Q(J^2; \Lambda)
= Q(JQ(J1_\Lambda \mid \H); \Lambda).
\]
Then $J=Q(J1_\Lambda \mid \H )$ on $\Lambda$.
}

\medskip
\begin{proof}
Using the hypothesis for second equality below, the subtraction being valid since $Q(\Lambda)<\infty$,
\begin{align*}
0 
& \leq Q((J- Q(J1_\Lambda \mid \H))^2 ; \Lambda)\\
&= Q(J^2 -2JQ(J1_\Lambda \mid \H) 
+[Q(J1_\Lambda \mid \H)]^2; \Lambda)\\
&= Q(-J1_\Lambda Q(J1_\Lambda \mid \H)) + Q([Q(J1_\Lambda \mid \H)]^2; \Lambda)\\
&=-Q([Q(J1_\Lambda \mid \H)]^2 (1-1_\Lambda)) \leq 0,
\end{align*}
from which the desired conclusion follows.
\end{proof}

\medskip\noindent
{\bf (3.24) Proposition}
{\it 
Let $Z \in p \O^m$ be $Q_m$ a.s. right continuous.  Then $^{\wp}Z \in \O^m$.
}

\medskip
\begin{proof}
It suffices to consider $Z \in bp \O^0$.  Then from (3.18), $^{\wp} Z$ is $Q_m$ a.s. right continuous and vanishes on $\rrbracket - \infty, \alpha \llbracket$ since by definition $Z$ vanishes on $\rrbracket -\infty, \alpha \llbracket \in \wP^0$.  
Of course $^{\wp}Z=0$ on $\llbracket \beta, \infty \llbracket$ by definition.  
Therefore it suffices to show that $^{\wp}Z$ is adapted to $(\G_t^m)$.  
Fix $t$ and let $H \in pb \G^0_{>t}$.  
For convenience set $\G_t^t := \G_t^0 \mid_{\{ \alpha <t< \beta\}}$.  
Then from the Markov property (2.6), $Q_m (H\mid \G_t^t)$ is $\sigma(Y_t) \subset \G_{\geq t}^0$ measurable on $\{ \alpha<t<\beta\}$.  
Let $\Lambda \in \G_{>t}^0 \mid_{\{\alpha <t< \beta\}}$ with $Q_m(\Lambda) <\infty$.  
Then $1_\Lambda {^{\wp} Z_t} \in \G_{>t}^0$ and so $Q_m (1_\Lambda {^{\wp} Z_t} \mid \G_t^t) \in \G_{\geq t}^0$.  Since $Q_m (\alpha =t) =0, \ Z_t =0$ on $]-\infty, \alpha]$ and $^{\wp} Z_t =0$ off $] \alpha, \beta[$ a.s. $Q_m$,
$Q_m (1_\Lambda {^{\wp} Z_t} \mid \G^m) \in \G_{> t}^m$. 
Therefore we compute
\begin{align*}
Q_m ( {^{\wp} Z_t}{^{\wp} Z_t}; \Lambda)
&= Q_m (Z_t {^{\wp} Z_t}; \Lambda)\\
&= Q_m (Z_t Q_m (1_\Lambda {^{\wp} Z_t} \mid \G_t^m); t <\beta)\\
&= Q_m ( {^{\wp} Z_t} Q_m (1_\Lambda {^{\wp} Z_t} \mid \G_t^m))
\end{align*}
where the first equality follows from the definition of copredictable projection, the second because $Z_t \in \G_t^m$ and third because $Q_m (1_\Lambda {^{\wp} Z_t} \mid \G_t^m) \in \G_{> t}^m$.  
Now (3.23) implies that ${^{\wp} Z_t} =Q_m ({^{\wp} Z_t} \ \lvert \ \G_t^0 \rvert_{\{ \alpha <t< \beta\}})$ on $\Lambda$.  Let $\Lambda \uparrow \{ \alpha < t<\beta \}$ through a sequence of such sets to conclude that ${^{\wp} Z_t} \in \G_t^m$ since $^{\wp} Z_t$ vanishes off $\{ \alpha <t<\beta\} \in \G_t^0$ almost surely $Q_m$.
\end{proof}

\medskip
For the final result of this section we need a definition.  As in [F87] a process $Z \in \M^m$ is {\it homogeneous} provided that for each $s \in \R$ the processes $t \to Z_{s+t}$ and $t \to Z_t \circ \sigma_s$, or equivalently $t \to Z_t$ and $t \to Z_{t-s} \circ \sigma_s$, are $Q_m$ indistinguishable.  That is, for each $s \in \R$ there exists $N_s \in \G^m$ with $Q_m (N_s) =0$ such that for each $t \in \R$ and $w \not \in N_s, \ Z_{t-s} ( \sigma_s w) = Z_t (w)$.

\medskip\noindent
{\bf (3.25) Theorem}
{\it
Let $Z \in p(\O^m \cap \wP^m)$. 
(a) Then $Z$ is $Q_m$ indistinguishable from a process of the form $g(t, Y_t^\ast)$ where $g \in p( \B \times \E)$ with $g(t,\Delta)=0$ by convention.
(b) If, in addition, $Z$ is homogeneous then there exists an $f \in p\E$ such that $1_{\Lambda^\ast} Z$ and $f \circ Y^\ast$ are $Q_m$ indistinguishable.
}

\medskip
\begin{proof}
(a)  We may suppose that $Z \in p(\O^m \cap \wP^0)$.  By definition $Z$ vanishes off $\llbracket \alpha, \beta \llbracket$ and by (3.5), $Z_t (w) =Z_t (b_t w)$ identically in $(t,w)$.
Since $b_t =\sigma_{-t} \theta_t, \ Z_t = Z_t \circ b_t = (Z_t \circ \sigma_{-t}) \circ \theta_t$.  
But $Z_t \in \G_{>t}^0$ and so $H_t := Z_t \circ \sigma_{-t} \in \F^0 = \G_{>0}^0 \mid_{\Omega}$ and $Z_t =H_t \circ \theta_t$.
Define $g(t,x)=P^x(H_t)$ if $(t,x) \in \R \times E$ and $g(t,\Delta)=0$.  Then $g \in p(\B \times \E)$ and $(w,\omega) \to H_{T(w)} (\omega)$ is $\G_T^m \times \F^0$ measurable if $T$ is a $(\G_t^m)$ stopping time.  Therefore the standard extension of (2.13) using a monotone class argument and the fact that $Q_m$ is $\sigma$-finite on $\G_T^m \mid_{\{Y^\ast (T) \in E\}}$ we find
\begin{align*}
Q_m(Z_T; Y_T^\ast \in E) 
&= Q_m (H_T (\theta_T), Y_T^\ast \in E)\\
&= \int_{\{Y^\ast (T) \in E\}} \int_{\Omega} H_{T(w)} (\omega)
P^{Y_T^\ast(w)} (d\omega) Q_m (dw)\\
&= Q_m (g (T, Y_T^\ast); Y^\ast(T) \in E).
\end{align*}
Since $Z_t \in \O^m$ while $g(t, Y_t^\ast)$ and $1_E(Y_t^\ast)$ are in $\O^0$, point 
(a) follows from the optional section theorem [DM IV; 87b] and the fact that $Q_m$ is $\sigma$-finite on $\{Y_T^\ast \in E\}$.

\medskip\noindent
(b) By (a) we may suppose that $Z_t= g(t, Y_t^\ast)$ with $g \in p(\B\times\E)$ with $g(t, \Delta)=0$ since homogeneity is preserved when replacing $Z$ by a $Q_m$ indistinguishable process.  Define for each $s \in \R$,
\[
\Lambda_s := \{w: \text{ for all } t \in \R, \
g(t-s, Y_t^\ast (w))
= g(t, Y_t^\ast (w))\}
\]
Then $Q_m (\Lambda_s^c)=0$.  Also $\Lambda_s \in \G^0$ (this follows by a monotone class argument since it is clear for bounded continuous $g$).  By Fubini's theorem for $Q_m$ a.e. $w \in W$ there exists a Lebesgue null set $N_w$ such that for $s \not \in N_w, \ g (t-s, Y_t^\ast (w))= g(t, Y_t^\ast (w))$ for all $t$.  Let $f(x):= \int_0^1 g(u,x) du$.  Then $f \in p\E$ and $f(\Delta)=0$, and a.e. $Q_m$ for all $t$
\[
Z_t =g(t, Y_t^\ast) 
= \int_{t-1}^t g(t-s, Y_t^\ast) ds
= f(Y_t^\ast),
\]
establishing point (b).
\end{proof}

\newpage
\noindent
{\bf 4. The Moderate Markov Dual Process}

\medskip
As in previous sections $X$ is a fixed Borel right process and we fix $m \in Exc$.  
In this section we shall construct a process $\wX$ that is dual to $X$ with respect to $m$.  
In contrast to $X, \wX$ will be left continuous and have the moderate Markov property to be defined shortly.  
The laws of $\wX, (\widehat{P}^x; x \in E_{\Delta})$ will be uniquely determined up to an $m$-exceptional set of $x$.  
Our construction follows that Fitzsimmons [F87] which in turn is based on earlier work of Azema [A73] and Jeulin [J78].

\medskip
We begin with some notation.  Define 
\begin{equation}\tag{4.1}
\begin{split}
&(i) \quad \wO 
:= \{ \beta =0\} \cup \{ [\Delta] \} \subset W;\\
&(ii) \quad \wX_t (\wo) 
:= Y_{-t}^\ast (\wo), \ t>0, \wo \in \wO;\\
&(iii)\quad \widehat{\mathcal{F}}^0
:= \sigma \{ \wX_s ; s>0 \}, \
\widehat{\mathcal{F}}^0_t  := \sigma \{ \wX_s; 0<s\leq t\}, \ t>0;\\
&(iv) \quad \check{\theta}_t w(s) 
:= \begin{cases}
w(s+t), & s <0\\
\Delta & s \geq 0
\end{cases}
, \ t \in \R;\\
&(v) \quad \widehat{\theta}_t =\check{\theta}_{-t}, \ t \in \R.
\end{split}
\end{equation}
Note that $\check{\theta}_t \{ 
-\infty < t \leq \beta \} \subset \wO$, that $\widehat{X}_s \circ \widehat{\theta}_t= \widehat{X}_{s+t}$ for $s>0, \ t \geq 0$ and that $t \to \wX_t (\wo)$ is left continuous on $]0,-\alpha (\wo)[$ for $\wo \in \wO$ and is left continuous on $]0,\infty[$ for $\wo \in \wO \cap W(b)$.  
Note also that $\wX_s$ depends on $m$ through $Y^\ast$ and is not the same as the coordinate map $s \to \wo(-s)$ on $\wO$.
However for $t>0$, since $Y^\ast$ is adapred to $(\G^0_{t+})$,
\[
\widehat{\F}_t^0 
= \sigma(Y^\ast_{-s}; 0 < s \leq t) \mid_{\wO}
= \sigma (Y_s^\ast; -t \leq s <0) \mid_{\wO}
=\G_{\geq -t}^0 \mid_{\wO}
\]
and that $\widehat{\F}^0= \G^0 \mid_{\wO}$.
It is important that $\wX$ is $(\widehat{\F}_t^0)$ predictable.  
This holds because $Y^\ast$ is copredictable.  
Indeed $t \to Y_{-t}^\ast \circ r^{-1}$ is $(\wG_t^0)$ predictable by (3.3a).  
Therefore $t \to Y_{-t}^\ast$ is $(r^{-1} \wG_t^0) = \G_{\geq -t}^0$ predictable and $\wX$ is $\G_{\geq -t}^0 \mid_{\wO} =\widehat{\F}_t^0$ predictable.

\medskip\noindent
{\bf (4.2) Definition}
{\it 
A moderate Markov dual family $\{\widehat{P}^x; x \in E_{\Delta} \}$ for $X$ with respect to $m$ is a Borel measurable family of probability measures on $(\wO, \widehat{\F}^0)$ with $\widehat{P}^{\Delta} = \varepsilon_{[\Delta]}$ having the following two properties:
\begin{equation}\tag{4.3}
Q_m (F \circ \check{\theta}_T; \Lambda, -\infty <T<\beta)
= Q_m [\widehat{P}^{Y^\ast (T)} (F); \Lambda, -\infty <T<\beta]
\end{equation}
with $Q_m \ \sigma$-finite on $\G_{>T}^0 \mid_{\{-\infty <T<\beta\}}$ for all copredictable times $T, \ F \in p \widehat{\F}^0$ and $\Lambda \in \G_{>T}^0$;
\begin{equation}\tag{4.4}
\widehat{P}^x [F \circ \widehat{\theta}_T \mid \widehat{\F}_{T-}^0]
= \widehat{P}^{\wX(T)}(F) \text{ on } \{ 0<T<\infty\}
\end{equation}
for all $(\widehat{\F}_t^0)$ predictable times $T$ and $F \in p \widehat{\F}^0$.
}

\medskip
If $\{ \widehat{P}^x; x \in E_{\Delta} \}$ is a moderate Markov dual family for $X$ with respect to $m$, the family $\wX := \{ \wX_t; t >0\}$ under the laws $\{ \widehat{P}^x; x \in E_{\Delta}\}$ is called a left continuous moderate Markov dual processes for $X$ with respect to $m$.  
For a discussion of moderate Markov processes in general we refer the reader to [CW05] and [CG79].  Clearly (4.3) is equivalent to $Q_m (F \circ \check{\theta}_T \mid \G_{> T}^0) = \widehat{P}^{Y^\ast (T)} (F)$ on $\{-\infty < T<\beta\}$.

\medskip\noindent
Of course, (4.3) extends to $m$-copredictable times $T, \ \Lambda \in \G_{> T}^m$ and $F \in p \widehat{\F}^m$.

\medskip
Given a moderate Markov dual family $\{\widehat{P}^x; x \in E_{\Delta} \}$, there exists an associated dual semigroup $(\widehat{P}_t; t>0)$.  Define 
\begin{equation}\tag{4.5}
\widehat{P}_t f(x) := \widehat{P}^x (f \circ \wX_t), \
t>0, f \in p \E
\text{ or }
b\E.
\end{equation}
Taking $T=t$ and $F=f\circ \wX_s$ in (4.4) it is immediate that $(\widehat{P}_t; t>0)$ is a semigroup. 
Recall that $Q_m (Y_s \neq Y_s^\ast) =0$ for $s \in \R$.  
Therefore using (4.3) for the fourth equality, we have
\begin{align*}
m(g P_t f)
&= Q_m (g \circ Y_0 P_t f \circ Y_0) 
= Q_m(g \circ Y_0 f \circ Y_t)\\
&= Q_m (g \circ Y_{-t} f \circ Y_0)
= Q_m [\widehat{P}^{Y^\ast (0)} (g \circ \wX_t) f \circ Y_0]\\
&= m(\widehat{P}_t g \cdot f).
\end{align*}
That is $(P_t)$ and $(\widehat{P}_t)$ are in duality with respect to $m$.

\medskip
Here is Fitzsimmons' existence theorem for a moderate Markov dual family [F87].

\medskip\noindent
{\bf (4.6) Theorem}
{\it
There exists a moderate Markov dual family $\{ \widehat{P}^x; x \in E_{\Delta} \}$ for $X$ with respect to $m$.
Moreover $\{ \widehat{P}^x ; x \in E_{\Delta}\}$ is unique modulo Borel $m$-exceptional sets; that is, if $\{\widetilde{P}^x; x \in E_{\Delta}\}$ is another such family, then $\{ x: \widehat{P}^x \neq \widetilde{P}^x\}$ is a Borel $m$-exceptional set, i.e. an element of $\N(m) \cap \E$. 
}

Before coming to the proof of (4.6) we need to introduce a convenient metric on $\wO$ following the appendix of [F87].  Since $E$ is Lusin there exists a totally bounded metric $d$, say bounded by 1, on $E$ compatible with the topology of $E$ and this is extended to $E_\Delta$ by setting $d(x,\Delta)=2$ for $x \in E$ and $d(\Delta, \Delta)=0$.
Next, define a metric $\rho$ on $\wO= \{w: \beta(w) =0\} \cup \{[\Delta]\}$ by 
\[
\rho (w, w') =\int_{-\infty}^0 e^t d[w(t), w'(t)] dt.
\]
Note that $\wO$ consists of functions from $\mathbf{R}$ to $E_\Delta$ that are right continuous except at $\alpha <0$ and are constantly equal to $\Delta$ on $[0,\infty[$.  
Thus elements of $\wO$ may be thought of as functions on $]-\infty,0[$ and this is convenient at times.
Clearly $\rho$ is a metric on $\wO$ bounded by 2 and the topology induced by $\rho$ is the topology of convergence in measure relative to $\eta(dt) := e^t dt$ on $]-\infty,0[$.  
The next lemma contains the properties of the metric space $(\wO ,\rho)$ that we shall need.

\medskip\noindent
{\bf (4.7) Lemma}
{\it
\begin{compactenum}[(i)]
\item
$(\wO, \rho)$ is separable.

\item
If $\B(\wO)$ denotes the Borel $\sigma$-algebra associated with the metric space $(\wO, \rho)$, then $\B(\wO)= \widehat{\F}^0$.

\item
There exists a countable class $\C(\rho)$ of $\rho$-uniformly continuous functions from $\wO$ to $[0,1]$, closed under finite products, such that $\widehat{\F}^0 =\B (\wO)= \sigma[\C(\rho)]$.
\end{compactenum}
}

\begin{proof}
Property (i) is easily checked.  
For example let $D$ be a countable dense subset of $E_\Delta$ with $\Delta \in D$ and let $\D$ be the $D$-valued elements of $\wO$ that are constant on each dyadic interval $I_{n,k} := [(k-1) 2^{-n}, k 2^{-n}[, \ -n 2^n \leq k \leq 0, n \geq 1$.  
Then $\D$ is easily seen to be dense in $\wO$.  For (ii) let $\overline{E}_\Delta$ be the compact completion of $(E_\Delta, d)$, and fix $f \in pC(\overline{E}_\Delta)$.  
Then $f$ is bounded and $d$-uniformly continuous on $E_\Delta$.  For $w \in \wO$, set for $n \geq 1, t \in ] -\infty,0[,$
\[
\phi_{n,t,f} (w) 
:= ne^{-t} \int_t^{t +1/n} e^s f \circ w(s) \ ds.
\]
Note that $\phi_{n,t,f} (w) \to f \circ w(t)$ as $n \to \infty$ provided $t \neq \alpha(w)$. 
If $w_k \to w$ in $\wO$, then $f \circ w_k \to f \circ w$ in $\eta$ measure, and so $\phi_{n,t,f} (w_k) \to \phi_{n,t,f} (w) $ by the bounded convergence theorem.  
Therefore $\phi_{n,t,f}$ is continuous on $\wO$, and hence $\B(\wO)$ measurable.  
Since $\Delta$ is isolated in $\overline{E}_\Delta, \ f:=1$ on $\overline{E}_\Delta \setminus \{\Delta\}$ and $f(\Delta)=0$ is uniformly continuous on $\overline{E}_\Delta$.  
Then $\phi_t(w) := \int_{-\infty}^t e^s f \circ w(s) \ ds$ is continuous on $\wO$ for $t <0$.
But $\{ \alpha \geq t\}= \{ \phi_t =0 \}$ and so $\alpha \in \B(\wO)$.  
Combining these observations, $1_{\{\alpha \neq t\}} f \circ Y_t$ is $\B(\wO)$ measurable and since $1_{\{\alpha =t\}} f \circ Y_t =1_{\{\alpha=t\}} f(\Delta)$, it follows that $f \circ Y_t \mid_{\wO}$ is $\B(\wO)$ measurable for $t <0$ and $f \in C(\overline{E}_\Delta)$.  
Consequently $\widehat{\F}^0= \G^0 \mid_{\wO} \subset \B(\wO)$.

\medskip
For the opposite inclusion suppose that $G \subset \wO$ is open.
Then there exists an increasing sequence $(G_n)$ of open sets with $G_n \subset \overline{G}_n \subset G_{n+1} \subset G$ with $G_n \uparrow G$.
Let $f_n(w) := [n\rho (w, \overline{G}_n^c)] \wedge 1$ and note that $f_n \uparrow 1_G$.  
If $w' \in G$ is fixed, $w \mapsto \rho(w,w')$ is $\widehat{\F}^0$ measurable, and since $\overline{G}_n^c$ is open, if $\D$ is a countable dense subset of $\wO$, then
\[
w \mapsto \rho(w, \overline{G}_n^c) = \infty \{ \rho(w, w') : w' \in \D \cap \overline{G}_n^c\}
\] 
is $\widehat{\F}^0$ measurable.  Therefore $G \in \widehat{\F}^0$ and this establishes (ii).

\medskip
For (iii) let $(G_n)$ be a countable base for the topology of $\wO$.  For each $n$, there exists a sequence $(F_{n,k})$ of $\rho$-uniformly continuous functions such that $0 \leq F_{n,k} \uparrow 1_{G_n}$ as $k \to \infty$.  
Then the closure under finite products of $\{F_{n,k}; n \geq 1, k \geq 1\}$ has the required properties.
\end{proof}

\medskip\noindent
{\bf (4.8) Remark}
Let $\wO^-$denote the completion of $(\wO, \rho)$ and let $\overline{\rho}$ denote the extension of $\rho$ to $\wO^-$.  
Then $(\wO^-, \overline{\rho})$ is a complete separable metric space; in particular, $\wO^-$ with the $\overline{\rho}$ topology is a polish space.

\medskip
We prepare one more lemma for the proof of (4.6).  If $\mu$ is a measure on $(E, \E)$, then $B \in \E$ is $\mu$-polar if and only if $P^\mu (D_B <\infty)=0$ where $D_B := \infty \{ t  \geq 0: X_t \in B \}$, and this is obviously equivalent to $P^\mu (e^{-D_B}) =0$.  
Because $m \in Exc, \ B$ is $m$-polar if and only if $P^m (T_B <\infty) = P^m (e^{-T_B})=0$.  
Now let $\mu$ be a probability measure equivalent to $m$ and define 
\begin{equation}\tag{4.9}
I(B) =P^\mu (e^{-D_B}); \ B \in \E.
\end{equation}

\medskip\noindent
{\bf (4.10) Lemma}
{\it
Let $A$ be a Souslin subset of $E$.  Then $A$ is $m$-polar if and only if $I(B)=0$ for all Borel subsets $B \subset A$.  See [DM III, 16] for the definition of a Souslin set.
}

\begin{proof}
Let $\overline{E}_\Delta$ be a Ray compactification of $E_\Delta$ for $X$ and $\overline{\G} (\ \overline{\mathcal{K}}\ )$ denote the Ray open (compact) subsets of $\overline{E}_\Delta$.
Since $X$ is a Borel right process $\E =\E^r := \overline{\E}_\Delta \mid_E$ where $\overline{\E}_\Delta$ is the Borel $\sigma$-algebra of $\overline{E}_\Delta$.
Extend $I$ to $\overline{\E}_\Delta$ by $I(B) =P^\mu (e^{-D_B})$ for $B \in \overline{\E}_\Delta$.  
Since $\mu$ is carried by $E, \ D_B =D_{B \cap E}$ a.s. $P^\mu$ for $B \in \overline{\E}_\Delta$.
It is shown during the proof of (12.10) in [G75], that if $I^\ast(A) := \infty \{I(G) : G \in \overline{\G}, G \supset A\}$ for $A \subset \overline{E}_\Delta$, then $I^\ast$ restricted to $E$ is a Choquet capacity on $E$ since $E$ is contained in the non-branch points $E_\Delta$.
Hence if $A$ is a Souslin subset of $E, \ A$ is capacitable and since $\E =\E^r$ it follows that
\[
\sup \{ I(B); B \subset A, \ B \in \E\} 
= I^\ast (A) 
=\infty \{ I(B): B \supset A, \ B \in \E\}.
\]
It is immediate from this that $A$ is $m$-polar if and only if $I^\ast(A) =0$. 
\end{proof}

\medskip\noindent
{\bf (4.11) Corollary}
{\it
Let $A$ be a Souslin subset of $E$.  Then $A$ is $m$-inessential if and only if all Borel subsets $B \subset A$ are $m$-inessential.
}

\begin{proof}
Let $B \in \N(m) \cap \E$ and $B \subset A$.  
Then $B$ is $m$-polar and if $m =\xi + \rho U$ is the decomposition of $m$ into its harmonic and potential parts, $\rho (B)=0$.  
Therefore (4.10) implies that $A$ is $m$-polar, and since $A$ is capacitable $\rho (A) =0$.  
Hence $A \in \N(m)$.  
The reverse implication is obvious.
\end{proof}

\medskip
We turn now to the proof of Theorem 4.6 which is rather long. 
Let $\widehat{\F}^-$ denote the Borel $\sigma$-algebra of $\wO^-$ defined in (4.8).  
Since the topology of $\wO$ is the subspace topology it inherits from $\wO^-$, it is easy to see that $B(\wO) = \widehat{\F}^- \mid_{\wO}$. 
Then (4.7) implies that $\widehat{\F}^0 =\widehat{\F}^- \mid_{\wO}$. 
However we do not know if $\wO \in \widehat{\F}^-$ in general. 
The bounded convergence theorem shows that if $w \in W$ (resp. $\wo \in \wO$) then $t \to \check{\theta}_t w$ (resp. $t \to \widehat{\theta}_t \wo$) is $\rho$-continuous as a map from $\R$ (resp. $]0,\infty[$ ) to $\wO$ since $w$ being right continuous from $]\alpha (w), \infty [$ to the Lusin space $E_\Delta$ and constant on $]-\infty, \alpha (w)]$ implies that it has at most a countable number of discontinuities.  Let $\mathbf{C}$ denote the bounded uniformly continuous functions on $(\wO^-, \overline{\rho})$. 
If  $F \in p \mathbf{C}$, then on $W, \ t \to (F1_{\wO}) \circ \check{\theta}_t = F \circ \check{\theta_t}$ is continuous on $\R$.  Therefore the process $U_t^F := 1_{[ \alpha, \beta [}(t) (F \circ 1_{\wO} ) \circ \check{\theta}_t$ is right continuous, adapted to $(\G_{t}^0)$ and vanishes on $]-\infty, \alpha[$.  
Consequently $U^F \in \O^0$.  
Since $\check{\theta}_{t-s} \sigma_s = \check{\theta}_t, \ U^F$ is also homogeneous.  Then
\begin{equation}\tag{4.12}
Z^F := {^{\wp} (1_{\Lambda^\ast} U^F)} 
= 1_{\Lambda^\ast} {^{\wp} U^F}
\end{equation}
since $\Lambda^\ast =1_E \circ Y^\ast \in \wP^0$.  
Therefore (3.24) implies that $Z^F \in (\O \cap \wP^m)$.  
On the other hand using the uniqueness of copredictable projections,
${^{\wp} U_F}$ is homogeneous.  
Hence from (3.25), $Z^F$ is $Q_m$-indistinguishable from a process of the form $f^F \circ Y^\ast$ where $f^F \in pb \E$ and (2.14) implies that $f^F$ is unique modulo $\N_0 (m) := \N(m) \cap \E$.  
Of course, $f^F$ is extend to $\Delta$ by $f^F(\Delta)=0$.
The existence of $f^F$ for $F \in b \widehat{\F}^-$ with $Z^F$ being defined by (4.12) follows by a monotone class argument.  
It is evident that the map $F \to f^F$ is a pseudo-kernel from $E$ to $\wO^-$ relative to $\N_0 (m)$ as defined in [DM IX; II].
Since $\wO^-$ is polish [DM IX; II] implies the existence of a sub-Markovian kernel $L_0 =L_0 (x, dw)$ from $(E, \E)$ to $(\wO^-, \widehat{\F}^-)$ such that for each $F \in b\widehat{\F}^-, \ \{x: L_0 (x,F) \neq f^F (x) \} \in \N_0 (m)$.  
Combining this with (4.12), for $F \in pb \widehat{\F}^-$ and any $Q_m$-copredictable time $T$ one has
\begin{equation}\tag{4.13}
Q_m [(F 1_{\wO}) \circ \check{\theta}_T; Y_T^\ast \in E]
= Q_m [L_0 (Y_T^\ast, F); Y_T^\ast \in E]
\end{equation}
since $Y^\ast \in \wP^0$.  
Take $F=1_{\wO^-}$ in (4.13).
Then since 
$1_{\wO}(\check{\theta}_t \omega) = 1$
on
$\rrbracket \alpha, \beta \llbracket$  
it follows from the uniqueness of copredictable projections
that $\{x: L_0 (x, \wO^-) <1\} \in \N_0(m)$.

\medskip
We shall modify $L_0$ to obtain a kernel $L$ from $(E_\Delta, \E_\Delta)$ to $(\wO, \widehat{\F}^0)$ for which (4.3) holds with $\widehat{P} =L$.  
As a first step let $B := \{x ; L_0 (x, \wO^-) < 1\} \in \N_0(m)$ and define $L(x, \cdot) = L_0 (x, \cdot)$ if $x \in E \setminus B$ and $L(x, \cdot) =\varepsilon_\Delta$ if $x \in B \cup \{ \Delta \}$.  
Then $L$ is a Markov kernel from $(E_\Delta, \E_\Delta)$ to $(\wO^-, \widehat{\F}^-)$ for which (4.13) holds with $L_0$ replaced by $L$.  
Now $]0,\infty[$ and $]-\infty,0[$ are order isomorphic and so from (2.16), $\wO$ is co-Souslin.  
Hence [DM III; 18] implies that $\wO^- \setminus \wO$ is 
Souslin---recall $\wO^0$ is polish.
Next define 
\[
A:= \{x \in E: L(x, \wO^- \setminus \wO) >0\}.
\]
Let $B \subset A, \ B \in \E$ and take $F =1_{\wO^- \setminus \wO}$ in (4.13) to see that for copredictable times $T$,
\[
0=Q_m ((F \cdot 1_{\wO}) \circ \check{\theta}_T; \ Y_T^\ast \in B)
=Q_m (L(Y_T^\ast, \wO^- \setminus \wO); \ Y_T^\ast \in B).
\]
But $L (Y_T^\ast; \wO^- \setminus \wO)>0$ on $\{Y_T^\ast \in B\}$ and so $Q_m (Y_T^\ast \in B)=0$.  
Consequently from (3.13), $B \in \N_0(m)$ and so $A \in \N(m)$ by (4.11).
By definition there exists $B \in \N_0(m)$ with $A \subset B$.  
Modify $L$ by setting $L(x,\cdot) =\varepsilon_\Delta$ if $x \in B$.  
Denoting the modified kernel again by $L, \ L(x,\cdot)$ is carried by $\wO$ for all $x \in E_\Delta$.
Thus $L$ is a Markov kernel from $(E_\Delta, \E_\Delta)$ to $(\wO^-, \widehat{\F}^-)$ that is carried by $\wO$ and satisfies (4.13) with $L_0$ replaced by $L$.  
However 
\[
\{ -\infty < t< \beta \} \subset \{Y_t^\ast \in E \} \cup 
\{
Y_t^\ast = \Delta, t \leq \alpha \},
\]
and since $\check{\theta}_tw =[\Delta]$ if $t \leq \alpha$ and $Q_m$ doesn't change $\{[\Delta]\}$, we may replace $\{Y_T^\ast \in E\}$ by $\{-\infty<T< \beta\}$ in (4.13).
Finally given $\Lambda \in \G_{>T}^0$, define $T_\Lambda =T$ on $\Lambda$ and $T_\Lambda =-\infty$ on $\Lambda^c$.  
Then $T_\Lambda$ is copredictable (see [DM IV; 73c] in the predictable case and use (3.3b)) and so writing (4.13) for $T_{\Lambda}$ we obtain for $F \in p \widehat{\F}^0$,
\begin{equation}\tag{4.14}
Q_m (F \circ \check{\theta}_T; \ \Lambda, \ -\infty<T<\beta)
=Q_m (L(Y_T^\ast, F) ; \ \Lambda, \ -\infty<T<\beta).
\end{equation}
Therefore (4.3) holds with $\widehat{P}^x =L(x,\cdot)$ although $L$ is a kernel from $(E_\Delta, \E_\Delta)$ to $(\wO^-, \widehat{\F})$.  
But $\widehat{\F}^0 =\widehat{\F}^- \mid_{\wO}$ and $L$ is carried by $\wO$.  
Hence we may restrict $L$ to $\wO$ to obtain a kernel from $(E_\Delta, \E_\Delta)$ to $(\wO, \widehat{\F}^0)$ satisfying (4.3) which is denoted by $L$ again.

\medskip
The next step of the proof consists in showing that $L$ satisfies (4.4) in addition to (4.3).  Let $\P$ denote the $(\widehat{\F}_t^0)$ predictable processes on $\wO$.  
Given $Z \in b\P$ and $F \in b \widehat{\F}^0$ we claim that for $t>0$ the Borel set
\begin{equation}\tag{4.15}
\{ x: L(x, Z_t F \circ \widehat{\theta}_t) \neq L(x, Z_t L( \wX_t, F)\}
\end{equation}
is $m$-inessential.  
It suffices to check this for $Z$ and $F$ positive.
Since $Z \in \P, \ Z_t \in \widehat{\F}_{t-}^0 =\G_{>-t}^0 \mid_{\wO}$ for $t >0$.
If $T$ is copredictable and $H \in \G_{>-t}^0 \mid_{\wO}$  with $t>0$, then $H \circ \check{\theta}_T \in \G_{>(T-t)}^0$.  
Hence $Z_t \circ \check{\theta}_T \in \G_{>(T-t)}^0$.  
But $t>0$ and so $T-t$ is also copredictable.  
Therefore using (4.14) for the first, third and fifth equality below, we compute
\begin{align*}
Q_m
& (L(Y_T^\ast, \ Z_t F \circ \widehat{\theta}_t); \ -\infty<T<\beta)\\
& =Q_m (Z_t \circ \check{\theta}_T F \circ \widehat{\theta}_t \circ \check{\theta}_T; \ -\infty<T<\beta)\\
& =Q_m (Z_t \circ \check{\theta}_T F \circ \check{\theta}_{T-t}; \ -\infty<T<\beta)\\
&= Q_m (Z_t \circ \check{\theta}_T L(Y_{T-t}^\ast, F); \ -\infty<T<\beta)\\
&= Q_m (Z_t \circ \check{\theta}_T L(\wX_t, F) \circ \check{\theta}_T; \-\infty<T<\beta)\\
&= Q_m (L(Y_T^\ast, Z_t L(\wX_t, F)); -\infty<T<\beta).
\end{align*}
But $t \to L(Y_t^\ast; \ G)$ is copredictable for $G \in \widehat{\F}^0$ and so (3.13) implies that the set in (4.15) is $m$-inessential as claimed. 
Let $\C(\rho)$ be as in (4.7-iii).  
Since $\widehat{\F}^0 =\G^0 \mid_{\wO}$ we may extend $L$ to a kernel from $(E_\Delta, \E_\Delta)$ to $(W, \G^0)$ by setting $L(x, \Lambda) =L(x, \Lambda \cap \wO)$ for $\Lambda \in \G^0$.  
Thus $L$ may be regarded either as a kernel from $(E_\Delta, \E_\Delta)$ to $(\wO, \widehat{\F}^0)$ or to $(W, \G^0)$.  
Bearing this in mind fix $F \in \C(\rho)$ and define
\begin{equation}\tag{4.16}
\Gamma := \{ w \in W: t \to L(Y_t^\ast (w), F) 
\text{ is right continuous on } \R \}.
\end{equation}
It follows from [DM IV; 19] that $\Gamma^c =W \setminus \Gamma$ is a Souslin subset of the co-Souslin space $(W, \G^0)$.  
Hence by (2.15)
\begin{equation}\tag{4.17}
J:= \{ x \in E_\Delta: L(x, \Gamma^c) >0\}
\end{equation}
is Souslin.
Thus in view of (4.11) in order to show that $J$ is $m$-inessential it suffices to show that if $B \subset J, \ B \in \E$, then $B$ is $m$-inessential. 
But $1_B \circ Y^\ast \in \wP^0$ and so by (3.13) it suffices to show that $Q_m [1_B \circ Y_T^\ast; -\infty <T<\beta]=0$ for copredictable times $T$.  But this certainly will follow if we show that for any copredictable time $T$,
\begin{equation}\tag{4.18}
Q_m(L(Y_T^\ast, \Gamma^c); - \infty<T<\beta)=0.
\end{equation}
Since $F \in \C(\rho)$ the process $t \to 1_{]-\infty,\beta[}(t) F \circ \check{\theta}_t$ is right continuous and $t \to L(Y_t^\ast, F)$ is a version of its $Q_m$ copredictable  projection (3.15).
As a result $t \to L(Y_t^\ast, F)$  is $Q_m$ a.s. right continuous (3.18). 
Thus the left side of (4.18) equals
\begin{align*}
Q_m
&(1_{\Gamma^c} \circ \check{\theta}_T; -\infty <T<\beta)\\
&= Q_m (t \to L(Y_t^\ast , F) \text{ is not rc on } 
]-\infty, T[; \ T<\beta]\\
&=0,
\end{align*}
and so $J$ is $m$-inessential.

\medskip
Since $\P$ is separable [DM IV; 67.2] we may choose a countable set $\{Z^n; \ n \geq 1\}$ of positive, bounded, left continuous processes closed under finite products which generates $\P$.  
Let $\{ F^n, \ n \geq 1\}$ be an enumeration of $\C(\rho)$.  
Since $J$ is $m$-inessential, for each $n \geq 1$ there exists a set $N_n \in \N_0 (m)$ such that for each $x \not \in N_n, \ t \to L (\wX_t, F^n)$ is left continuous on $]0,\infty[$ a.s. $L(x,\cdot)$.  
Also by (4.15), for each $n \geq 1, \ k \geq 1$ and rational $t$ the set
\[
N_{n,k,t} := \{ x: L(x, Z_t^k F^n \circ \widehat{\theta}_t) 
\neq L(x, Z_t^k L(\wX_t, F^n) \}
\]
is in $\N_0(m)$.  Consequently
\[
N:= (\underset{n \geq 1} \cup N_n) \cup (\underset{n,k \geq 1, t >0 ,\text{ rational}} \cup N_{n,k,t})
\]
is a Borel $m$-inessential set.  If $x \in E \setminus N$, then
\begin{equation}\tag{4.19}
L(x, Z_t^k F^n \circ \widehat{\theta}_t)
= L(x, Z_t^k L(\wX_t, F^n))
\end{equation}
for $n,k \geq 1$ and all $t$ since both sides of (4.19) are left continuous by the discussion at the beginning of this paragraph.  
But $\{Z^k : k \geq 1\}$ generates $\wP$ and so for each $x \in E\setminus N, \ t \to L(\wX_t, F^n)$ is a version of the $L(x, \cdot)$ predictable projection of $L(x, F^n \circ \widehat{\theta}_t)$.  
Since 
$\{F^n; \ n \geq 1 \} = \C(\rho)$
is closed under finite products and
$\widehat{\F}^0 = \B (\wO)= \sigma( \C(\rho))$,
\begin{equation}\tag{4.20}
L(x, F \circ \widehat{\theta}_T; \ 0<T<\infty)
= L(x, L(\wX_T, F);\ 0<T<\infty),
\end{equation}
provided $x \in E \setminus N, \ F \in b\widehat{\F}^0$, and $T$ an $(\widehat{\F}^0_t)$ predictable time. 
Define $\widehat{P}^x =L(x, \cdot)$ if $x \in E \setminus N$ and  $\widehat{P}^x =\varepsilon_{[\Delta]}$ for $x \in N \cup \{[\Delta]\}$.  
Then $(\widehat{P}^x; \ x \in E_\Delta)$ is a moderate Markov dual family for $X$ with respect to $m$ .

\medskip
It remains to establish the uniqueness assertion in order to complete the proof of Theorem 4.6.  But this is easy.  
Extend $\widetilde{P}$ to a kernel from $(E_\Delta, \E_\Delta)$ to $(W, \G^0)$ by setting $\widetilde{P}^x(F) = \widetilde{P}^x (F 1_{\wO})$ for $F \in b\G^0$.
Recall that $\widehat{\F}^0 =\G^0 \mid_{\wO}$. 
If $D$ is a countable dense subset of $\mathbf{C}$, then
\[
\{ x: \widehat{P}^x \neq \widetilde{P}^x\} =
\underset{F \in D} \cup 
\{x: \widehat{P}^x (F) \neq \widetilde{P}^x (F) \}.
\]
Since both $\widehat{P}^x$ and $\widetilde{P}^x$ are moderate Markov dual families for $X$ it follows from (4.3) and (3.13) that $\{x: \widehat{P}^x (F) \neq \widetilde{P}^x (F) \} \in \N_0 (m)$ and hence so is $\{ x: \widehat{P}^x \neq \widetilde{P}^x \} \in \N_0 (m)$.
\qed

\medskip
Armed with the existence of a dual family $(\widehat{P}^x; \ x \in E_\Delta)$ we are now able to express the copredictable projection, ${^{\wp} Z}$, by a kernel as was done for the optional projection, ${^0 Z}$, in (3.21).  
To this end we define a ``splicing map" from $\wO \times \R \times \Omega \to W$ by
\begin{equation}\tag{4.21}
\begin{split}
(\wo \lvert t \rvert \omega)(s)
& =\wo(s-t) \text{ if } s <t\\
&=\omega^+(s-t) \text{ if } s \geq t.
\end{split}
\end{equation}
Recall that $\omega^+ (t) =\omega(t+), \ t \geq 0$.
Then the splicing map from $W \times \R \times \Omega \to W$ defined in (3.19) is given by $w \lvert t \rvert \omega = \check{\theta}_t w \lvert t \rvert \omega$ in the notation of (4.21).  
If $(t,w) \in \Lambda^\ast$, then $w = \check{\theta}_t w \lvert t \rvert \theta_t w$.  
Also in this notation (3.20) becomes
\begin{equation}\tag{4.22}
{^0 Z_t} (w) = 1_{\Lambda^\ast} (t,w) 
\int_{\Omega} Z_t (\check{\theta}_t w \lvert t \rvert \omega) P^{Y^\ast( t,w)} (d\omega),
\end{equation}
for $Z \in p (\B \times \G^0)$.

\medskip\noindent
{\bf (4.23) Proposition}
{\it
If $Z \in p ( \B \times \G^0)$, then a version of ${^{\wp}Z_t}$ is given by
\begin{equation}\tag{4.24}
(t,w) \to 1_{\Lambda^\ast} (t,w) 
\int_{\wO} Z_t (\wo \lvert t \rvert \theta_t w) \widehat{P}^{Y^\ast (t,w)} (d\wo).
\end{equation}
}

\begin{proof}
Using (4.3) in place of (2.13), the proof parallels that of (3.21).  
We omit translating the details.
\end{proof}

\medskip
We end this section with a version of the familiar commutation of the projections.

\medskip\noindent
{\bf (4.25) Theorem}
{\it
Let $Z \in p \M$.  Then $^{0 {\wp}} Z$ and $^{\wp 0} Z$ are $Q_m$ indistinguishable.
}

\begin{proof}
It suffices to prove this for $Z \in p (\B \times \G^0)$ since both the optional and copredictable projections of a $Q_m$-evanescent process are $Q_m$-evanescent.
We shall prove (4.25) by showing that a common version of both $^{0 {\wp}} Z$ and $^{\wp 0} Z$ is given by
\begin{equation}\tag{4.26}
\widetilde{Z} (t,w) =1_{\Lambda^\ast} (t,w) 
\int_{\wO} \int_\Omega Z_t (\wo \lvert t \rvert \omega) 
P^{Y^\ast (t,w)} (d \omega) \widehat{P}^{Y^\ast (t,w)} (d\wo).
\end{equation}
Fix $Z \in p(\B \times \G^0)$.  Then from (4.24)
\[
^{\wp 0} Z_t (w) = 1_{\Lambda^\ast}(t,w) 
\int_{\wO} {^0 Z_t} ( \wo \lvert t \rvert \theta_t w)
\widehat{P}^{Y^\ast (t,w)} (d\wo).
\]
Setting $\widetilde{w} =\wo \lvert t \rvert \theta_t w$ and using (4.22) this becomes 
\[
1_{\Lambda^\ast}(t,w) \int_{\wO} 1_{\Lambda^\ast}(t,\widetilde{w})
\int_\Omega Z_t (\check{\theta}_t \widetilde{w} \lvert t \rvert \omega)
P^{Y^\ast (t,\widetilde{w})} (d\omega) \widehat{P}^{Y^\ast (t,w)} (d\wo).
\]
Now $\widetilde{w}(s) =\wo (s-t)$ if $s<t$ and $\widetilde{w}(s) = (\theta_t w)^+ (s-t) =w(s+)$ if $s \geq t$.  
Hence $\check{\theta}_t \widetilde{w} (s) =\wo (s)$ if $s <t$ and so $\check{\theta}_t \widetilde{w} \lvert t \rvert \omega = \wo \lvert t \rvert \omega$.  
But $Y^\ast$ is copredictable which implies that $Y_t^\ast = Y_t^\ast \circ b_t$ and $b_t \widetilde{w} =b_t w$ since $(t,w) \in \Lambda^\ast$.  
Hence $Y_t^\ast( \widetilde{w})= Y_t^\ast(w)$.  
Combining these facts we see that $^{\wp 0} Z= \widetilde{Z}$.

\medskip
The computation for $^{\wp 0} Z$ is slightly different.  
Arguing just as above one obtains for each fixed $t$ and $w$
\[
^{0 \wp} Z_t (w) =1_{\Lambda^\ast}(t,w) 
\int_\Omega 1_{\Lambda^\ast} (t,\widetilde{w}) 
\int_{\wO} Z_t (\wo \lvert t \rvert \theta_t \widetilde{w})
\widehat{P}^{Y^\ast (t,\widetilde{w})} (d\wo) 
P^{Y^\ast(t,w)} (d\omega)
\]
where in this case $\widetilde{w} =(\check{\theta}_t w \lvert t \rvert \omega)$.  
Thus if $s \geq t, \ \widetilde{w}(s) =\omega^+(s-t)$ and so $\theta_t \widetilde{w} =\omega^+$.  
Hence $Y_t^\ast (\widetilde{w}) = \omega(0+) = X_0 (\omega)$. 
But for each $x \in E, \ P^x (X_0 =x) =1$ which implies that if $(t,w) \in \Lambda^\ast$, then a.e. $P^{Y^\ast (t,w)}$ in $\omega$ one has $Y^\ast(t,w) =X_0 (\omega)$.  
But $Y^\ast (t,\widetilde{w})= X_0 (\omega)$.
Combining these facts yields $^{0 \wp}Z =\widetilde{Z}$.
\end{proof}




\noindent
{\bf 5. Homogenous Random Measures}


\medskip
In this section some of the tools developed in preceding sections will be applied to prove Fitzsimmons' basic existence theorem for homogeneous random measures (abbreviated HRMs). We begin with a preliminary definition. As before $m \in Exc$ is fixed and $Q_m$ is the corresponding Kuznetsov measure. A {\it random kernel} is a {\it positive} kernel from $(W, \G^m)$ to $(\R, \B(\R)), (w,B) \to K(w, B)$ such that $K = \sum_{n \geq 1} K_n$ where each $K_n$ is such a kernel with $Q_m(K_n(\R) = \infty) =0$ for each $n \geq 1$. Two random kernels $K$ and $L$ are $Q_m$-indistinguishable provided $K(w, \cdot) = L(w, \cdot)$ for $Q_m$ a.e. $w$. Since $m$ is fixed we shall just write indistinguishable (resp. evanescent) for $Q_m$-indistinguishable (resp. $Q_m$-evanescent). Notice that a random kernel, $K$, is indistinguishable from $\sum_{n \geq 1} K_n$ where each $K_n$ is a subMarkovian kernel, i.e. satisfying $K_n(w, \R) \leq 1$ for all $w \in W, n \geq 1$. Indeed, define $a_{n,k}(w) = (K_n(w, \R) \wedge k) - (K_n(w, \R) \wedge (k-1))$ if $K_n(w, \R) < \infty$ and $a_{n,k}(w)=0$ if $K_n(w, \R) = \infty$ so that $a_{n,k}(w) \leq 1$ and $\sum_{k \geq 1} a_{n,k}(w) = K_n(w, \R)$ if $K_n(w, \R) < \infty$. Now set $$ K_{n,k}(w, \cdot) = a_{n,k}(w)K_n(w, \cdot) \big/ K_n(w, \R) $$ if $K_n(w,\R) < \infty$ and equal to zero otherwise. Relabeling the $K_{n,k}$ as a single sequence suffices. This countable finiteness condition justifies the use of Fubini's theorem in the sequel. Recall the definition of $\M^m$ in the first paragraph of Section 3. A random kernel $K$ is {\it carried} by a set $\Lambda \in \M^m$ provided $Q_m \int 1_{\Lambda^c}(t, \cdot) K(\cdot, dt) = 0.$ We shall often write just $K(dt)$ in place of $K(\cdot, dt)$ in such integrals. Since $m \in Exc$ is fixed in this section we shall write $\M$ for $\M^m$.

\medskip\noindent
{\bf (5.1) Definition}
{\it A random measure $\kappa$ is a random kernel which is carried by $\Lambda^*$.}

\medskip\noindent
Recall that $\Lambda^* = \{ (t,w) : Y_t^*(w) \in E \}$ and that $\rrbracket \alpha, \beta \llbracket \subset \Lambda^* \subset \llbracket \alpha, \beta \llbracket$. This definition differs slightly from that in [Fi87]. Fitzsimmons just assumes that $\kappa$ is carried by $\llbracket \alpha, \beta \llbracket$. But only random measures carried by $\Lambda^*$ will arise in what follows and so it is convenient to build it into the definition. A random measure (RM) $\kappa$ is {\it $\sigma$-integrable} over a $\sigma$-algebra $\H \subset \M$ provided there exists $Z \in p\H$ with $Q_m \int Z_t \kappa(dt) < \infty$ and $\kappa$ carried by $\{Z > 0\}$. Clearly one may suppose that $Z>0$ without loss of generality. This is equivalent to the {\it Doleans} measure, $M_\kappa$, of $\kappa$ defined by 
\begin{equation}\tag{5.2}
M_\kappa(Z) = Q_m \int Z_t \kappa(dt), Z \in p\M
\end{equation}
being $\sigma$-finite on $\H$. The class of random measures $\sigma$-integrable over $\H$ is denoted by $\sigma\I(\H)$. We define equality of RMs to mean indistinguishability. Since $\B$ is countably generated, if $\kappa_1$ and $\kappa_2$ are RMs in $\sigma \mathcal{I}(\mathcal{M})$, then $\kappa_1 = \kappa_2$ if and only if for each $B \in \B$, $\kappa_1(\cdot,B) = \kappa_2(\cdot,B)$ a.e. $Q_m$. If $\kappa$ is a RM and $Z \in p\M$, let $(Z * \kappa)(B) := \int_B
 Z_t \kappa(dt)$. Then $Z * \kappa$ is a RM provided it satisfies the appropriate finiteness condition, in particular if $Z$ is bounded.
 
 The following proposition is the first step in the constructions to follow. It is a standard result but we shall prove it for completeness.
 
 \medskip\noindent
 {\bf (5.3) Proposition}
 {\it Let $M$ be a $\sigma$-finite measure on $\M$ carried by $\Lambda^*$. Then $M$ is the Doleans measure $M_\kappa$ of a unique} RM, {\it $\kappa \in \sigma\I(\M)$ if and only if $M$ does not change evanescent sets.}
 \begin{proof}
 If $\kappa \in \sigma\I(\M)$ then it is clear from the definition (5.2) that $M_\kappa$ is $\sigma$-finite, carried by $\Lambda^*$ and does not change evanescent sets. For the converse suppose first that $M$ is finite. Let $\I \subset \M$ be the ideal of evanescent processes. Let $\varphi \in b\B$ and $f \in pb\G^o$. Define $M^\varphi(f) = M(\varphi \otimes f)$ where $\varphi \otimes f(t, w) = \varphi(t) f(w)$. If $f=0$ a.e. $Q_m$, then $\varphi \otimes f \in \I$ and so $M^\varphi(f)=0$. Therefore the finite signed measure $M^\varphi << Q_m$. If $\varphi \geq 0$ then $M^\varphi$ is a positive measure with $M^\varphi(1) \leq \sup \varphi \cdot M(1) < \infty$. Therefore there exists $f_1 \in \G^o$ with $0 \leq f_1 < \infty$ such that $M^1 = f_1 Q_m$. Clearly $M^\varphi << M^1$ for $\varphi \in b\B$, and so $$M^\varphi(f) = \int f(w) k(w, \varphi) M^1(dw)$$ where $k(\cdot, \varphi)$ is a version of $dM^\varphi / dM^1$. It is evident that $\varphi \to k(\cdot, \varphi)$ is a subMarkovian pseudo kernel from $(W, \G^o)$ to $(\R, \B)$ relative to $\N := \{ \Lambda \in \G^o : M^1(\Lambda) = 0 \}$ as defined in [DM VI; 11]. See also [G75b]. Consequently by [DM VI; 13] or [G75b; 4.5] there exists a subMarkovian kernel $K$ from $(W, \G^o)$ to $(\R, \B)$ such that $k(\cdot, \varphi) = K\varphi$ a.e. $Q_m$ for $\varphi \in b\B$. Finally define $\kappa(w, \varphi) = f_1(w) K(w, \varphi)$. Then 
 \begin{equation}\tag{5.4}
 \begin{split}
M(\varphi \otimes f) &= \int_W f(w) K(w, \varphi) M^1(dw)\\
&= \int_W f(w) K(w, \varphi) f_1(w) Q_m(dw)\\
&= \int_W \int_\R \varphi(t) f(w) \kappa(w, dt) Q_m(dw).
\end{split}
\end{equation}
Let $N_\kappa(F) := Q_m \int_\R F(t, \cdot) \kappa(\cdot, dt)$ for $F \in p(\B \times \G^o)$. Then since $K$ is subMarkovian, $$N_\kappa(1) \leq Q_m(f_1) = M^1(1)=M(1)<\infty,$$ that is, $Q_m[\kappa(\R)] < \infty$. Hence $\kappa$ is a random kernel. Since $M$ and $N_\kappa$ are finite measures and agree on $F \in p(\B \times \G^o)$ of the form $F = \varphi \otimes f, M=N_\kappa$ on $\B \times \G^o$ and hence on $\M$ since they both vanish on $\I$. Therefore $\kappa$ is carried by $\Lambda^*$ since $M$ is carried by $\Lambda^*$. Thus finally $\kappa$ is a RM and from (5.4), $M = M_\kappa$. 

If $M$ is $\sigma$-finite on $\M$, then $\R \times W = \cup_{n=1}^\infty \Lambda_n$ where each $\Lambda_n \in \M$ with $M(\Lambda_n) < \infty$ and the $\Lambda_n$ are disjoint. Set $M_n = 1_{\Lambda_n} M$. Then by what has been proved there exist RMs, $\kappa_n$, with $Q_m(\kappa_n(\R)) < \infty$, $\kappa_n$ carried by $\Lambda_n$, and $M_n = M_{\kappa_n}$. Thus $\kappa := \sum_{n=1}^\infty \kappa_n$ is a RM and since the $\Lambda_n$ are disjoint $M = M_\kappa$ and $\kappa \in \sigma\I(\M)$ because $M$ is $\sigma$-finite. It remains to prove the uniqueness. Suppose $\kappa_i \in \sigma\I(\M), i=1,2$. Then there exist $Z^i \in \M, Z^i > 0$ with $Q_m \int Z_t^i \kappa(dt) < \infty, i=1,2$. Then for $Z := Z^1 \wedge Z^2$ one has $Q_m \int Z_t \kappa_i(dt) < \infty$ for $i=1,2$ and $Z > 0$. Replacing $\kappa_i$ by $Z * \kappa_i, i=1,2$, it suffices to prove the uniqueness when $M_{\kappa_i}$ is finite, $i=1,2$. If $B \in \B$ and $f \in p\G^o$, then from (5.4), $Q_m[f \kappa_i(\cdot, B)]$ does not depend on $i$. It follows from this that $\kappa_1(\cdot, B) = \kappa_2(\cdot, B)$ a.e. $Q_m$ and as remarked earlier this implies that $\kappa_1 = \kappa_2$. This completes the proof of (5.3).
 \end{proof}
 
 \medskip\noindent
 {\bf Remark } The RM actually constructed in the proof of (5.3) is a kernel from $(W, \G^0)$ to $(\R, \B)$. In particular a RM in $\sigma(\mathcal{I}\mathcal{M})$ is indistinguishable from a RM which is also a kernel from $(W, \mathcal{G}^0)$ to $(\mathbb{R},\mathcal{B})$.
 
\medskip\noindent
{\bf (5.5) Definition}
{\it A random measure $\kappa$ is optional (resp. copredictable) provided 
$\kappa \in \sigma\I(\O^0)$ and 
$Q_m \int Z_t \kappa(dt) = Q_m \int {^0 Z_t} \kappa(dt)$ 
(resp. $\kappa \in \sigma\I({\wP}^0)$ and 
$Q_m \int Z_t \kappa(dt) = Q_m \int{^{\hat{p}} Z_t} \kappa(dt)$) for $Z \in p(\B \times \G^0)$.}

\medskip\noindent
Note that $\sigma\I(\O^0) = \sigma\I(\O^m), \sigma\I(\wP^0) = \sigma\I(\wP^m)$ and that the equalities in (5.5) actually hold for $Z \in p\M$. It is convenient to introduce the notation
\begin{equation}\tag{5.6}
\langle Z, \kappa \rangle := Q_m \int Z_t \kappa(dt) = M_\kappa(Z)
\end{equation}
when $\kappa$ is a random measure and $Z \in p\M$. Using this notation $\kappa \in \sigma\I(\wP^m)$ is copredictable provided $\langle Z, \kappa \rangle = \langle {^{\wp}Z}, \kappa \rangle$. Suppose that $\kappa \in \sigma \I (\wP^m)$. If $Z \in \I$ then ${^{\wp}Z} \in \I$ by the uniqueness of the coprojection. Consequently the measure $M(Z) := \langle {^{\wp}Z}, \kappa \rangle$ defined on $p\M$ satisfies the hypotheses of (5.3) and so there exists a unique RM, $\kappa^{\wp}$, with $\langle {^{\wp}{Z}}, \kappa \rangle = \langle Z, \kappa^{\wp} \rangle$ for $Z \in p\M$. But $^{\wp}(^{\wp}Z) = {^{\wp}Z}$ and so for $Z \in p\M$, 
$$\langle Z, \kappa^{\wp} \rangle = \langle {^{\wp}Z}, \kappa \rangle = \langle {^{\wp}({^{\wp}Z})}, \kappa \rangle = \langle {^{\wp}Z}, \kappa^{\wp} \rangle = \langle Z, (\kappa^{\wp})^{\wp} \rangle.$$
The equality of the first and fourth term in this display implies that $\kappa^{\wp}$ is copredictable and the equality of the first and last that $\kappa^{\wp} = (\kappa^{\wp})^{\wp}$. $\kappa^{\wp}$ is called the {\it dual copredictable projection} of $\kappa$. Similarly the {\it dual optional projection}. $\kappa^0$, is defined by $\langle Z, \kappa^0 \rangle = \langle {^{0}{Z}}, \kappa \rangle$ for $\kappa \in \sigma\I(\O^m)$. The fact that ${^0Z}$ is defined only on $\Lambda^*$ causes no difficulty since $\kappa$ is carried by $\Lambda^*$

It is convenient to define the ``big shifts", $\Sigma_t$, for $t \in \R$. Their action on processes and random measures is defined by \begin{equation}\tag{5.7}\begin{split}
&(i)\quad (\Sigma_t Z) (s, w) := Z(s-t, \sigma_t w), Z \in \M,\\
&(ii)\quad (\Sigma_t \kappa) (w, B) := \kappa(\sigma_t w, B-t), \kappa \hbox{ a RM.} 
\end{split}\end{equation}
It is clear that $\Sigma_t Z \in \M$ and that $\Sigma_t \kappa$ is a random measure. Note that $Z \in \M$ is homogeneous as defined above (3.25) if and only if $\Sigma_t Z = Z$ for $t \in \R$. Of course, equality in $\M$ is defined modulo $\I$ as it is for random measures. If $Z \in p\M$ and $\kappa$ is a random measure, then \begin{equation}\tag{5.8}
\langle Z, \Sigma_t \kappa \rangle = \langle \Sigma_{-t} Z, \kappa \rangle, t \in \R.
\end{equation}
The stationarity of $Q_m$ implies that $\Sigma_t$ preserves the classes $\I, \O^0, \O^m, \wP^0, \wP^m$ and that $\Sigma_t({}^{\wp}Z) = {}^{\wp}(\Sigma_t Z)$ and $\Sigma_t({}^0Z) = {}^0(\Sigma_t Z)$. Therefore it follows from (5.8) that if $\H \subset \M$ is invariant under $\Sigma_t$ for all $t \in \R$, then so is $\sigma \I (\H)$. A RM, $\kappa$, is {\it homogenous} and we write $\kappa$ is a HRM provided $\Sigma_t \kappa = \kappa$ for all $t \in \R$. Note that because $\B$ is countably generated, $\kappa \in \sigma \mathcal{I}(\mathcal{M})$ is homogeneous if and only if for each $t \in \R$ and $B \in \B$, $\kappa ( \sigma_t, B - t ) = \kappa ( B )$ a.e. $Q_m$ where the exceptional set is allowed to depend on both $t$ and $B$. 
If $Z \in p\M$ and $\kappa$ is a RM, define
$Z * \kappa(w,B) := \int_{B} Z_t(w) \kappa(w, dt))$
and note that $Z * \kappa$ is a RM provided $Z * \kappa$ satisfies the appropriate finiteness condition.
If, in addition, $Z$ is homogeneous, and $\kappa$ is a HRM, then $Z * \kappa$ is a HRM.
The following proposition is an immediate consequence of the definitions, the above discussion and the fact that ${}^{0 \wp}Z = {}^{\wp 0}Z$ for $Z \in p\M$. See (4.24).

\medskip\noindent
{\bf (5.9) Proposition}
{\it \begin{enumerate}[(a)]
\item 
If $\kappa \in \sigma \I (\wP^m)$ (resp. $\sigma \I (\O^m)$) , then $\Sigma_t \kappa \in \sigma \I (\wP^m)$ (resp. $\sigma \I (\O^m)$) and $(\Sigma_t \kappa)^{\wp} = \Sigma_t (\kappa^{\wp})$ (resp. $(\Sigma_t \kappa)^0 = \Sigma_t(\kappa^0)$) for all $t \in \R$. 
\item 
A RM $\kappa \in \sigma\I(\M)$ is homogeneous if and only if $M_\kappa(\Sigma_tZ) = M_\kappa(Z)$ for $Z \in p\M$ and $t \in \R$. 
\item 
If $\kappa$ is a RM in $\sigma\I(\wP^m) \cap \sigma\I(\O^m)$, then $\kappa^{0\wp} = \kappa^{\wp 0}$.
\end{enumerate}
}

\medskip\noindent
{\bf Remark}
If $f \in p\E^*$ and $\kappa$ is a HRM, then $(f \circ Y_t^*) * \kappa$ is a HRM provided it satisfies the appropriate finiteness condition. In particular, this is the case of $f$ is bounded.

A measure $\mu$ on a measurable space $(A, \A)$ is {\it $s$-finite} (sometimes called $\Sigma$-finite) provided it is a countable sum of finite measures.

\medskip\noindent
{\bf (5.10) Proposition } 
{\it Let $\kappa$ be a HRM. Then 
$$ J(F) := Q_m \int F(t, Y_t^*) \kappa(dt), F \in p(\B \times \E) $$
defines an $s$-finite measure on $\B \times \E$. There exists a unique $s$-finite measure $\mu_\kappa$ on $(E, \E)$ such that $J(F) = \int_\R \int_E F(t, x) \mu_\kappa(dx) dt$ for $F \in p(\B \times \E)$ and $\mu_\kappa$ is called the characteristic measure of $\kappa$. If $\varphi \in p\B$ with $\int_\R \varphi(t) dt = 1$, then 
\begin{equation}\tag{5.11}
\mu_\kappa(f) = Q_m \int_\R \varphi(t) f \circ Y_t^* \kappa(dt).
\end{equation}
Also $\mu_\kappa$ is $\sigma$-finite if and only if $J$ is $\sigma$-finite.}
\begin{proof}
Since $\kappa$ is a countable sum of subMarkov kernels and 
$Q_m$ is $\sigma$-finite it is easy to check that $J$ is $s$-finite. Given $F \in p(\B \times \E)$ let $F_s(t, x) := F(t+s, x)$. Using the stationarity of $Q_m$ and the homogeneity of $\kappa$, 

\begin{equation*}
\begin{split} 
J(F) &= Q_m \Bigg( \int F(t, Y_t^*) \kappa(dt)  \circ \sigma_{- s}\Bigg)\\
       &= Q_m \int F(t+s, Y_t^*) \kappa(dt) = J(F_s)
\end{split}
\end{equation*}
for $s \in \R$. Thus $J$ is an $s$-finite measure on $(\B \times \E)$ that is {\it translation invariant in its first coordinate} as defined in [G; 8.23]. Consequently from [G; 8.23] or [G87] there exists a unique $s$-finite measure, $\mu_\kappa$ on $(E, \E)$ such that 
$J(F) = 
\int \int 
F(t,x) dt \mu_\kappa(dx)$. If $F(t,x) = \varphi(t) f(x)$ with $\varphi \in p\B, f \in p\E$ then $J(f \varphi) = \int \varphi(t) dt \cdot \int f d\mu_\kappa$. Thus if $\int \varphi = 1$ we obtain $\mu_\kappa(f) = Q_m \int \varphi(t) f \circ Y_t^* \kappa(dt)$. It is clear that $\mu_\kappa$ $\sigma$-finite implies that $J$ is $\sigma$-finite. Conversely if $J$ is $\sigma$-finite there exists $F \in p(\B \times \E)$ with $F > 0$ and $J(F) < \infty$. Then $f(x) = \int_\R F(t,x) dt > 0$ and $\mu_\kappa(f) < \infty$, completing the proof of (5.10).
\end{proof}

The characteristic measure $\mu_\kappa$ is sometimes called the Revuz measure of $\kappa$. The next proposition contains several important properties of $\mu_\kappa$ under various assumptions on $\kappa$. 

\medskip\noindent
{\bf (5.12) Proposition} {\it
\begin{enumerate}[(i)]
\item 
Let $\kappa$ be an optional HRM. Then $\kappa \in \sigma\I(\wP^m)$ if and only if $\mu_\kappa$ is $\sigma$-finite. In this case there exists $f \in \E, \varphi \in \B$ with $\mu_\kappa(f) < \infty, 
\int_{\R} \varphi(t) dt < \infty$ such that
$F := \varphi \otimes f > 0$ on $\R \times E$ and $Q_m \int F(t, Y_t^*) \kappa(dt) < \infty$. 
\item 
If $\kappa_1$ and $\kappa_2$ are optional copredictable HRMs then $\kappa_1 = \kappa_2$ if and only if $\mu_{\kappa_1} = \mu_{\kappa_2}$ is a $\sigma$-finite measure. 
\end{enumerate}
}

\begin{proof}
(\textit{i}) If $\mu_\kappa$ is $\sigma$-finite, there exists $f \in \E$ with $f > 0$ and $\mu_\kappa(f) < \infty$.
Let $\varphi \in \B$ with $\varphi > 0$ and $\int_{\R} \varphi(t) dt < \infty$. Then $Z_t := \varphi(t) f \circ Y_t^* > 0$ on $\Lambda^*$ which carries $\kappa$ by definition, $Z$ is copredictable and $Q_m \int Z_t \kappa(dt) = \mu_\kappa(f) \cdot \int \varphi < \infty$. Hence $\kappa \in \sigma \I(\wP^0)$ and one may take $F(t,x) = \varphi(t) f(x)$. 
Conversely suppose $\kappa \in \sigma \I (\wP^m) = \sigma \I (\wP^0)$. Then there exists $Z \in \wP^0$ with $Z>0$ and $Q_m \int Z_t \kappa(dt) < \infty$. Since $\kappa$ is optional, $Q_m \int {}^{0}Z_t \kappa(dt) = Q_m \int Z_t \kappa(dt) < \infty$. From (3.21) and its proof, ${}^{0}Z_t(w) = 1_E(Y_t^*) \int_{\Omega} Z_t(w | t | \omega) P^{Y^*(t, w)} (d \omega)$ where $(w | t | \omega)(s) = w(s)$ if $s<t$ and $(w | t | \omega)(s) = \omega^+(s-t)$ if $s \geq t$. Now $Z \in \wP^0$ and so $Z_t(w) = Z_t(b_t w)$ by (3.5\textit{i}) and $b_t ( w | t | \omega ) = \omega ( s-t )$ if $s > t$ and equals $\Delta$ if $s \leq t$. Hence $Z_t( w | t | \omega)$ does not depend on 
$w$.
Define $F(t, x) := \int Z_t (w | t | \omega) P^x (d \omega)$ if $x \in E$ and $F(t,  \Delta)=0$. Then arguing as in the proof of (3.21), $F \in \B \times \E$. Clearly ${}^{0}Z_t = F(t, Y_t^*)$ on $\{ Y_t^* \in E \}$ and $F > 0$ on $\R \times E$ since $Z > 0$. Let $f(x) := \int F(t, x) dt > 0$ on $E$. Then $\mu_\kappa(f) = Q_m \int F(t, Y_t^*) \kappa(dt) < \infty$. Hence 
$\mu_\kappa$ is $\sigma$-finite.

(\textit{ii}) If $\kappa_1 = \kappa_2$ it is obvious from (5.11) and (i) that
$\mu_{\kappa_1} = \mu_{\kappa_2}$ is a $\sigma$-finite measure. 
On the other hand if
$\kappa_1$ and $\kappa_2$ are optional HRMs with $\mu_{\kappa_1}$ and
$\mu_{\kappa_2}$ $\sigma$-finite, then it follows from (\textit{i})
that there exists $F \in \B \times \E$ with $0 < F \leq 1$ and $Q_m
\int F(t, Y_t^*) \kappa_i (dt) < \infty, i=1,2$. Let $Z \in p(\B
\times \G^0)$. 
Then from (4.25), 
$^{o\hat{p}} Z = \ ^{\hat{p}o}Z \in p(\O^o \cap \wP^o)$.
Hence (3.25) implies that there exists 
$ G \in p(\B \times \E)$
so that 
$^{o\hat{p}} Z = G(t, Y_t^*)$.  Since
$t \rightarrow F(t, Y_t^*)$ is in
$p(\O^o \cap \wP^o)$ and since the $\kappa_i$ are optional and copredictable, we find that
\begin{equation*}
\begin{split}
Q_m \int Z_t F(t, Y_t^*) \kappa_i (dt) &= Q_m \int G(t, Y_t^*) F(t, Y_t^*) \kappa_i(dt)\\
&= \mu_{\kappa_i} (h),
\end{split}
\end{equation*}
where $h := \int G(t, \cdot) F(t, \cdot) dt$. Consequently
\begin{equation}\tag{5.13}
Q_m \int Z_t F(t, Y_t^*) \kappa_1 (dt) = Q_m \int Z_t F(t, Y_t^*) \kappa_2 (dt).
\end{equation}
Since $F$ is bounded $\widetilde{\kappa}_i (dt) := F(t, Y_t^*) \kappa_i (dt)$ are RMs with $Q_m (\widetilde{\kappa}_i(\R)) < \infty$. Use (5.13) with $Z = 1_{B} f$ where $B \subset \R$ is a bounded Borel set and $f \in \G^0$ with $0 < f \leq 1$ and $Q_m (f) < \infty$ to obtain $Q_m ( f \widetilde{\kappa}_1 (\cdot, B)) = Q_m ( f \widetilde{\kappa}_2 (\cdot, B)) < \infty$. But this implies that $\widetilde{\kappa}_1( \cdot, B) = \widetilde{\kappa}_2( \cdot, B)$, $Q_m$ a.s. for each fixed $B \in \B$. As remarked earlier this, in turn, implies that $\widetilde{\kappa}_1 = \widetilde{\kappa}_2$ because $\B$ is countably generated by bounded Borel sets. Finally since $F > 0$ we obtain $\kappa_1 = \kappa_2$.
%
\end{proof} 

Before coming to Fitzsimmons' basic existence theorem for HRMs we need a definition and a lemma.

\medskip\noindent 
{\bf (5.14) Definition}
{\it A set $A \in \G^0$ is $m$-$\theta$-evanescent (resp. $m$-$\theta$-polar) provided $t \to 1_A \circ \theta_t 1_E (Y_t^*)$ (resp. $t \to 1_A \circ \theta_t 1_{]\alpha,\beta[} (t)$) is $Q_m$-evanescent.}

\medskip\noindent
This definition of $m$-$\theta$-evanescence differs slightly from that in [Fi87]. Since $Y_t^* \in E$ and $t \notin ]\alpha, \beta[$ can happen only if $\alpha = t$, it is clear that $A$ is $m$-$\theta$-evanescent if and only if it is $m$-$\theta$-polar and $Q_m ( 1_A \circ \theta_{\alpha}; Y_{\alpha}^* \in E) = 0$. Recall the definition of $\N(m)$ above (2.14).

\medskip\noindent 
{\bf (5.15) Lemma}
{\it Let $\mu$ be a measure on $E$. If $\mu$ doesn't charge sets in $\N(m) \cap \E$, then $P^{\mu}$ doesn't charge $m$-$\theta$-evanescent sets.}
\begin{proof}
Let $A \in \G^0$ be $m$-$\theta$-evanescent and $\mu$ a measure not charging Borel $m$-inessential sets. Using (2.13), more precisely (2.17), for the second equality and $W(b) \subset \{ Y_{\alpha}^* \in E \}$ for the inequality we have
\begin{equation*}
\begin{split}
0 &= Q_m ( 1_A \circ \theta_{\alpha}; Y_{\alpha}^* \in E ) = Q_m ( P^{ Y_{\alpha}^*}(A); Y_{\alpha}^* \in E)\\
&\geq Q_m ( P^{ Y_{\alpha}^*} (A); W(b) ).
\end{split}
\end{equation*}
Let $\rho U$ be the potential part of $m$. [G; 6.20] in the current notation states that $\rho(dx) = Q_m ( Y_{\alpha}^* \in dx, 0 < \alpha < 1, W(b) )$. Consequently from the above display $\int P^x (A) \rho (dx) = 0$. Let $B = \{ x : P^x(A) > 0\}$. Then $B \in \E, \rho(B)=0$ and so to show that $B \in \N(m)$ we must show that $B$ is $m$-polar. For this it suffices to show that $B_{\varepsilon} := \{ x : P^x(A) > \varepsilon \}$ is $m$-polar for $\varepsilon > 0$. Since $\{ (t, w) : P^{ Y(t,w) } (A) > \varepsilon \} \in \O^0$, by the section theorem there exists a $(\G_t^m)$ stopping time $T$ such that $P^{ Y(T) } (A) > \varepsilon$ on $\{ \alpha < T < \beta \}$ and $Q_m (\alpha < T < \beta) > 0$. 
Therefore
\begin{equation*}
\begin{split}
Q_m ( 1_A \circ \theta_T ; \alpha < T < \beta ) &= Q_m ( P^{Y(T)}(A) ; \alpha < T < \beta)\\
&\geq \varepsilon P^{Y(T)} (\alpha < T < \beta) > 0,
\end{split}
\end{equation*}
and this contradicts the $m$-$\theta$-polarity of $A$. Therefore $B_{\varepsilon}$, and hence $B$, is $m$-polar. Thus $B \in \N(m)$ and so $\mu(B)=0$. As a result $P^{\mu}(A) = \int \mu(dx) P^x(A) = 0$; that is $P^{\mu}$ does not charge the arbitrary $m$-$\theta$-evanescent set $A$.
\end{proof} 

It is clear that if $\kappa$ is a HRM, then its characteristic measure, $\mu_\kappa$, does not charge sets in $\N(m)$. Fitzsimmons' basic existence theorem for HRMs is the converse which we now state and prove. (Actually this is a corollary to the basic 
existence theorem in [Fi87].)

\medskip\noindent
{\bf (5.16) Theorem} 
{\it Let $\mu$ be a $\sigma$-finite measure on $E$. Then $\mu$ is the characteristic measure of a unique optional copredictable HRM, $\kappa$, if and only if $\mu$ does not charge $m$-exceptional sets.}

\begin{proof}
Let $\mu$ be a $\sigma$-finite measure on $E$ not charging $m$-exceptional sets. Recall that $\I = \I^m$ is the ideal of $Q_m$-evanescent processes. Given $Z \in p\M$ in view of (3.25) and (4.25) there exists $f_Z \in p(\B \times \E)$ such that
\begin{equation}\tag{5.17}
{}^{0 \wp}Z_t = {}^{\wp 0}Z_t = f_Z(t, Y_t^*) \text{ on } \Lambda^*,
\end{equation}
where as before, equality in $\M$ is modulo $\I$. 
Define
\begin{equation}\tag{5.18}
M(Z) := P^{\mu} \int_{\R} f_Z( t, Y_0^* ) dt, Z \in p\M.
\end{equation}
We claim (5.18) defines $M(Z)$ uniquely; that is, it does not depend on the particular choice of $f_Z$. For this it suffices to show that $M(Z) = 0$ if $Z \in pb\I$. To this end let $Z \in pb\I$. Then we may suppose that $f_Z$ is bounded. Let $\varphi \in p\B$ be a function on $\R$ with compact support and define 
$$ H := \int_{\R} \varphi( t ) f_Z( t, Y_0^* ) dt < \infty. $$
Observe that 
\begin{equation}\tag{5.19}
H \circ \theta_s = \int_{\R} \varphi( t ) f_Z( t, Y_s^* ) dt
                 = \int_{\R} \varphi( s - u ) f_Z( s - u, Y_s^* ) du \text{ on } \Lambda^*,
\end{equation}
and that $f_Z( s-u, Y_s^* ) = f_Z( s-u, Y_{s-u}^* ) \circ \sigma_u$ for each $u \in \R$. Since (5.17) holds modulo $\I$, there exists a set $N \in \G^0$ with $Q_m(N)=0$ such that for all
 $t \in \R$ and $w \notin N, f_Z(t, Y_t^*(w)) = {}^{0 \wp}Z_t(w) = {}^{\wp 0}Z_t(w)=0$ because the projections preserve $\I$. Consequently for each fixed $u \in \R, f_Z(s-u, Y_s^*(w))=0$ for all $s \in \R$ and $w \notin \sigma_u^{-1}(N)$, and of course, $Q_m(\sigma_u^{-1}(N))=0$. Thus for each fixed $s$ and $u$, $f_Z(s-u, Y_s^*)=0$ a.s. $Q_m$, and so by Fubini's theorem for $Q_m$ a.e. $w$ there exists a Lebesgue null set $N_w \in \B$ such that $f_Z(s-u, Y_s^*(w))=0$ for $s \in \R$ and $u \notin N_w$. Hence from (5.19), $Q_m$ a.s.
$$ H \circ \theta_s = \int_{\R} \varphi(s-u) f_Z(s-u, Y_s^*) du = 0$$
for each $s \in \R$. That is, $H = \int \varphi(t) f_Z(t, Y_0^*) dt$ is
$m$-$\theta$-evanescent. Hence from (5.15) and the hypothesis $P^{\mu}(H)=0$. Let $\varphi \uparrow 1$ through a sequence of $\varphi \in p\B$ with compact support to conclude that $P^{\mu} \int_{\R} f_Z(t, Y_0^*)dt=0$. Consequently (5.18) uniquely determines $M(Z)$. It is now clear that $M$ is a measure on $\M$ not charging $Q_m$-evanescent sets. Let $\varphi>0$ on $\R$ and $g>0$ on $E$ with $\int_{\R} \varphi < \infty$ 
and $\int_E g d\mu < \infty$. If $Z_t := \varphi(t) g(Y_t^*)$ then $F_Z = \varphi \otimes g > 0$ and $M(Z) = \int \varphi \cdot \int g d\mu < \infty$ since $Y_0^* = X_0$ a.s. $P^{\mu}$ on $W(b) \cap \Omega$. Thus $M$ is $\sigma$-finite and now it follows from (5.3) that there is a unique RM, $\kappa$, such that $M(Z) = Q_m \int Z_t \kappa(dt), Z \in p\M$. Since $f_{{}^{0\wp}Z} = f_{{}^{\wp0}Z} = f_Z$ by construction $\kappa$ is optional and copredictable. Finally ${}^{0\wp}(\Sigma_{-s}Z_t) = {}^{0\wp}(Z_{t+s}\circ\sigma_{-s}) = f_Z(t+s,Y_{t+s}^*)\circ\sigma_{-s} = f_Z(t+s,Y_t^*)$, and so $M(\Sigma_{-s}Z) = P^{\mu} \int f_Z(t+s, Y_0^*)dt = M(Z)$. Hence $\kappa$ is an optional copredictable HRM and $\mu_\kappa = \mu$. The 
uniqueness assertion follows from (5.12\textit{ii}).
\end{proof}

In the remainder of this section we shall show that the $\kappa$ constructed in Theorem 5.16 may be chosen to have additional measurability and algebraic properties. We begin with a simple lemma that is useful in checking optionality or copredictability of a RM.

\medskip\noindent
{\bf (5.20) Lemma }
{\it Let $\kappa$ be a random measure.
\begin{enumerate}[(a)]
\item If $\kappa$ is optional (resp. copredictable) and $Z \in p\O^m$ (resp. $p\wP^m$), then $Z * \kappa$ is optional (resp copredictable).

\item If there exists $Z > 0, Z \in p\O^m$ (resp. $p\wP^m$) with $\langle Z, \kappa \rangle < \infty$ and $Z * \kappa$ optional (resp. copredictable), 
then $\kappa$ is optional (resp. copredictable).
\end{enumerate}
} 
\begin{proof}
\begin{enumerate}[(a)]
\item Suppose $\kappa$ is optional and $Z \in p\O^0$. If $U \in p\M$, then ${}^{0}(UZ) = {}^{0}U Z$ and so $\langle U, Z * \kappa \rangle = \langle UZ, \kappa \rangle = \langle {}^{0}UZ, \kappa \rangle = \langle {}^{0}U, Z * \kappa \rangle$. Hence $Z * \kappa$ is optional. If $U_1, U_2 \in p\M^m$, then ${}^{\wp}(U_1U_2) = {}^{\wp}U_1 {}^{\wp}U_2$ is clear from (3.17). Hence the copredictable case is established in exactly the same manner.

\item Let $Z$ be as in the statement, say in the optional case. If $U \in p\M$, then 
$$ \langle U, \kappa \rangle = \langle UZ^{-1}, Z * \kappa \rangle
	= \langle {}^{0}U Z^{-1}, Z * \kappa \rangle
	= \langle {}^{0}U, \kappa \rangle $$
and so $\kappa$ is optional. Again the copredictable case is similar. 
\end{enumerate}
\end{proof}
\medskip\noindent
{\bf (5.21) Lemma }{\it Let $\kappa$ be an optional (resp. copredictable) random measure. Then $\kappa = \sum_{n \geq 1} \kappa_n$ where for each $n$, $\kappa_n$ is an optional (resp. copredictable) random measure with $Q_m(\kappa_n(\R)) < \infty$ and the $\kappa_n$ are carried by disjoint sets in $\O^0$ (resp. $\wP^0$). }
\begin{proof}
Since $\kappa \in \sigma \I( \O^0 )$ (resp. $\sigma \I( \wP^0 )$), there exists $Z \in \O^0$ (resp. $\wP^0$) with $0 < Z \leq 1$ and $Q_m \int Z_t \kappa(dt) < \infty$. Let $\Lambda_n := \{ (t,w) : (n+1)^{-1} < Z_t(w) \leq n^{-1} \}$. Then $\Lambda_n \in \O^0$ (resp. $\wP^0$) and the $\Lambda_n$ are disjoint. Define $\kappa_n := 1_{\Lambda_{n}} * \kappa$. Then (5.20) states that $\kappa_n$ is an optional (resp. copredictable) random measure. But $1_{\Lambda_n} \leq (n+1)Z$ and so $Q_m( \kappa_n(\R)) \leq (n+1) \langle Z, \kappa \rangle < \infty$.
\end{proof}
\medskip\noindent
{\bf (5.22) Proposition}{\it Let $\kappa$ be a RM. Then $\kappa$ is optional (resp. copredictable) if and only if there exists $Z > 0, Z \in \O^m$ (resp. $\wP^m$) with $\langle Z, \kappa \rangle < \infty$ such that
\begin{equation}\tag{5.23}
t \to \int_{]-\infty,t]} Z_s \kappa(ds) \in \O^m
\text{ (resp. } t \to \int_{[t, \infty[} Z_s \kappa(ds) \in \wP^m \text{ )}.
\end{equation}
Then (5.23) holds for any $Z \in p\O^m$ (resp. $p\wP^m$) with $\langle Z, \kappa \rangle < \infty$. }
\begin{proof}
In view of (5.20) we must show that the above conditions are equivalent to $Z * \kappa$ being optional (resp. copredictable). In either case $Q_m( (Z*\kappa)(\R)) < \infty$ and so it suffices to show that if $\kappa$ is a RM with $Q_m(\kappa(\R)) < \infty$, then $\kappa$ is optional (resp. copredictable) if and only if $t \to \kappa(]-\infty,t]) \in \O^m$ (resp. $t \to \kappa([t,\infty[) \in \wP^m$). We shall write the proof in the copredictable case, the optional case is similar and somewhat 
simpler.

Therefore suppose $Q_m(\kappa(\R)) < \infty$ and define $A_t := \kappa([t, \infty[)$. Then a.s. $Q_m, A_t < \infty$ for $t > -\infty$, $t \to A_t$ is decreasing, left continuous and 
$ A_t \rightarrow 0$ as $t \rightarrow \infty$.
. Set $\wA_t = A_{-t} \circ r^{-1}$ where $r := W \to \widehat{W}$ is the reversal operator defined in section 3. Then $\wA$ is increasing, right continuous and $\wA_{-\infty} = \lim_{t \to \infty} A_t \circ r^{-1} = 0$. Moreover $\wQ_m(\wA_{\infty}) = \wQ_m(\kappa(\R)) < \infty$ where $\wQ_m = r Q_m$ is defined above (3.13). Therefore $\wA$ is a raw ({\it i.e.} not necessarily adapted) increasing integrable process on $[-\infty, \infty[$ as defined in [DM VI; 51] relative to the system $(\wW, \wG^m, \wG_t^m, \wQ^m)$ defined above (3.13) except that the parameter set is $[-\infty, \infty[$ rather than $[0,\infty[$. Note that $d\wA_t = \kappa(-dt) \circ r^{-1}$. If $Z \in p(\B \times \G^0)$, 
define $\wZ_t = Z_{-t} \circ r^{-1}$ and recall from the proof of (3.15) that ${}^{\wp}Z_t = {}^{p}\wZ_{-t} \circ r$ where of course ${}^{p}\wZ$ is the predictable projection of $\wZ$ relative to $(\wW, \wG^m, \wG_t^m, \wQ^m)$. Observe that 
\begin{equation*}
\wQ_m \int \wZ_t d\wA_t = Q_m \int Z_{-t} \kappa(-dt) = Q_m \int Z_t \kappa(dt).
\end{equation*}
From [DM V; 57, 59], $\wA$ is $(\wG_t^m)$ predictable if and only if $\wQ_m \int \wZ_t d \wA_t = \wQ_m \int {}^{p}\wZ_t d \wA_t$ for $\wZ \in p(\B \times \wG^0)$. Combining this with the above remarks and the definition (5.5), it follows that $\kappa$ is copredictable if and only if $t \to \wA_t = A_{-t} \circ r^{-1}$
is $(\wG_t^m)$ predictable. Since $A$ vanishes on $\llbracket \beta, \infty \llbracket, A_t = \wA_{-t} \circ r \cdot 1_{]-\infty, \beta[}(t)$ is in $p\wP^m$ if and only if $\wA$ is $(\wG_t^m)$ predictable according to (3.3). Consequently $\kappa$ is copredictable (as a RM) is equivalent to $t \to \kappa([t, \infty[)$ being $Q_m$-copredictable (as a process).
\end{proof}

\medskip\noindent
{\bf (5.24) Corollary }
{\it If $\kappa$ is optional (resp. copredictable), then $\kappa(B \cap ]-\infty,t]) \in \G_t^m$ (resp. $\kappa(B \cap [t, \infty[) \in \G_{>t}^m$) for $B \in \B$ and $t \in \R$.}
\begin{proof}
Suppose first that $Q_m(\kappa(\R)) < \infty$. Then $\kappa(w, \cdot)$ is a finite measure on $(\R, \B)$ for $Q_m$ a.e. $w$. Then (5.23) implies that
$\kappa(]-\infty, s]) \in \G_t^m$ for $s \leq t \in \R$ and $\kappa([s, \infty[) \in \G_{>t}^m$ for $t \leq s \in \R$. Clearly this establishes (5.24) when $Q_m(\kappa(\R)) < \infty$. The general case now follows from (5.21).
\end{proof}

If $\kappa$ is a RM, then for each $w$ let $\kappa^d(w, \cdot)$ be the discrete part and $\kappa^c(w,\cdot)$ be the diffuse part of $\kappa(w,\cdot)$. This decomposition takes an especially nice form when $\kappa$ is an optional copredictable HRM. In the remainder of this section it will be
convenient to eliminate parentheses when writing random measures if no confusion is possible. For example, we shall write $\kappa ]a,b]$ in place of $\kappa (]a,b])$ or $\kappa[a]$ for $\kappa(\{a\})$, etc. In addition we shall abbreviate optional copredictable HRM by OCHRM.

\medskip\noindent
{\bf (5.25) Proposition }
{\it Let $\kappa$ be an} OCHRM. {\it Then $\kappa = \kappa^d + \kappa^c$ where $\kappa^c$ is a diffuse} OCHRM {\it and 
$$ \kappa^d = \sum_{t \in \R} j \circ Y_t^* \varepsilon_t $$
where $j \in p\E$ and $\{j>0\}$ is $m$-semipolar.}
\begin{proof}
Since $\k(w, \cdot)$ is a $\sigma$-finite measure on $(\R, \B)$ for each $w \in W, \k[t]>0$ for at most countably many $t$ depending on $w$. Now suppose that $Q_m(\k(\R)) < \infty$. Then 
$$\k[t] = \k ] -\infty, t] - \k ] -\infty, t[ = \k[ t, \infty[ - \k ] t, \infty [.$$ But $t \to \k ] -\infty, t] \in \O^m$ and $t \to \k [ t, \infty [ \in \wP^m$ by (5.22). On the other hand $t \to \k ] -\infty, t [$ is left continuous and adapted to $(\G_{t-}^m)$, hence $(\G_t^m)$ predictable and so $(\G_t^m)$ optional. Since it vanishes on $]-\infty,\alpha[$, it is in $\O^m$. Similarly $t \to \k ] t, \infty [$ is right continuous and adapted to $(G_{>t}^m)$, hence in $\wP^m$. Consequently $t \to \k[t] \in \O^m \cap \wP^m$. In view of (5.21), this then holds for any optional copredictable RM. Since $\kappa$ is homogeneous for each $s \in \R$, $\k[t] \circ \sigma_s = \k[t+s]$ for all $t, Q_m$ a.s. Now (3.25) implies that there exists $j \in p\E$ such that $\k[t]$ and $j \circ Y_t^*$ are $Q_m$-indistinguishable. Therefore a.s. $Q_m, j \circ Y_t^* > 0$ for at most countably 
many $t$, and so $\{j>0\}$ is $m$-semipolar. Define $\k^d := \sum_{t \in \R} j \circ Y_t^* \varepsilon_t$ and $\k^c = \k - \k^d$. Clearly $\k^d$ is a discrete RM and $\k^d$ is homogeneous. Also $\k^d$ is carried by $\{j>0\}$ and $\k^c$ does not charge $\{j>0\}$. Therefore, letting $h := 1_{\{j>0\}}, \k^d = h \circ Y^* * \k$. If $Z \in p\M$, then
\begin{equation*}\begin{split}
Q_m \int Z_t \k^d(dt) 
&= Q_m \int Z_t h \circ Y_t^* \k(dt)\\
&= Q_m \int {}^{0 \wp}Z_t h \circ Y_t^* \k(dt)\\
&= Q_m \int {}^{0 \wp}Z_t \k^d(dt)\\
\end{split}\end{equation*}
where the second inequality follows since $h \circ Y^* \in p(\O^0 \cap \wP^0)$. Consequently $\k^d$ is a discrete OCHRM and $\k^c$ is a diffuse OCHRM.
\end{proof}

Clearly a diffuse RM measure does not charge $m$-semipolar sets. Consequently the following is an immediate corollary of (5.25) and Theorem 5.16.

\medskip\noindent
{\bf (5.26) Corollary }
{\it A $\sigma$-finite measure $\mu$ on $E$ is the characteristic measure of a diffuse} OCHRM {\it if and only if $\mu$ does not charge $m$-semipolar sets.}

So far we have identified random measures that are $Q_m$-indistinguishable. However it is important in many applications to be able to choose an especially nice representative from a given equivalence class of random measures. Our last theorem of this section states that this is possible for optional copredictable homogeneous random measures. Its proof depends on the following result which is a version for additive functionals of Meyer's master perfection theorem for multiplicative functionals in Appendix A of [G]. It would be possible just to appeal to this result. However, the first part of the proof is somewhat simpler in the present situation. Since for this part of the proof in [G] I quote results from [S] for the most part, I have decided to include a complete proof of this very important result which allows one to operate freely with OCHRMs. For its statement we use our standard notation (introduced above (2.8)). And remind the reader that this is isomorphic to the canonical representation of the Borel right semigroup $(P_t)$. As before $m \in \Exc$. \\

\numberwithin{equation}{section}

\numberwithin{theorem}{section}
\setcounter{theorem}{26}
\setcounter{section}{5}
\section*{}
\setcounter{subsection}{2}
\subsection*{}


\begin{theorem} Let $A = (A_t(\omega), t \ge 0)$ be a process such that $A_t: \Omega \rightarrow [0, \infty[$ for each $t \in \mathbb{R}^+ = [0, \infty[$ and 

\begin{enumerate}[(i)]

\item for each $\omega$, $A_0(\omega) =0$, $t \rightarrow A_t(\omega)$ is finite, right continuous, and increasing on $\mathbb{R}^+$ and constant on $[ \zeta(\omega), \infty[$ ; \\

\item $A_t \in \mathcal{F}^m_{t+}$ for $t \ge 0$; \\

\item for each fixed $(s,t) \in \mathbb{R}^+$, $A_{t+s} = A_s + A_t \circ \theta_s$ a.s. $P^m$. \\

Then $A$ is $P^m$ indistinguishable from a process $L=\{ L_t; t \in \mathbb{R}^+\}$ with $L_t : \Omega \rightarrow \overline{\mathbb{R}}^+:=[0,\infty]$ and such that the following properties hold:

\item for each $\omega$, $t \rightarrow L_t(\omega)$ is right continuous, increasing on $\mathbb{R}^+$ and constant on $[\zeta(\omega), \infty[, s \rightarrow L_{t-s}(\theta_s\omega)$ is right continuous and decreasing on $[0,t[,$ $L_0(\omega)$ is either zero or infinite and $t \rightarrow L_t(\omega)$ is finite on  $\mathbb{R}^+$ if $L_0(\omega) =0, L_0([\Delta]) =0$ and $P^m(L_0=\infty) = 0$;

\item $L_t \in \mathcal{F}^*_{t+}$ for $t \ge 0$ and $L_{t-s} \circ \theta_s \in \mathcal{F}^*_{>s}$ for $0 \le s < t$;  

\item $L_{t+s}(\omega) = L_s(\omega) + L_t(\theta_s \omega)$ indentically in $(s,t,\omega) \in  \mathbb{R}^+ \times  \mathbb{R}^+ \times \Omega$. 

\end{enumerate}
\end{theorem}
Note that since $L_0[\Delta] =0$ and $\zeta([\Delta]) = 0$, $L_t([\Delta]) =0$ for all $t \ge 0$. The proof of Theorem 5.27 is postponed to the appendix following this section. However it will be used to prove the next result which establishes the existence of a ``good'' version of an OCHRM.

\begin{theorem} An OCHRM is $Q_m$ indistinguishable from an OCHRM $\kappa$, with the following additional properties:

\begin{enumerate}[(i)]

\item $\kappa$ is a kernel from $(W, \mathcal{G}^*)$ to $(\mathbb{R}, \mathcal{B})$ where $\mathcal{G}^*$ is the universal completion of $\mathcal{G}^0$, $\kappa(w, \cdot)$ is $\sigma$-finite for every $w \in W$ and $\kappa([\Delta], \cdot) =0$ ;

\item $\kappa(\sigma_tw, B) = \kappa( w, B+t)$ for all $w \in W, t \in \mathbb{R}$ and $B \in \mathcal{B}$;
\item $\kappa(b_t w, B) = \kappa( w, B \cap \ ]t, \infty[)$ for all $w \in W, t \in \mathbb{R}$ and $B \in \mathcal{B}$;
\item $\kappa(k_t w, B) = \kappa( w, B \cap \ ]-\infty,t[)$ for all $w \in W, t \in \mathbb{R}$ and $B \in \mathcal{B}$;

\item for each $w \in W, \kappa(w, \cdot) = \sum_{t \in \mathbb{R}} j \circ Y^*_t(w) \epsilon_t + \kappa^c( w, \cdot)$, where $j \in p\mathcal{E}$ with $\{j>0\}$ semipolar and $k^c(w, \cdot)$ diffuse. 

\end{enumerate}\end{theorem}

\begin{definition} A perfect RM is an OCHRM having the additional properties enumerated in (5.28) \end{definition}

The following corollary is now a consequence of Theorems 5.16 and 5.28.

\begin{corollary}A $\sigma$-finite measure, $\mu$ on $E$ not charging $m$-exceptional sets is the characteristic measure of a perfect $RM$, $\kappa$. Moreover $\kappa$ is diffuse if and only if $\mu$ does not charge semipolar sets. \end{corollary}

\noindent {\bf Remark} Note that (5.28 $v$) is an improvement of (5.25) in that $\{j >0\}$ is semipolar rather than $m$-semipolar. \\

\textit{Proof of (5.28).} Suppose first of all that $\mu$ is a finite measure not charging $m$-semipolar sets. Let $\kappa$ be an OCHRM with $\mu^\kappa = \mu$. In view of (5.16)  we may suppose that $\kappa$ is diffuse. Then $Q_m \int e^{-|t|} \kappa (dt) = 2 \mu(E)  < \infty$, and so $\kappa$ is a Radon measure on $(\mathbb{R}, \mathcal{B})$ a.s. $Q_m$. Redefining $\kappa = 0$ on the exceptional set, we may and shall suppose that $\kappa(w, \cdot)$ is a diffuse Radon measure for each $w \in W$. Now define $A_t(\omega) := \kappa(\omega, ]0,t])$ for $t \ge 0$ and $\omega \in \Omega$. Thus $t \rightarrow A_t(\omega)$ is a finite increasing continuous function on $[0, \infty[$ with $A_0(\omega) = 0$ for all $\omega \in \Omega$. Since $A_t = \kappa (]0, t]|)_\Omega$ it follows from (5.24) that $A_t \in \mathcal{G}^m_t |_\Omega = \mathcal{F}^m_t$. Recall that both $(\mathcal{G}^m_t; t \in \mathbb{R})$ and $(\mathcal{F}^m_t ; t \ge 0)$  are right continuous filtrations on $W$ and $\Omega$ respectively. If $B \subset ]0, \infty [, B \in \mathcal{B}$ then from (5.24) since $b_0 = \theta_0$, $\kappa(B) = \kappa(B) \circ b_0 = \kappa(B) \circ \theta_0 \in \mathcal{G}^m_{>0} |_\Omega = \mathcal{F}^m_{>0}.$ In particular  $A_t \in \mathcal{F}^m_{>0}$ for $t >0$ and so if $0 \le s < t, A_{t-s} \circ \theta_s \in \theta_s^{-1} \mathcal{F}^m_{>0} = \mathcal{F}^m_{>s}$. Also if $t> 0$, $\kappa(B) \circ  \theta_t = \kappa(B) \circ  b_0 \sigma_t = \kappa(B) \circ  \sigma_t$ a.s $Q_m$ for such B since $b_0 Q_m = Q_m$ on $\mathcal{G}^m_{>0}$. Let $s,t \ge 0$ and define $\Lambda = \{ \kappa ]0,t+s] \ne \kappa ]0,t]+\kappa ]0,s] \circ \theta_t \}$. By the above considerations since $\kappa$ is a HRM, $Q_m(\Lambda) = 0$ and also $\Lambda \in \mathcal{G}^m_{>0}$. Therefore

\[ 0 = Q_m(\Lambda) \ge Q_m(\Lambda; \alpha < 0 < \beta) = P^m(\Lambda). \]

\noindent Hence $A_{t+s} = A_t + A_s\cdot \theta_t$ a.s. $P^m$ for fixed $s,t \ge 0$. Finally since $\kappa$ is carried by $\Lambda^* \subset [ \alpha, \beta [, t \rightarrow A_t(\omega)$ is constant on $[\zeta(\omega), \infty [.$ Thus $A$ satisfies the hypotheses of Theorem 5.27 and so $A$ is $P^m$ indistinguishable from a process $L$ satisfying \textit{(iv) --- (vi)} in the conclusion of Theorem 5.27. Since $A$ is continuous, $t \rightarrow L_t$ is continuous $P^m$ a.s. on $\mathbb{R}^+$. But we cannot modify $L$ to be continuous for all $\omega$ without destroying the good properties  \textit{(v) and (vi)} of (5.27). However $L$ is an additive functional (AF) of $X$ that is $P^m$ indistinguishable from $A$. \\ 

In order to continue we need a lemma that is a simple extension of a well known result in measure theory. See for example [F084, 1.16]. However for completeness we are going to give a proof, adopting the argument in [F084] with some simple modifications. In order to state the lemma concisely we first introduce some notation. Fix $\alpha \in [-\infty, \infty[$ and define $\mathcal{H}_\alpha$ to be the collection of all half-open intervals $]a,b]$ with $\alpha < a < b < \infty$ together with intervals of the form $]a, \infty[, a > \alpha$, and of the form $]\alpha, b], b < \infty$.  Note that $] \alpha, \infty[$ is the disjoint union of two elements of $\mathcal{H}_\alpha$. When $I \in \mathcal{H}_\alpha$, the complement of $I$, $I^c$, means the complement of $I$ in $]\alpha, \infty[$; that is, $I^c = ]\alpha, \infty[ \ \setminus \ I$. Clearly $\mathcal{H}_\alpha$ is closed under finite intersections and with the above convention the complement of $I \in \mathcal{H}_\alpha$ is either in $\mathcal{H}_\alpha$ or the disjoint union of two intervals in $\mathcal{H}_\alpha$. It will be convenient to define $H_\alpha = ]\alpha, \infty [$. It is a standard fact that the collection $\mathcal{A}_\alpha$, of finite disjoint unions of intervals in $\mathcal{H}_\alpha$ forms an algebra of subsets of $H_\alpha$. Of course, $\mathcal{B}_\alpha = \sigma(\mathcal{H}_\alpha) = \sigma ( \mathcal{A}_\alpha)$ where $\mathcal{B}_\alpha$ is the Borel $\sigma$-algebra of $H_\alpha$. Since the elements of $\mathcal{H}_\alpha$ are contained in $H_\alpha \subset \mathbb{R}$, we adopt the convention that if $b = \infty$, the statement $I = ]a,b] \in \mathcal{H}_\alpha$ means $I \cap H_\alpha = ] a, \infty[ \in \mathcal{H}_\alpha$. Suppose given a function $\lambda: \{ (a,b); \alpha<a<b<\infty \} \rightarrow 
\oR^+ = [0, \infty]$ satisfying the following conditions:


\numberwithin{equation}{section}
\setcounter{equation}{30}
\setcounter{section}{5}

\begin{align}
(i) \qquad & b \rightarrow \lambda (a,b) \mathrm{ \ is \  increasing \  and \ right \  continuous\  on} \notag \ ]a,\infty[, a > \alpha; \\
(ii)\qquad  & a \rightarrow \lambda (a,b) \mathrm{\ is\  decreasing \ and\  right\  continuous\ on} \ ]\alpha, b[, b < \infty; \\
(iii) \qquad  & \alpha < a < b < c < \infty, \mathrm{ \  then}\  \notag \lambda(a,c) = \lambda(a,b) +\lambda(b,c).
\end{align}

\noindent Note that if $F: \mathbb{R}: \rightarrow [0, \infty[$ is right continuous and increasing then $\lambda(a,b) = F(b)-F(a), \alpha < a < b < \infty$ satisfies (5.31) for any $\alpha \ge -\infty$. However we make no finiteness assumptions on $\lambda$; $\lambda$ might be indentically infinite. Define $\lambda(\alpha,b) := \uparrow \lim_{a \downarrow \alpha} \lambda(a,b)$ if $b< \infty$ and define $\lambda(a, \infty):= \uparrow \lim_{b\uparrow \infty}\lambda(a,b)$ if $a > \alpha$ and note that (5.3 \textit{iii}) continues to hold when $\alpha \le a< b \le \infty$. \\


\setcounter{theorem}{31}

\begin{lemma} Fix $\alpha \ge - \infty$ and let $\lambda:\{(a,b); \alpha <a <b<\infty\} \rightarrow [0, \infty]$ satisfy (5.31). Then there exists a measure $\nu$, on $(H_\alpha, \mathcal{B}_\alpha)$ such that if $I= ]a,b] \in \mathcal{H}_\alpha$, $\nu(I) = \lambda(a,b)$. Moreover if $\overline{\nu}$ is another measure on $(H_\alpha, \mathcal{B}_\alpha)$ with $\overline{\nu}=\nu$ on $\mathcal{A}_\alpha$, then $\overline{\nu}(B) \le \nu(B)$ for all $B \in \mathcal{B}_\alpha$ with equality when $\nu(B) < \infty$. Of course, $\overline{\nu} = \nu$ on $\mathcal{A}_\alpha$ if and only if $\overline{\nu}=\nu$ on $\mathcal{H}_\alpha$. \end{lemma}

\begin{proof} If $I=]a,b] \in \mathcal{H}_\alpha$, define $\nu(I) = \lambda(a,b)$ and observe that $\nu$ is finitely additive on $\mathcal{H}_\alpha$ and monotone increasing on $\mathcal{H}_\alpha$ in view of the conditions \textit{(i)---(iii)} in (5.31). Extend $\nu$ to $\mathcal{A}_\alpha$ by defining $\nu(A) = \sum_{j=1}^n \nu(I_j)$ when $A$ is  the finite disjoint union of $I_j \in \mathcal{H}_\alpha$, $i \le j \le m$. If $A$ may also be expressed as the disjoint union of $J_k \in \mathcal{H}_\alpha, 1 \le k \le m$, then using the finite additivity of $\nu$ on $\mathcal{H}_\alpha$ we have 

\[ \nu(A) = \sum_j \nu(I_j) = \sum_{j,k} \nu (I_j \cap J_k) = \sum_k \nu(J_k) \]

\noindent and so $\nu$ is well defined on $\mathcal{A}_\alpha$. It is evident that $\nu$ is finitely additive on $\mathcal{A}_\alpha$. It is a standard fact that $\nu$ extends to a measure on $\mathcal{B}_\alpha = \sigma(\mathcal{A}_\alpha)$ if and only if $\nu$ is countably additive on $\mathcal{A}_\alpha$. See Theorem 1.14 in [Fo84], for example. The extension is constructed via the familiar Caratheodory extension theorem. That is one first defines an outer measure $\nu^*$ on all subsets $E \subset H_\alpha$ by 

\[ \nu^*(E) = \inf\{ \sum_{j=1}^\infty \nu(A_j): A_j \in \mathcal{A}_\alpha, E \subset \cup_{j=1}^\infty A_j\} \]

\noindent and then shows that the restriction of $\nu^*$ to $\mathcal{B}_\alpha = \sigma(\mathcal{A}_\alpha)$ is a measure on $\mathcal{B}_\alpha$. 

Thus we must show that $\nu$ is countably additive on $\mathcal{A}_\alpha$. Since $\nu$ is finitely additive on $\mathcal{A}_\alpha$ it suffices to show that if $(I_j)$ is a disjoint sequence in $\mathcal{H}_\alpha$ with $A = \cup I_j$, then $\nu(A) = \sum_j \nu(I_j)$. Since $A \in \mathcal{A}_\alpha$, $A$ is a finite union of disjoint elements of $\mathcal{H}_\alpha$. Hence by the finite additivity of $\nu$ on $\mathcal{H}_\alpha$, it suffices to consider the case $\cup_{j=1}^\infty I_j = I \in \mathcal{H}_\alpha$. Since $\nu$ is increasing on $\mathcal{H}_\alpha$, $\nu(I) \ge \sum_{j=1}^n \nu(I_j)$ and letting $n \rightarrow \infty$ we obtain $\nu(I) \ge \sum_{j=1}^\infty \nu(I_j)$. In checking the opposite inequality we suppose that $\nu(I_j) <\infty$ for all $j$ since it is obvious if $\nu(I_j) = \infty$ for some $j$. Consider first the case $I = ]a,b]$ with $ \alpha < a < b < \infty$. Fix $\epsilon >0$. Suppose first that there exists $c$, $a<c<b$ with $ \lambda(a,c)<\infty$. By the right continuity of $\lambda$ in its first variable we may choose $\delta$ with $0 < \delta < c-a$ so that $ \lambda(a, a + \delta) < \epsilon$. If $I_j = ]a_j,b_j]$, then because $\nu(I_j) < \infty$ and the right continuity of $\lambda$ in its second variable, for each $j$ there exists $\delta_j > 0$ such that $ \lambda(a_j, b_j + \delta_j) - \nu(I_j) < \epsilon 2^{-j}$. Then the open intervals $]a_j, b_j+\delta_j[$ cover the compact interval $[a+\delta, b]$ because $]a,b] = \cup_j ]a_j,b_j]$ is only possible if $b_j =b$ for some $j$. Choose a finite subcover and discard any $]a_j, b_j +\delta_j]$ that is contained in a larger one. Relabeling the intervals of the subcover we may suppose that 
\textit{(i)} the intervals $]a_1, b_1+\delta_1] \dots ]a_n, b_n + \delta_n]$ cover $[a+\delta, b],$ 
\textit{(ii)} $a_1 < a_2 < \dots a_n$, and 
\textit{(iii)} $b_j + \delta_j \in ]a_{j+1}, b_{j+1} + \delta_{j+1}[$ for $i \le j < n-1$. The conditions 
\textit{(ii)} and \textit{(iii)} hold because the $I_j$ are disjoint. Then since $a_1 < a+ \delta$ and $  b< b_n + \delta_n$ for $1 \le j \le n-1$ we have

\begin{align*} \nu(I)   \le \lambda(a+ \delta, b) + \epsilon & \le \lambda(a_1, b_n + \delta_n) + \epsilon \\
 &\le \sum_{j=1}^{n-1} \lambda(a_j, b_j+\delta_j) + \lambda (a_n, b_n + \delta_n) + \epsilon\\
 &\le \sum_{j=1}^n ( \nu(I_j) + \epsilon 2^{-j}) + \epsilon \le \sum_{j=1}^n \nu(I_j) + 2 \epsilon\\ 
 &\le \sum_{j=1}^\infty \nu(I_j) + 2\epsilon \end{align*}

\noindent Since $\epsilon >0$ is arbitrary we obtain $\nu(I) \le \sum_{j=1}^\infty \nu(I_j)$ in this case. \\

Next, still assuming $I = ]a,b]$, $\alpha < a < b < \infty$, suppose that $\lambda(a,c) = \infty$ for all $c \in ]a,b[.$ Then $\lambda(a,b) \ge \lambda(a,c) = \infty$, $a<c<b$. In particular $\nu(I) = \lambda(a,b) = \infty$ and so we must show that $\sum_j \nu(I_j) = \infty$. In this case, $\lambda(a+\delta, b) \uparrow \infty$ as $\delta \downarrow 0$ by (5.31-\textit{ii}). Fix $ 0 < \delta < b-a$ and $\epsilon >0$. Let the $\epsilon_j, j\ge 1$ be chosen as before. Once again the open intervals $]a_j, b_j+\delta_j[$ cover the compact interval $[a+\delta, b]$. We choose a finite subcover and arrange the notation as before so that the conditions \textit{i ---\ iii} above hold. Then

\begin{align*} \lambda(a+ \delta, b)  \le &  
 \sum_{j=1}^n \lambda(a_j, b_j+\delta_j)   \le  \sum_{j=1}^n (\nu(I_j) + \epsilon 2^{-j}) \\
 \le & \sum_{j=1}^n \nu(I_j) + \epsilon  
\le   \sum_{j = 1}^{\infty} \nu(I_j) + \epsilon \end{align*}

and since $\epsilon >0$ is arbitrary 
\[ \lambda(a+\delta, b) \le \sum_{j=1}^{\infty} \nu(I_j).\]

\noindent Letting $\delta \downarrow 0$ forces $\sum_{j=1}^\infty \nu(I_j) =\infty$, establishing (5.32) whenever $I = ]a,b]$ and $\alpha < a< b <\infty$. Finally if $b < \infty$ and $I = ]\alpha, b]$ is the disjoint union of $(I_j, j \ge 1) \subset \mathcal{H}_\alpha$ then $I \cap ]c,b]$ is the disjoint union of $I_j \cap ]c,b]$ if $\alpha < c <b$. Then by the previous case

\[ \nu(]c,b]) = \sum_j \nu(I_j \cap ]c,b]) \le \sum_j \nu(I_j) \]

\noindent and so 

\[ \nu(I) = \lim_{c \downarrow \alpha} \lambda(c,b) = \lim_{c\rightarrow \infty} \nu(]c,b]) \le \sum_j \nu(I_j).\]

\noindent Similarly we obtain the same result if $I =]a, \infty[, \alpha <a$. Thus the existence of $\nu$ is established in all cases. 

We have based the existence of $\nu$ on Theorem 1.14 in [Fo84]. But the last assertion in (5.32) is part of the conclusion of the result just cited, completing the proof of Lemma 5.32.
\end{proof}

Before returning to the proof of (5.28), we formulate a well-known result in measure theory as a lemma. For completeness we shall furnish a proof. \\

\begin{lemma}

Let $H \in \mathcal{G}_{t-}^*$ (resp. $\mathcal{G}_{>t}^*)$ with $H([\Delta]) =0$. Then $H \circ k_t = H$ and $H \circ b_t =0$ (resp. $H \circ b_t = H $ and $H \circ k_t = 0)$. 

\end{lemma}

\begin{proof} Let $H \in \mathcal{G}^*_{t-}, s<t$ and $P$ a probability on $( \Omega, \mathcal{G}^0)$. If $H \in \mathcal{G}^*_s$, then there exist $H_1, H_2 \in \mathcal{G}_s^0$ with $H_1 \le H \le H_2$ and $P(H_1 <H_2) =0$. Since $t>s$, $H_ j \circ k_t = H_j$ for $j=1,2$. Hence $H_1 \le H \circ k_t \le H_2$ and so $P(H \circ k_t = H)=1$.  But $H_j \circ b_t = H_j([\Delta])$ and if $H_1([\Delta]) < H_2 ( [\Delta])$, then $P( \{[\Delta] \}) =0$. Moreover $H_1([\Delta]) \le H([\Delta]) =0$. Consequentially $P(H \circ b_t) =0$. Since $(W, \mathcal{G}^0)$ is a coSouslin measurable space according to (2.16), $
\{w \} \in \mathcal{G}^0$ for $w \in W$. Fix $w \in W$.  Then $P^w = \epsilon_w$ is a probability on $(\Omega, \mathcal{G}^0)$ and taking $P = P^\omega$ we find that $H(k_t w) = H(w)$ and $H(b_t w) =0$. But this holds for all $w \in W$ and so $H \circ k_t = H$ and $H \circ b_t =0$ if $H \in \mathcal{G}_s^*, s<t$. Now $\mathcal{G}_{t-}^* = (\mathcal{G}_{t-}^0)^* = ( \sigma \cup_{s<t} \mathcal{G}_s^0)^* \subset \sigma \cup_{s<t} \mathcal{G}_s^*$---in fact the last inclusion is an equality. Hence $H \circ k_t = H$ and $H \circ b_t = 0$ when $H \in \mathcal{G}_{t-}^*$. The other cases are treated similarly. \end{proof}

We now return to the proof of $(5.28)$. Recall that $L = (L_t(\omega), t \ge 0)$ is a process $P^m$ indistinguishable from $A_t = \kappa (] 0,t]|_\Omega)$ satisfying the conclusion of Theorem 5.27 and that $t \rightarrow L_t$ is $P^m$ a.s. continuous on $[0, \infty[$. Moreover if $\alpha(w) \le a < \beta(w)$, then $\alpha(\theta_a w)=0$ so that $\theta_a w \in \Omega$. If $a \ge \beta(w)$, clearly $\theta_a w(s) = \Delta$ for all $s \ge 0$ so that $\theta_a w [\Delta] \in \Omega$. Then the following definition makes sense for all $w \in W$ and $-\infty < a < b < \infty$. 

\pagebreak

\setcounter{equation}{33}

\begin{equation} \quad \lambda(a,b,w)  := L_{b-a}(\theta_a w) \quad \mathrm{if} \ \alpha(w) \le a\end{equation} \begin{equation} \notag \ \ \ \ \ :=0 \quad \mathrm{if} \ a < \alpha(w). \end{equation}

\noindent If $a \ge \beta(w)$, then $\theta_a w = [ \Delta]$ and so from (5.27 \textit{iv}), $\lambda(a,b,w) =0$ when $a \ge \beta(w)$ as well as when $\alpha(w) <a$. Observe that for each $w \in W$, (5.27 \textit{iv}) implies that $\lambda $ with $\alpha = \alpha(w)$ satisfies \textit{(i)} and \textit{(ii)} of (5.31), while (5.27 \textit{vi}) implies that $\lambda$ satisfies \textit{(iii)} of (5.31).  Thus, writing $\lambda(w)$ for the function $(a,b) \rightarrow \lambda(a,b,w), \alpha(w) < a < b < \infty$, then for each $w \in W$, $\lambda(w)$ satisfies the hypotheses of Lemma 5.32. Therefore $\lambda(w)$ induces a measure $\overline{\kappa}(w)= \overline{\kappa}(w, \cdot)$ on $]\alpha(w), \infty[$ for each $w \in W$. From (5.27 \textit{iv}), $L_t([\Delta]) =0$ for $t \ge 0$ and so $\overline{\kappa}([\Delta]) =0$. Since $\theta_a w = [\Delta]$ if $a \ge \beta(w)$, $\overline{\kappa}$ is carried by $] \alpha, \beta [$. It is convenient to regard $\overline{\kappa}$ as a measure on $\mathbb{R}$ that is carried by $] \alpha, \beta[$. Then 

\begin{align} \notag Q_m (\overline{\kappa}]a,b] & = \infty) = Q_m [(L_{b-a}(\theta_a) = \infty) = \infty, \alpha < a < \beta] \\ \notag &= Q_m [P^{Ya} (L_{b-a} = \infty) ] = P^m(L_{b-a} = \infty) = 0 \end{align}

\noindent provided $a,b \in \mathbb{R}$ and $a<b$. Consequently a.s. $Q_m, \overline{\kappa}$ is a Radon measure on $(\mathbb{R}, \mathcal{B})$. But $t \rightarrow \overline{\kappa} ]a,t] = L_{t-a}(\theta_a)$ is continuous on $]a, \infty[$ provided $a > \alpha$. Therefore $Q_m$ a.s, $\overline{\kappa}$ is a diffuse Radon measure. Finally note that (5.27 \textit{v}) implies that $\overline{\kappa}]a,b] \in \mathcal{G}^*_{>a} \cap \mathcal{G}^*_{b+} \subset \mathcal{G}^*$. \\

Let $V(w,a,b)$ be the indicator of the set $\{ w,a,b: \alpha(w) < a<b< \beta(w)\} \subset W \times \mathbb{R} \times \mathbb{R}$. Then 

\setcounter{equation}{34}
\begin{equation} \overline{\kappa}( \sigma_tw, ]a,b])  = L_{b-a} (\theta_a \sigma_t w) V(\sigma_t w, a,b)\end{equation}
\begin{equation} \notag \qquad \ \qquad \qquad \qquad \qquad = L_{b+t-(a+t)}(\theta_{a+t}w) V (w, a+t, b+t) \end{equation}
\begin{equation} \notag \ \ \ \ \ = \overline{\kappa}(w, ]a,b]+t) \end{equation}

\noindent where the second equality follows because $\alpha(\sigma_tw) = \alpha(w)-t$ and $\beta(\sigma_t w) = \beta(w) -t)$. Define 

\begin{equation} G_0 := \{ w: \overline{\kappa} (w) \ \mathrm{ \ is \ a \ Radon \ measure} \} = \bigcap_n \{w: \overline{\kappa}(w, ]-n,n]) < \infty \} \in \mathcal{G}^* \end{equation}
\noindent and using (5.35)

\[ G_0(\sigma_t w) = \bigcap_n \{ w: \overline{\kappa}(w, ]-n+t, n+t]) < \infty\} = G_0 (w). \]

\noindent Hence $G_0 \in \mathcal{G}^*$ and $\sigma_t^{-1} G_0 = G_0$ for $t \in \mathbb{R}$. Next define

\begin{equation}   G  = \{ w \in G_0: \overline{\kappa}(w, \cdot)  \ \mathrm{is \ diffuse} \} \end{equation}
\noindent If $w \in G_0$ and $a \in \mathbb{R}$, define $f^a(t,w) := \overline{\kappa}(w, ]a,t])1_{]\alpha(w), \infty[}(a)$. Then $f^a(\cdot, w)$ is right continuous, finite,and increasing on $\mathbb{R}$. Set $f_{-}^a(t,w) = f^a(t-,w)$ for $w \in G_0$, so that $f_{-}^a(\cdot, w)$ is left continuous on $\mathbb{R}$. Both $f^a(t, \cdot)$ and $f_{-}^a(t, \cdot)$ are $\mathcal{G}^*|_{G_0} \subset \mathcal{G}^*$ measurable. Consequentially both $f^a$ and $f_{-}^a$ are $\mathcal{B}\times \mathcal{G}^*$ measurable and so $\Gamma := \{ (t,w) : f^a(t,w) \neq f_{-}^a(t,w)\} \in \mathcal{B} \times \mathcal{G}^*$. 

\noindent Observe that

\[ G^a := \{ w \in G_0 : f^a(\cdot, w) \mathrm{ \ is \ not \ continuous} \}\] is the projection of $\Gamma$ on $W$. If $P$ is a probability on $(W, \mathcal{G}^*)$, it follows from [DM III, 13 and 33] that $G^a\in \overline{\mathcal{G}^*}^P$, and since $P$ is arbitrary $G^a \in \mathcal{G}^*$. Let $a \downarrow - \infty $ through a sequence and take the complement of the resulting set to obtain $G \in \mathcal{G}^*$ where $G$ is defined in (5.37). Next observe that (5.35) and the fact that $\sigma_t^{-1}G_0 = G_0$ imply that $\sigma_t^{-1}G = G$. \\

Finally define 
\begin{equation}\widetilde{\kappa}(w, \cdot) := 1_G(w) \overline{\kappa}(w,\cdot). \end{equation}

\noindent From the above $\widetilde{\kappa}(w,\cdot)$ is a diffuse Radon for each $w$ and $\widetilde{\kappa}$ is a kernel from $(W, \mathcal{G}^*)$ to $(\mathbb{R}, \mathcal{B})$. Moreover $\widetilde{\kappa}([\Delta])=0$. Also (5.35) and the fact that $\widetilde{\kappa}(w,\cdot)$ is Radon imply that

\begin{equation} \widetilde{\kappa}(\sigma_t w, B) = \widetilde{\kappa} (w, B+t) \end{equation}

\noindent for each $w \in W, t \in \mathbb{R}$ and $B \in \mathcal{B}$. We next claim that $\kappa$ and $\widetilde{\kappa}$ are $Q_m$-indistinguishable. It was noted in the first paragraph of the proof of (5.28) that if $B \in \mathcal{B}$, $B \subset ]0,\infty[$, $\kappa (B) \circ \sigma_a = \kappa(B) \circ \theta_a$, a.s. $Q_m$. Therefore a.s. $Q_m$, if $a<b$
\begin{align*} \qquad \qquad \qquad \kappa (]a,b]) = &\kappa (]0,b-a]) \circ \sigma_a = \kappa(]0, b-a])\circ \theta_a\\  
= & A_{b-a} \circ \theta_a, \end{align*}

\noindent and so 

\begin{equation}Q_m(\kappa (]a,b])  \ne \overline{\kappa} (]a,b]); \alpha < a) \end{equation}
\begin{equation} \notag \qquad\qquad \qquad \qquad \ \qquad \ \ \ \  \ = Q_m (A_{b-a}\theta_a \ne L_{b-a} \circ \theta_a ; \alpha <a ) \end{equation}
\begin{equation} \notag \qquad \qquad \qquad\qquad  \ = P^m (A_{b-a} \ne L_{b-a}) = 0. \end{equation}

\noindent Since $\kappa$ is a Radon measure a.s. $Q_m$, it follows from (5.36) and (5.40), noting that both $\kappa$ and $\overline{\kappa}$ are carried by $]\alpha,\beta[ \ Q_m$ a.s., that $\overline{\kappa}$ is also a Radon measure a.s. $Q_m$. In particular, $Q_m(G_0^c) =0$. Moreover since $\mathcal{B}$ is countably generated, (5.40) now implies that $\kappa$ and $\overline{\kappa}$ are $Q_m$-indistinguishable. But $\kappa$ is diffuse a.s. $Q_m$ and hence so is $\overline{\kappa}$. Consequently, $Q_m(G^c)=0$ and so $\kappa$ and $\widetilde{\kappa}$ are $Q_m$-indistinguishable. In particular $\mu_{\widetilde{\kappa}} = \mu_\kappa = \mu$. We drop the tilde ``$\sim$'' in our notation and write $\kappa$ for $\widetilde{\kappa}$. \\

To summarize what has been established so far: Given a finite measure, $\mu$ not changing $m$-semipolars there exists an OCHRM, $\kappa$ with $\mu_\kappa = \mu$ such that $\kappa$ has the properties 
\textit{(i)} and \textit{(ii)} of Theorem 5.28, $\kappa(w)$ is diffuse Radon measure for every $w \in W$ and $\kappa(]a,b]) \in \mathcal{G}^*_{>a} \cap \mathcal{G}^*_{b+}$ for $-\infty < a < b < \infty$. We are next going to check that $\kappa$ also satisfies \textit{(iii)} and \textit{(iv)} of (5.28). \\

To this end observe that if $-\infty < u < v < \infty$, then $\kappa(]u,v[) \in \mathcal{G}^*_{>u} \cap\mathcal{G}^*_{v-}$. Therefore in view of Lemma 5.33, $\kappa(]u,v[)  \circ b_u = \kappa(]u,v[) $ and $\kappa(]u,v[) \circ b_v =0$. Thus if $t \le u$ or $t \ge v$, $\kappa(]u,v[)  \circ b_t = \kappa(]u,v[ \cap ]t,\infty[)$. If $u<t<v$, $\kappa(]u,v[) = \kappa(]u,t]) + \kappa(]t,v[)$. But $\kappa$ being diffuse, $\kappa(]u,t]) = \kappa(]u,t[)$ and so from the above $\kappa(]u,v[)\circ b_t = \kappa(]t,v[)= \kappa(]u,v[ \cap ]t, \infty[)$. Since $\kappa(b_t,B)$ and $\kappa(B\cap ]t,\infty[)$ are diffuse Radon measures in $B$ that agree when $B = ]u,v[$, they agree for all $B \in \mathcal{B}$. Thus $\kappa$ satisfies (5.28 \textit{iii}). A similar argument yields (5.28 \textit{iv}). \\

Now suppose that $\mu$ is a $\sigma$-finite measure not charging $m$-semipolars. Let $f \in \mathcal{E}$ with $f>0$. and  $\mu(f) < \infty$. Then $\mu_f(dx):=f(x)\mu(dx)$ is finite and does not charge $m$-semipolars. By what has been established so far there exist an OCHRM, $\kappa_f$, satisfying \textit{(i)---(iv)} of (5.28) with $\mu_{\kappa_f} = \mu_f$ such that for each $w$, $\kappa_f(w, \cdot)$ is diffuse Radon measure on $(\mathbb{R},B)$.  Define $\kappa(w,dt):=(f \circ Y_t(w))^{-1} \kappa_f(w,dt)$. It is easy to check that $\kappa$ is an OCHRM satisfying \textit{(i)---(iv)} of (5.28) such that each $w$, $\kappa(w, \cdot)$ is a diffuse $\sigma$-finite measure on $(\mathbb{R},\mathcal{B})$. Clearly $\mu_\kappa = \mu$. Of course, $\kappa(w, \cdot)$ is not a Radon measure in general. \\

Next let $\kappa$ be an arbitrary OCHRM. By (5.11) and (5.12i), $\mu_k$ is $\sigma$-finite. Let ``$\sim$'' denote the equivalence relation of $Q_m$ indistinguishability. By (5.25), $\kappa \sim \kappa_j + \kappa^c$, where $\kappa^c$ is diffuse and $\kappa_j = \sum_t j(Y_t^*)\epsilon_t$ with $j \in p \mathcal{E}$ and $\{j>0\}$ is $m$-semipolar. Note that if $\kappa^d$ is the discrete part of $\kappa$, we only have $\kappa^d\sim\kappa_j$. Now $\mu_{\kappa^c}$ is $\sigma$-finite and doesn't charge $m$-semipolars. Therefore there exists an OCHRM, $\widetilde{\kappa}$ 
with $\mu_{\widetilde{\kappa}} =
 \mu_{\kappa^c}$ and $\widetilde{\kappa}$ 
satisfies \textit{(i)---(iv)} of (5.28) and $\widetilde{\kappa}(w,
\cdot)$ is diffuse for every $w$. But (5.12 \textit{ii}) states that $\kappa^c
\sim \widetilde{\kappa}$ and so $\kappa \sim \kappa_j +
\widetilde{\kappa}$. On the other hand it is easy to check that
$\kappa_j$ satisfies \textit{(i)---(iv)} of (5.28). Thus $\kappa^* := \kappa_j
+ \widetilde{\kappa} \sim \kappa$ and $\kappa^*$ has the properties
\textit{(i)---(iv)} of (5.28). \\

It remains to show that $j$ may be chosen so that $\{j>0\}$ is semipolar rather than $m$-semipolar. This important improvement comes from [FG 03]. It depends on the following result of Dellacherie [D 88, p.70]. Let $\rho$ be a $\sigma$-finite measure on $E$. Recall that $A \subset E$ is $\rho$-semipolar provided $A \subset B \in \mathcal{E}$ and $\{ t \ge 0: X_t \in B\}$ is countable a.s. $P^\rho$. Then a $\rho$-semipolar set $A$ has the form $A = A_1 \cup A_2$ with $A_1$ $m$-polar and $A_2$ semipolar. In fact, this  is Dellacherie's definition of a $\rho$-semipolar set; its equivalence with the $P^\rho$ a.s. countability of $\{t: X_t \in A\}$ is what he proves. I am going to abandon my desire to prove all results not available in books in this instance, and just refer the reader to [De 88]. We need a lemma which is of considerable interest in its own right. Recall that $\mathcal{N}(m)$ denotes the $m$-exceptional sets. \\

\setcounter{theorem}{40}

\begin{lemma} A Borel $m$-semipolar set is the disjoint union of a Borel semipolar set and a Borel set in $\mathcal{N}(m)$. In particular, a Borel $m$-polar set is the disjoint union of a Borel semipolar set and a Borel set in $\mathcal{N}(m)$. \end{lemma}

We shall use (5.41) to complete the proof of (5.28), after which we shall prove (5.41). Now $\kappa_j = \sum_t j \circ Y_t^* \epsilon_t$ with $j \in p \mathcal{E}$ with $\{j>0\}$ $m$-semipolar. Then by (5.41), $\{j>0\} = A \cup B$ with $A,B \in \mathcal{E}$, $A \in \mathcal{N}(m)$, $B$ semipolar and $A \cap B = \emptyset$. Define $j':= j 1_B \in \mathcal{E}$, then $\{j'>0\} =B$ is semipolar. Let $\kappa_j' := \sum_t j'(Y_t^*)\epsilon_t$. Let $\phi\in p \mathcal{B}$ and $\int \phi(t)dt =1$. If $D \in \mathcal{E}$, then \[\mu_{\kappa_j}(D) = Q_m \int \phi(t)1_D (Y_t^*) \kappa_j'(dt)+ Q_m \int \phi(t) 1_{D \cap A} (Y_t^*) \kappa_j(dt).\]


\noindent But $A \in \mathcal{N}(m)$; hence $\{Y^* \in A \}$ is $Q_m$ evanescent. Therefore \hbox{$\mu_{\kappa_j} = \mu_{\kappa_j}'$} and so $\kappa_j \sim\kappa_j'$; hence $\kappa \sim \kappa_j' + \widetilde{\kappa}$ completing the proof of (5.28). \hspace{105 mm}$\square$

\begin{proof}[Proof of (5.41)] Let $A \in \mathcal{B}$ be an $m$-semipolar set. Then by Dellacherie's theorem $A= B \cup A_2$ with $B$ $m$-polar and $A_2$ semipolar. Using the fact that an $m$-polar set is contained in a Borel $m$-polar set and that an $m$-semipolar set is contained in a Borel $m$-semipolar set by definition, it is easy to see that if $A$ is Borel, then both $B$ and $A_2$ may be chosen Borel with $B \cap A_2 = \emptyset$. Let $m = \nu + \rho U$ be the Riesz decomposition of $m$ into harmonic and potential parts. Then $B$ is $\rho U$-polar. Let $D = D_B := \inf\{t\ge: X_t \in B\}$ denote the d\'ebut of $B$. Recall that for a $\sigma$-finite measure $\mu$, the phrase ``$B$ is $\mu$-polar'' means that $\mathbf{P}^\mu[D < \infty] =0$. Thus,

\[ 0 = \mathbf{P}^{\rho U} [D < \infty] = \int_0^\infty dt  \int \rho (dy) \int P_t (y,dx) P^x(D<\infty) \]

\[= \int_0^\infty \mathbf{P}^\rho [D \circ \theta_t < \infty]dt,\]

\noindent and so $\mathbf{P}^\rho[D \circ \theta_t < \infty] =0$ for (Lebesgue) a.e. $t>0$. But $\lim_{t \downarrow 0}  D \circ \theta_t = \inf\{t>0: X_t \in B \}$ and so $\{ t \ge 0: X_t \in B\} \subset \{0\}$, $\mathbf{P}^\rho$-a.s. In particular $B$ is $\rho$-semipolar. Using Dellacherie's theorem as above $B = B_1 \cup B_2$ with $B_1$ $\rho$-polar, $B_2$ semipolar and with $B_1$ and $B_2$ disjoint Borel sets. But $B_1$ being $\rho$-polar implies $\rho(B_1)=0$ and since $B_1 \subset B$, $B_1$ is $m$-polar; that is $B_1 \in \mathcal{N}(m)$. Finally $A = B_1 \cup B_2 \cup A_2$ establishing (5.41) since the union of two semipolar sets is semipolar.\end{proof}



\section*{Appendix to Section 5}

\numberwithin{equation}{section}
\renewcommand{\theequation}{A.\arabic{equation}}
\renewcommand{\thelemma}{A.\arabic{equation}}
\setcounter{equation}{0}

This appendix is devoted to a proof of Theorem 5.27. Thus let $A$ satisfy the hypotheses of (5.27). The first step is to replace $A$ with an indistinguishable process that is $(\mathcal{F}_{t+}^0)$ adapted for which (i) and (iii) of (5.27) continue to hold. \textit{In this appendix indistinguishable means $P^m$-indistinguishable unless explicitly stated otherwise.} To begin since $A_t \in \mathcal{F}_t^m$, there exists $A_t^0 \in \mathcal{F}_t^0$ with $P^m(A_t \ne A_t^0) =0$. But $A_t \in [0,\infty[$ and so we may suppose that $A_t^0 \in [0,\infty[$ and we may suppose that $A_t^0([\Delta]) =0$ since $A_t([\Delta]) =0$. During this proof $r,s,t,$ etc. denote elements of $\mathbb{R}^+ = [0,\infty[$. For $t \ge 0$, define $\widetilde{A}_t:=\inf \{ A_r^0 :r>t, r \in \mathbb{Q} \}$. Since $t \rightarrow A_t$ is increasing and right continuous on $[0,\infty[$, $\widetilde{A}_t \in \mathcal{F}^0_{t+}$ and $P^m(\widetilde{A}_t \ne A_t)=0$. But $\widetilde{A}_t$ is finite, increasing and right continuous by construction and so $A$ and $\widetilde{A}$ are indistinguishable. Since $A$ is constant on $[\zeta, \infty[$, one easily checks that $t \rightarrow \widetilde{A}_{t \wedge \zeta} = \widetilde{A}_t 1_{[0,\zeta[}(t)+\widetilde{A}_\zeta 1_{[\zeta, \infty[}(t) $ and $A$ are indistinguishable and that $\widetilde{A}_{t \wedge \zeta} \in \mathcal{F}_{t+}^0$. Thus replacing $\widetilde{A}_t$ by $\widetilde{A}_{t+\zeta}$ we may assume that $\widetilde{A}$ is constant on $[\zeta, \infty[$. \\

For each fixed $s,t$ with $0 \le s < t$ \begin{align}  P^m   (\widetilde{A}_{t-s}   \circ \theta_s \ne A_{t-s} \circ \theta_s)  & = P^m(P^{X(s)}(\widetilde{A}_{t-s}\ \ne A_{t-s})) \\ \notag & = P^{mP_s} (\widetilde{A}_{t-s} \ne A_{t-s})\le P^m (\widetilde{A}_{t-s} \ne A_{t-s}) \\ & \notag \ne 0 \end{align}

\noindent because $m \in \Exc$. By right continuity, for each $s>0$, $t \rightarrow \widetilde{A}_{t-s} \circ \theta_s$ and $t \rightarrow A_{t-s}\circ \theta_s$ are indistinguishable on $[s, \infty[$. Also for fixed $s \ge 0$ and $t \ge 0$, $P^m$ a.s. $\widetilde{A}_{t+s} = A_{t+s} = \widetilde{A}_s + A_t \circ \theta_s$. But as in (A.1), $P^m (\widetilde{A}_t \circ \theta_s \ne A_t \circ \theta_s) = 0;$ hence $\widetilde{A}$ satisfies (i) and (iii) of (5.27). \\

\noindent Since $\widetilde{A}_t \in \mathcal{F}_{t+}^0 \subset \mathcal{F}^0, \widetilde{A}_{t-s} \circ \theta_s \in \theta_s^{-1}(\mathcal{F}^0) = \mathcal{F}^0_{\ge s}$. \\

We now replace $A$ by $\widetilde{A}$ and drop the tilde ``$\sim$'' in our notation. More precisely we suppose that $A$ satisfies the following conditions: \vspace{-.2in}
\setcounter{equation}{2}

\noindent

\begin{equation}\tag{A.2}
\begin{cases}

\mathrm{(i)  \ for \ each} \ \omega \in \Omega, A_0(\omega) = 0, t \rightarrow A_t(\omega) \mathrm{ \ is \ finite, \ right \ continuous}  \\

 \quad \ \mathrm{ and \ increasing \ on \ } [0, \infty[, 
\notag  \mathrm{ \ constant \ on \ } [\zeta(\omega), \infty[    \mathrm{ \ and \ } \\ \quad A_{t}[\Delta]=0 \mathrm{ \ for \ all} \ t \ge 0;\\

 \notag \mathrm{ (ii) \ } A_t \in \mathcal{F}_{t+}^0 \mathrm{ \ for \ } t \ge 0 \mathrm{ \ and\ } A_{t-s}\circ \theta_s \in \mathcal{F}_{\ge s}^0 \mathrm{ \ for} \ 0 \le s \le t; \\
 
\notag \mathrm{(iii) \ for \  each \ } (s,t) \in \mathbb{R}^+, A_{t+s} = A_s + A_t \circ \theta_s \mathrm{\ a.s. \ } P^m. 
\end{cases}
\end{equation}

\noindent If $Z \in p \mathcal{F}^0$ then by considering $Z$ of the form $\prod_{j=1}^n f_j \circ X_{t_j}$, with $f_j \in C_b(E)$, a monotone class argument shows that $(s, \omega) \rightarrow Z(\theta_s\omega)$ is $\mathcal{B}^+ \times \mathcal{F}_0$ measurable. In what follows a.e. refers to Lebesgue measure in the relevant dimensional Euclidean space. Define \begin{equation} \Omega_1 := \{ \omega: \mathrm{\ for \ a.e.\ } (s,t), 0 \le s \le t, A_{t-s}(\theta_s \omega)+A_s(\omega)= A_t(\omega)\} \end{equation} 
and
\begin{equation} \hspace{-1.4in}\Omega_0 =: \{ \omega: \theta_s \omega \in \Omega_1 \mathrm{ \ for \ a.e. \ } s \ge 0 \}. \end{equation}
\begin{align} \mathbf {Lemma.\  } & (i) \  [\Delta] \in \Omega_1 \mathit{\ and \ hence  \ in \ } \Omega_0. \ \mathit{If} \ \omega \in \Omega_0, \mathit{\ then \ } \theta_s \omega \in \Omega_0 \mathit{ \ for \ all\ } s \ge 0. \\
\notag & (ii) \mathit{ \  If \ } \theta_{t_m}\omega \in \Omega_0 \mathit{\ and \ } t_m \downarrow t \ge 0, \mathit{ \ then \ } \theta_t \omega \in \Omega_0. \\  \notag & (iii)  \mathit{ \ Both \ }  \Omega_0 \mathit{\ and \ } \Omega_1 \mathit{\ are \ } \mathcal{F}^* \mathit{ \ measurable \ }.
 \\
  \notag &  \mathit{(iv) \ } P^m(\Omega_0^c) = P^m(\Omega_1^c) =0. \end{align} 

\noindent Proof. Assertions (i) and (ii) are clear from the definitions and because $A_t([\Delta]) =0$ for all $t$. For (iii) and (iv), note that from the measurability argument just above and the right continuity of $t \rightarrow A_t(\omega)$, the set $\{(s,t,\omega: s \le t, A_t(\omega) \ne A_s(\omega) + A_{t-s}(\theta_s \omega)\}$ is $\mathcal{B}^+ \times \mathcal{B}^+ \times \mathcal{F}^0$ measurable. Let $J(s,t,\omega)$ be the indicator function of this set. Let $C_u(\omega):= \int_0^u \int_0^t J(s,t,\omega ds dt$, then $C_u \in \mathcal{F}^0$ and $u \rightarrow C_u(\omega)$ is continuous. Hence $C \in \mathcal{B}^+ \times \mathcal{F}^0$ where $\mathcal{B}^+$ is the Borel $\sigma$-algebra of $\mathbb{R}^+$. Therefore $T:= \inf\{u: C_u >0 \} \in \mathcal{F}^*$. See [DM IV, 50, 51]. But $\Omega_1 = \{ T = \infty \}$ and consequently $\Omega_1 \in \mathcal{F}^*$. Moreover \[ P^m(\Omega_1^c) = P^m(T< \infty) = P^m \{ \lim_{u\rightarrow \infty} C_u >0 \}.  \]

\noindent But (A.2 iii) and Fubini's theorem imply that for $P^m(\lim_{u \rightarrow \infty} C_u) =0$. Hence $P^m(\Omega_1^c) = 0$. \\

In order to show that $\Omega_0 \in \mathcal{F}^*$ we shall use a ``sandwiching'' argument that will be used again in what follows. Let $Q$ be an arbitrary probability on $(\Omega, \mathcal{F}^*)$ and define $\overline{Q} : = \int_0^\infty e^{-s} \theta_s Q ds$. Since $\Omega_1 \in \mathcal{F}^*$ there exist $\Lambda_1 \subset \Omega_1 \subset \Lambda_2$ with $\Lambda_i \in \mathcal{F}^0, i=1,2$ and $\overline{Q}(\Lambda_2 \ \setminus \ \Lambda_1) =0$. Now $(s, \omega) \rightarrow 1_{\Lambda_i}(\theta_s \omega)$ is $\mathcal{B}^+ \times \mathcal{F}^0$ measurable for $i =1,2$ and $Q \int_0^\infty e^{-s} 1_{\Lambda_2 \setminus \Lambda_1} \circ \theta_s ds=0$. Therefore $(s, \omega) \rightarrow 1_{\Omega_1} (\theta_s \omega)$ is $(\mathcal{B}^+ \times \mathcal{F}^0)^{\lambda \times Q}$ measurable---the completion of $\mathcal{B}^+ \times \mathcal{F}^0$ with respect to $\lambda \times Q$---where $\lambda$ is Lebesgue measure on $\mathbb{R}^+$.  Let $S : = \inf \{t: \int_0^t e^{-s}1_{\Omega^c_1}(\theta_s)ds>0\}$ and $S_i$ be defined similarly relative to $\Lambda_i^c$, $i=1,2$. Then, as in the case of $T$ in the previous paragraph, it follows that $S_i \in \mathcal{F}^*, i =1,2$. But $S_1 \le S \le S_2$ and  $\{S_1 < S_2\} \subset \{ \lambda(s:e^{-s}1_{\Lambda_2 -\Lambda_1} \circ \theta_s) >0 \}$. Therefore $Q(S_1 < S_2) =0$ and since $Q$ is arbitrary, $S \in \mathcal{F}^*$. But $\Omega_0= \{S = \infty\}$ and so $\Omega_0 \in \mathcal{F}^*$. Also \[ P^m (\Omega_0^c) = P^m(S < \infty) = P^m \{ \int_0^\infty 1_{\Omega_1^c}(\theta_s)ds > 0\} \] and $P^m \int_0^\infty 1_{\Omega_1^c}(\theta_s)ds = \int_0^\infty P^{mP_s} (\Omega_1^c)ds=0$ since $mP_s \le m$ for each $s$. Consequently $P^m(\Omega_0^c)=0$. Thus both (iii) and (iv) are established. \hfill{$\square$}

Define \begin{align} \Lambda_t(\omega) &:= \ess A_{t-s}(\theta_s \omega) \mathrm{ \ if \ } t >0 ; \\ 
& = \notag \ess \Lambda_s(\omega) \mathrm{ \ if \ } t =0. \end{align}

\noindent Since $A_t(\omega)$ is constant on $[\zeta(\omega), \infty[$, if $0 < s < t$ and $t \ge \zeta(\omega)$, then $t-s \ge (\zeta(\omega) - s)^+ = \zeta(\theta_s \omega)$. Hence $t \rightarrow \Lambda_t(\omega)$ is constant on $[\zeta(\omega), \infty[$. Clearly $\Lambda_t([\Delta]) =0$ for all $t \ge 0$. From (A.2 ii), if $0 \le s < t$, $A_{t-s} \in \mathcal{F}^0_{(t-s)+}$ and so $A_{t-s}\circ \theta_s \in \theta_s^{-1}\mathcal{F}^0_{(t-s)+} \subset \mathcal{F}_{t+}^0$.  Therefore the process $Z_{s}:=A_{t-s}\circ \theta_s$ defined on $[0,t[$ is adapted to the constant filtration $\mathcal{H}_s := \mathcal{F}_{t+}^0, 0 \le s < t$. It follows from [DM IV, 38.1a] that given an arbitrary probability $P$ on $(\Omega, \mathcal{F}_{t+}^0)$, $\Lambda_t \in \overline{\mathcal{F}^0_{t+}}^P$. Consequently $\Lambda_t \in \mathcal{F}_{t+}^*$, the universal completion of $\mathcal{F}_{t+}^0$, for $t>0$. Letting $t \downarrow 0$, $\Lambda_0 \in \mathcal{F}_{0+}^*$. Also $Z$ is adapted to the (reverse) filtration $\mathcal{H}_s := \mathcal{F}_{\ge s}^0$, $0 \le s <t$ by (A.2ii). Invoking [DM IV; 38.1a] again, given an arbitrary probability $P$ on $(\Omega, \mathcal{F}^0)$, the process $\overline{Z}_s := \mathrm{ess}\lim_{u\downarrow s} A_{t-u}\circ \theta_u$ is $P$ indistinguishable from an $(\mathcal{F}_{\ge s}^0)$ adapted process on $[0,t[$. In particular, it follows that $\overline{Z}_s \in \mathcal{F}_{\ge s}^*, 0 \le s < t$. Observe that \begin{align*} \Lambda_{t-s} (\theta_s) = \ess A_{t-s-u}(\theta_u \theta_s) \\ & = \essv A_{t-v}(\theta_v) = \overline{Z}_s. \end{align*}

\noindent Hence $\Lambda_{t-s} \circ \theta_s \in \mathcal{F}_{\ge s}^*, o \le s <t$. \\

The next lemma is a key technical point and is due to John Walsh [W72]. 

\begin{align}  \mathit{ Lemma.} & \mathit{  \vspace{-.5in} \ For \ fixed}  \ \omega \in \Omega_0,  \mathit{ \ the \ following \ obtain}: \\
& \notag \mathrm{(i)} \  s \rightarrow \Lambda_{t-s}(\theta_s\omega) \mathit{\ is \ right \ continuous \ and \ decreasing \ on \ } [0,t[; \\ & \notag \mathit{(ii)} \ \Lambda_{t-s} (\theta_s \omega) = A_{t-s} (\theta_s \omega) \mathit{ \ for \ a.e. } s<t ; \\ & \notag \mathrm{(iii)} \ t \rightarrow \Lambda_t(\omega) \mathit{\ is \ right \ continuous, \ increasing, \ and \ finite \ on } \  [0, \infty[.  \end{align}

\begin{proof}Note first of all, that identically on $\Omega$ for $0 \le s < t$ 
\begin{align} \Lambda_{t-s}\circ \theta_s & = \mathrm{ess}\ \underset{\epsilon\downarrow 0}{\lim} \sup
 A_{t-s-\epsilon} \circ \theta_{s+\epsilon} \\ \notag & =\mathrm{ess} \lim_{u\downarrow s} \sup A_{t-u}\circ \theta_{u} =  \lim_{u\downarrow s} \sup \Lambda_{t-u} \circ \theta_u \end{align}

\noindent since for any \[f : \mathbb{R}^+ \rightarrow \mathbb{R}, \overline{f}(t):= \mathrm{ess}\lim_{r\downarrow t} \sup f(r) = \lim_{r\downarrow t} \sup \overline{f}(r).\] See [DM IV; 37.7]. For fixed $\omega \in \Omega_0$, let $\mathcal{S}(\omega) := \{ v:A_t(\omega) = A_v(\omega) +A_{t-v}(\theta_v\omega)$ for all $t>v \}$.  If $\theta_s \omega \in \Omega$, then by Fubini's theorem and the right continuity of $A$, $\lambda(\mathcal{S}(\theta_s\omega)^c) =0$. Here, as above, $\lambda$ denotes Lebesgue measure on $\mathbb{R}^+$. Hence for a.e. $u>s$, $u-s \in \mathcal{S}(\theta_s\omega)$, that is for a.e. $u>s$ and all $t \ge u-s$, \[A_t(\theta_s \omega) = A_{u-s}(\theta_s\omega) + A_{t-(u-s)}(\theta_u\omega).\]

\noindent If, in addition, $t-s > u-s$, then replacing $t$ by $t-s$ in the last display we obtain 

\[ A_{t-s}(\theta_s \omega) = A_{u-s}(\theta_s\omega) + A_{t-u}(\theta_u\omega) \mathrm { \ for \ a.e.} (s,u), s< u < t. \]

\noindent But $A_{t-u}(\theta_u \omega) \ge 0$ and so $A_{t-s} (\theta_s \omega) \ge A_{u-s} (\theta_s \omega)$ for a.e. $(s,u)$ with $s \le u \le t$. By Fubini's theorem for a.e. $s$, $u \rightarrow A_{t-u}(\theta_u \omega)$ is decreasing on $]s,t]$. Hence by the equality of the first and last terms in (A.8) yields (i). Now Lebesgue's differentiation theorem and the fact that for a.e. $s,u \rightarrow A_{t-u}(\theta_u \omega)$ is decreasing a.e. on $]s,t[$ imply that 

\begin{align*} \mathrm{ess} \lim_{u\downarrow s}  A_{t-u} (\theta_u \omega) & = \mathrm{ess} \lim_{u\downarrow s}  A_{t-u} (\theta_u \omega) \\
& = \lim_{\epsilon \downarrow 0} \frac{1}{\epsilon} \int_s^{s+\epsilon} A_{t-u} (\theta_u \omega) du = A_{t-s}(\theta_s) \end{align*}
for a.e. $s < t$, and so (ii) folllows from (A.8). Thus (i) and (ii) hold. Since $t \rightarrow A_t(\omega)$ is finite and increasing on $[0,\infty[$ for all $\omega$, so is $t \rightarrow \Lambda_t(\omega)$. For $\omega \in \Omega_0$, the definition (A.6) of $\Lambda_t$ with $t$ replaced by $t+ \epsilon$, gives
\begin{align*} \Lambda_{t+\epsilon}(\omega) & = \essu0 A_{t+\epsilon -u}(\theta_u \omega)  = \mathrm{ess} \lim_{u\downarrow 0} A_{t+\epsilon -u}(\theta_u \omega) \end{align*}

\noindent where the last equality obtains because $u \rightarrow A_{t+\epsilon -u}(\theta_u\omega)$ is decreasing a.e on $[0,t+\epsilon[$ for $\omega \in \Omega_0$ as was shown during the proof of (i). Since $\omega \in \Omega_0$, $\theta_u \omega \in \Omega_0$ for all $u$. Then $v-u \in \mathcal{S}(\theta_u \omega)$ for a.e. $v >u$. If $t > v, t-u > v-u$ and so for a.e. $v, u <v<t, A_{t-u} (\theta_u \omega) = A_{v-u}(\theta_u \omega) + A_{t-u-(v-u)}(\theta_{v-u}\theta_u \omega) = A_{v-u}(\theta_u \omega) + A_{t-v}(\theta_v \omega)$. But this also holds with $t$ replaced by $t+\epsilon$. Therefore for a.e. $v$, $u < v <t$ \begin{align*} A_{t+\epsilon -u}(\theta_u \omega) & = A_{t-u}(\theta_u \omega) +A_{t+\epsilon -u}(\theta_u \omega) - A_{t-u}(\theta_u \omega) \\ & = A_{t-u} (\theta_u \omega) + A_{t+\epsilon -v} (\theta_v \omega) -A_{t-v}(\theta_v \omega). \end{align*}

\noindent Now fix such a $v$ and substitute the above into the previous display. Taking essential limits as $u \downarrow 0$ we obtain \[\Lambda_{t+\epsilon}(\omega) = \Lambda_t(\omega) + A_{t+\epsilon -v}(\theta_v \omega) - A_{t-v}(\theta_v \omega). \] 
But $A$ is right continuous and so $\Lambda_{t+\epsilon}(\omega) \rightarrow \Lambda_t(\omega)$ as $\epsilon \rightarrow 0$. Thus $t\rightarrow \Lambda_t(\omega)$ is right continuous on $]0,\infty[$ and $\Lambda_t (\omega) \rightarrow \Lambda_0(\omega)$ as $t \downarrow 0$ by definition, completing the proof of (A.7).

\end{proof}

We return now to the proof of (5.27). Fix $t>0$ and $ \omega \in \Omega_0$. Using (A.7 ii), $\Lambda_{t-s}(\theta_s\omega) = A_{t-s}(\theta_s \omega)$ for a.e $s<t$. Fix such an $s$ with $0 < s < t$. Then \[A_{t-\epsilon}(\theta_\epsilon \omega) = A_{s-\epsilon}(\theta_\epsilon \omega) + \Lambda_{t-s}(\theta_s \omega) \mathrm{ \ for \  a.e.} \  \epsilon < s\]
and taking essential limits as $\epsilon \downarrow 0$ we obtain $\Lambda_t(\omega) = \Lambda_s(\omega)+\Lambda_{t-s}(\theta_s\omega)$ for a.e. $s<t$ and $\omega \in \Omega_0$. Then (i) and (iii) of (A.7) imply that this last equality holds for all $s<t$ and $\omega \in \Omega_0$. But $t>0$ is arbitrary and so $\Lambda_{t+s}(\omega) = \Lambda_s(\omega)+\Lambda_t(\theta_s \omega)$ for $\omega \in \Omega_0$ and by (A.7 iii) this holds for $t=0$ also. \\

\noindent Next note that from the right continuity of $A$, (5.27 iii) and Fubini's theorem imply that for each fixed $t>0$ for $P^m$ a.e $\omega$, $A_{t-s}(\theta_s\omega) =A_t(\omega)-A_s(\omega)$ for a.e $s \in ]0,t[$. Since $A_s(\omega) \rightarrow 0$ as $s \downarrow 0$ we see that $\ess A_{t-s}\circ \theta_s = A_t$ a.s. $P^m$. Combining this with (A.6), $\Lambda_t=A_t $ $P^m$ a.s. for each fixed $t>0$, and hence for all $t\ge0$ since $A$ is right continuous on $\Omega$ and $\Lambda$ is right continuous on $\Omega_0$. Consequently $\Lambda$ and $A$ are indistinguishable and $\Lambda_{t+s}(\omega) = \Lambda_s(\omega) +A_t(\theta_s \omega)$ on $\mathbb{R}^+ \times \mathbb{R}^+ \times \Omega_0$ Moreover $\Lambda$ is $(\mathcal{F}^\star_{t})$ adapted and $\Lambda_{t-s}\circ \theta_s \in \mathcal{F}^*_{\ge s}$, $0 \le s < t$.  \\

Finally at the cost of destroying the good measurability of $\Lambda$ we shall remove the exception set $\Omega \setminus \ \Omega_0$. Define $R(\omega) := \inf \{ t: \theta_t\omega \in \Omega_0\}.$ In view of (A.5 ii), $\theta_t \omega \in \Omega_0$ for all $t \ge R(\omega)$ and so $R = \inf \{t: \int_0^t 1_{\Omega_0} (\theta_s \omega) ds > 0 \}.$ Then a ``sandwiching'' argument shows that $R \in \mathcal{F}^*$---see the proof that $S$ defined in the proof of (A.5 iii) is $\mathcal{F}^*$ measurable. Since $\Omega_0 = \{ R=0\}$ by (i) and (ii) of (A.5), $P^m (R>0) = P^m (\Omega_0^c) =0$. Define \begin{equation}\begin{cases} A_t^*(\omega):= 0 \quad \mathrm{if}\  t< R(\omega) \\ A_t^*(\omega) := \Lambda_{t-R(\omega)}(\theta_R \omega) \quad \mathrm{if} \ t\ge R(\omega) \end{cases}\end{equation}

\noindent Noting that $\theta_R \omega \in \Omega_0$ if $R(\omega) <\infty$ and that $R\circ \theta_s = s + R \circ \theta_s$ if $s < R$ while $R \circ \theta_s =0$ if $s \ge R$, one readily verifies that $A_t^*(\omega) = A^*_s(\omega)+A^*_{t-s}(\theta_s\omega)$ identically in $0 \le s \le t$, $\omega \in \Omega$. Since $P^m(R>0)=0$, $\Lambda$ and $A^*$, and hence $A^*$ and $A$, are indistinguishable. Moreover $t \rightarrow A_t^*(\omega)$ is right continuous for all $\omega$ since $\theta_R \omega \in \Omega_0$ if $R(\omega) < \infty$. But $R$ is only $\mathcal{F}^*$ measurable and so one only knows that $A_t^* \in \mathcal{F}^*$. Of course, since $P^m(\Omega \setminus \ \Omega_0) =0$, $A^*$ is $(\mathcal{F}^m_t)$ adapted. It seems as though we are very close to establishing (5.27). For example it would suffice to show that $R$ is an $(\mathcal{F}^*_{t+})$ stopping time. But this is not at all clear, at least to me. \\

Therefore it is necessary to proceed in a different manner. We shall follow the argument in Appendix $A$ of [G90] which, we should emphasize, comes from [Me74]. It is convenient to summarize the properties of $\{\Lambda_t; t \ge 0\}$ that have been established so far for ease of reference: 
\begin{align}\mathrm{(A.10)} \quad \notag (i)  \ &   \Lambda_t \in  \mathcal{F}^*_{t+} \mathrm{ \ if\ } t \ge 0 \mathrm{\  and} \ \Lambda_{t-s}\circ \theta_s \in \mathcal{F}^*_{\ge s} \mathrm{ \ if \ }  0 \le s <t;   \\ \notag
 (ii) \ & \mathrm{  for \ every \ } \omega \in \Omega, t \rightarrow \Lambda_t(\omega) \mathrm{ \ is \ increasing, \ finite \ on } \ [0,\infty[ \mathrm{ \ and \ constant \ on} [\zeta(\omega), \infty[; \\ \notag (iii) \ & \Omega_0 \in \mathcal{F}^* \mathrm{ \ where \ } \Omega_0 \mathrm{  \ is \ defined \ in \ (A.4)} ;\\
 \notag (iv) \ & [\Delta] \in \Omega_0 \mathrm{ \ and \ } \Lambda_t([\Delta]) = 0 \mathrm{ \ for \ all \ } t \ge 0;
 \\ (v) & \ P^m(\Omega \setminus \ \Omega_0) =0 ; \notag \\
  (vi) & \ t_n \downarrow t, \theta_{t_n} \in \Omega_0 \mathrm{ \ for \ each\ } n \mathrm{ \ implies \ that \ } \theta_{t}\omega \in \Omega_0, t\ge0, \mathrm{ \ and \ } \omega \in \Omega_0 \Rightarrow \theta_t \omega \in \Omega_0, t \ge 0 \notag ; \\
 (vii) & \ \Lambda_{t+s}(\omega) = \Lambda_t(\omega) + \Lambda_s(\theta_t \omega) \mathrm{ \ for \ } \omega \in \Omega_0 \mathrm{\ and \ } s,t \ge 0 \notag ; \\
 (viii) & \ \mathrm{if} \ \omega \in \Omega_0,t \rightarrow \Lambda_t(\omega) \mathrm{\ is \ right \ continuous, \ increasing \ and \ finite \ on \ } [0,\infty[ \mathrm{\ and \ } s \rightarrow \Lambda_{t-s}(\theta_s \omega) \notag \\ &  \mathrm{ \ is \ right \ continuous, \ decreasing \ and \ finite \ on } \ [0,t[ \notag;\\
 (ix) & \ \Lambda \mathrm{ \ and \ } A \mathrm{\ are \ indistinguishable. } \notag \end{align}
 
 
 \setcounter{equation}{10}
 
 At this point we are going to make an additional assumption; namely that $A$ is continuous at $\zeta$. More precisely: 
 \begin{equation} \mathit{For \ } \omega \in \Omega, \mathit{ \ if \ } \zeta(\omega) < \infty, \mathit{ \ then \ } A_t(\omega) = A_{\zeta-}(\omega) \mathit{ \ for \ all \ } t \ge \zeta(\omega). \end{equation}
 
\noindent However (A.11) will be removed in the final step of the proof of Theorem 5.27. Note that (A.11) does not follow automatically if in hypothesis (5.27 i) right continuous is replaced by continuous because in the first step of the proof we replaced the $A$ in the hypothesis of (5.27) by $\tilde{A}$ which is only $P^m$ indistinguishable from the original $A$.  It is this $\tilde{A}$, now being called $A$, that is assumed to satisfy (A.11). \\
 
 \textit{Condition (A.11) is in force until further notice.} \\
  \begin{align} & \hspace{-.55in}  \mathbf{Lemma}\ \hspace{.15in} \mathrm{For\ } \omega \in \Omega_0, k_u \omega \in \Omega_0 \mathrm{ \ for \ each \ } u \ge 0 \hspace{1in}  \end{align}
  
  \begin{proof} Fix $\omega \in \Omega_0$ and $u>0$ ($k_0 \omega = [ \Delta ] \in \Omega_0$). From the definitions of $\Omega_0$ (A.4) and $\Omega_1$ (A.3) one must show that for a.e $s, \theta_sk_u \omega \in \Omega_1$. But $\theta_s k_u \omega = k_{u-s} \theta_s \omega$ if $s<u$ and $\theta_s k_u \omega = [\Delta]$ if $s \ge u$. Since $[\Delta] \in \Omega_1$, we may suppose that $s<u$ and and since $\theta_s \omega \in \Omega_0$ by (A.10 vi), it certainly suffices to show that if $\omega \in \Omega_0$, $k_u \omega \in \Omega_1$ for all $u>0$. That is,  for a.e. $(s,t), 0 \le s \le t$ one has 
  \begin{equation} A_{t-s}(\theta_s k_u \omega) + A_s(k_u \omega) = A_t(k_u \omega), \end{equation}
  and using the Fubini theorem and the right continuity of $t \rightarrow A_t$ this amounts to showing that (A.13) holds for a.e. $s \le t$ for each $t>0$. Suppose first that $t<u$. Since $\zeta(k_u\omega)=\zeta(\omega) \wedge u$ and $A_{t-s} \circ \theta \in \mathcal{F}^0_{t+}$, the conclusion is clear. Next suppose that $0 < u \le t$ and let $z:= \zeta(k_u \omega) \le u$. If $z \le s$, then $\theta_s k_u \omega = [\Delta]$ since $k_u \omega(r) = \Delta$ where $ r \ge z$. Therefore $A_{t-s}(\theta_sk_u \omega) = A_{t-s}([\Delta]) =0$, while $A_s(k_u \omega) =A_z(k_u \omega) = A_t(k_u \omega)$ since $A$ is constant on $[\zeta, \infty[$. Hence (A.13) holds identically in $s$ in this range. Finally consider the case $s < z$. Choose $v$ with $s < v < z \le u$. Then by the first case above, $A_{v-s}(\theta_s k_u \omega) + A_s (k_u \omega) = A_v (k_u \omega)$ for $v>s$. Let $v \uparrow z$ to obtain 
  \begin{equation}A_{(z-s)-} (\theta_s k_u \omega) +A_s (k_u \omega) = A_{z-} (k_u \omega) \end{equation}
 
 for $s<z$. But if $s<z = \zeta(k_u\omega)$, then $z-s = \zeta(\theta_s k_u \omega)$. Since $t \ge u$ one has $ t \ge z$ and $t-s \ge z-s = \zeta(\theta_s k_u \omega)$, and one concludes from (A.11) and (A.10i) that (A.14) holds with $(z-s)-$ and $z-$ replaced by $t-s$ and $t$ respectively.  These cases taken together show that (A.13) holds. Hence $k_u \omega \in \Omega_0$.  
\end{proof}

We are now going to remove the expectional set $\Omega \setminus \ \Omega_0$ while preserving the properties of $\Lambda := \{ \Lambda_t, t \ge 0\}$ listed in (A.10). To this end given $\omega \in \Omega$ and $t >0$ call $\tau : = \{ (s_i, t_i): 1 \le i \le n\}$ a good partition for $\omega$ of $[0,t]$ (g.p. $\omega$/$t$) provided: 

\begin{align} (i) & \ 0 \le s_1 < t_1  \le s_2 < t_2 \le \dots \le s_n < t_n \le t; \\
(ii) \notag & \ k_{u -s_i} \theta_{si} \omega \in \Omega_0 \mathrm{ \ for \ some \ } u> t_i, 1 \le i \le n.\end{align} 

\noindent If, in addition, each $s_i$ and $t_i$ is rational except for $t_n$ in case $t_n =t$ we shall define $\tau$ to be rational; that is $\tau$ is rational provided $\tau \cap [ 0,t[ \subset \mathbb{Q}$. We shall use $\tau$ both for the collection of ordered pairs and the set $\{s_i, t_i; 1 \le i \le n\}$. Given $\omega \in \Omega$ and $>0$ define 
\begin{align} J_t^\tau(\omega) & := \sum_{i=1}^n \Lambda_{t_i -s_i} (\theta_{s_i}\omega); \tau \  \mathrm{a  \ g.p.}\omega / t; \\
J_t(\omega) & := \sup_\tau J_t^\tau(\omega) \end{align}

\noindent where the supremum in (A.17) is over all g.p. $\omega/ t$. By convention, if there are no g.p $\omega/t$ we say that the empty partition is a g.p. $\omega/ t$ and set $J_t^\tau(\omega) =0$ if $\tau$ is empty. Observe that if $\tau$ is a g.p $\omega/t$ then $\tau$ is a g.p $\omega/t'$ for $t' >t$. Hence for all $\omega \in \Omega$, $t \rightarrow J_t(\omega)$ is increasing and $J_t(\omega) \ge 0$. Note $J_t$ is defined only when $t>0$. 
  \begin{align}&  \hspace{-.5in} \mathbf{Lemma}\  \mathit{If} \ \omega \in \Omega_0  \mathit{ \ and \ } t >0, \mathit{ \ then \ } \Lambda_t(\omega) = J_t(\omega) \hspace{1in}  \end{align}
  
  \begin{proof} By (A.12) and (A.10 vi), $\tau = \{ (0,t)\}$ is a g.p $\omega$/$t$, and so $J_t(\omega) \ge \Lambda_t(\omega)$. In what follows $\omega \in \Omega_0$ is suppressed in our notation. Given a g.p $\omega/ t$, $\tau = \{ (s_i, t_i); 1 \le i \le n\}$ using (A.10 vii), 
  
  \begin{align*} \Lambda_t &= \Lambda_{s_1}+ \Lambda_{t-s_1}(\theta_{s_1}) \ge  \Lambda_{t-s_1}(\theta_{s_1}) \\
  & = \Lambda_{(t_1 - s_1) + (t-t_1)}(\theta_{s_1}) = \Lambda_{t_1 -s_1}(\theta_{s_1}) + \Lambda_{t-t_1}(\theta_{t_1}) \end{align*}
  and
  \begin{align*} \Lambda_{t-t_1}(\theta_{t_1}) & = \Lambda_{s_2 - t_1}(\theta_{t_1}) + \Lambda_{t-s_2}(\theta_{s_2}) \\ & \ge \Lambda_{t-s_2}(\theta_{s_2})= \Lambda_{t_2-s_2}(\theta_{s_2}) + \Lambda_{t-t_2}(\theta_{t_2}). \end{align*}

\noindent Hence $\Lambda_t \ge \Lambda_{t_1 -s_1}(\theta_{s_1}) + \Lambda_{t_2 -s_2}(\theta_{s_2}) + \Lambda_{t-t_2}(\theta_{t_2})$ and proceeding inductively $\Lambda_t \ge J_t^\tau$. Therefore $\Lambda_t \ge J_t$ and so $\Lambda_t = J_t$.  \end{proof}
  \begin{equation} \mathbf{Lemma. \ } J_{t+s}(\omega) = J_t(\omega) + J_s(\theta_t \omega) \mathit{ \ for \ all \ } \omega \in \Omega \mathit{ \ and \ all \ } s,t >0. \end{equation}

\begin{proof} Fix $s,t >0$ and $\omega \in \Omega$. Unless required for clarity we again suppress $\omega$ in the notation. Let $\tau$ be a g.p $\omega/ t$ and $\tau'$ a g.p $\theta_t$ $\omega/s$. Then one easily checks that $\tau'':= \tau \cup (\tau' +t)$ is a g.p.$\omega/t+s$. To show $J_{t+s} \ge J_t+J_s(\theta_t)$ it suffices to suppose that $J_{t+s} < \infty$. Then given $\epsilon >0$ we may choose $\tau$ and $\tau'$  so that $J_t^\tau \ge J_t - \epsilon$ and $J_s^{\tau'}(\theta_t) \ge J_s(\theta_t) -\epsilon$. Hence
\[J_{t + s} \ge J_{t+s}^{\tau''} = J_t^\tau +J_s^{\tau'}(\theta_t) \ge J_t + J_s(\theta_t) - 2 \epsilon\]
and so $J_{t+s} \ge J_t + J_s(\theta_t)$. For the opposite inequality one may suppose that $J_t +J_s(\theta_t) < \infty$.  Then fix $\epsilon >0$ and choose a g.p $\omega/t+s$, $\tau = \{ (s_i,t_i); 1 \le i \le n\}$ such that $J_{t+s}^\tau \ge J_{t+s} - \epsilon$. We consider two cases. \\

\noindent Case 1: It is \textit{not} the case that $s_k < t < t_k$ for some $k$, $1 \le k \le n$. If $t> t_n$ let $\tau' = \tau$ and $\tau''$ be empty, while if $t < s$, let $\tau'$ be empty and $\tau'' = \tau -t$. If $t = s_k$ for some $k$, let $\tau' =\{ (s_i,t_i): 1 \le i < k\}$ and $\tau''= \{ (s_i-t, t_i-t); k \le i \le n\}$ while if $ t = t_k$ for some $k$, let $\tau' = \{ (s_i,t_i) ; 1 \le i \le k\}$ and $\tau'' = \{ (s_i-t, t_i-t); k < i \le n\}$. Note that $\tau'$ is empty if $t \le s_1$, and $\tau''$ is empty if $t \ge t_n$. It is easily checked that $\tau'$ is a g.p. $\omega/t$ and that $\tau''$ is a g.p. $\theta_t$ $\omega /s$. Therefore \[J_t+J_s(\theta_t) \ge J_t^{\tau'} +J_s^{\tau''}(\theta_s) = J_{t+s}^\tau \ge J_{t+s} - \epsilon,\] and so $J_t +J_s(\theta_t) \ge J_{t+s}$, establishing (A.19) in this case.

\noindent Case 2: For some $k$, $s_k < t < t_k$. In this case let \[\tau' := \{ (s_i, t_i); 1 \le i < k \} \cup \{ (s_k,t)\} \]
\[ \tau '' : = \{ (0,t_k-t)\} \cup \{ (s_i-t, t_i -t) ; k+1 \le i < n \}. \]

\noindent Since $\tau$ is a g.p. $\omega / t+s$, there exists $u > t_k$ such that $k_{u-s_k} \theta_{s_k} \omega \in \Omega_0$. But $t_k > t$ and so $\tau'$ is a g.p. $\omega/t$. Also using (A.10 vi) 
\[ k_{u-t}\theta_t \omega = k_{u-t} \theta_{t-s_k}\theta_{s_k} \omega = \theta_{t-s_k}k_{u-s_k} \theta_{s_k} \omega \in \Omega_0. \]

\noindent Since $u-t > t_k -t$ this implies that $\tau''$ is a g.p. $\theta_t \omega/s$. Now $u>t_k > t > s_k$ and so we compute as follows where the second equality below holds because $k_{u-s_k}\theta_{s_k} \omega \in \Omega_0$ where $\omega \in \Omega$ is fixed but suppressed in our notation while the first and fourth inequality below use the fact that if $v > r, \Lambda_r \circ k_v = \Lambda_r$ since $\Lambda_r \in \mathcal{F}^*_{r+}$ and $\Lambda_r ([\Delta]) =0$ (see lemma 5.33):

\begin{align*} \Lambda_{t_k-s_k}(\theta_{s_k}) & = \Lambda_{t-s_k + (t_k -t)}(k_{u-s_k}\theta_{s_k}) \\
& = \Lambda_{t-s_k}(k_{u-s_k}\theta_{s_k}) + \Lambda_{t_k-t}(\theta_{t - s_k} k_{u-s_k}\theta_{s_k}) \\
& = \Lambda_{t-s_k}(\theta_{s_k}) + \Lambda_{t_k-t}(k_{u-t}\theta_t) \\
& = \Lambda_{t-s_k}(\theta_{s_k}) + \Lambda_{t_k-t}(\theta_t). 
\end{align*}

\noindent Consequently $J_{t+s}^\tau = J_t^{\tau'} +J_s^{\tau''}(\theta_t)$, and now arguing as in Case 1 we obtain $J_{t+s} \ge J_t +J_s(\theta_t)$ completing the proof of (A.19).
\end{proof}
\begin{align} \notag \mathrm{(A.20)} \quad & \mathbf{Lemma.} \mathit{ \ Fix \ }  \omega \in \Omega \mathit{ \ and \ } 0 \le s < t. \mathit{ \ If \ there \ exists\ a \ } v>t \mathrm{ \ with \ } k_{v-s}\theta_s \omega \in \Omega_0, \mathit{ \ then \ } \\ & u \notag
\rightarrow \Lambda_{t-u}(\theta_u\omega) \mathit{ \ is \ decreasing \ and \ right \ continuous \ on \ } [s,t[. \end{align}
  
 \begin{proof} If $u \in [s,t[$, then using (5.33) \begin{align*} \Lambda_{t-u}(\theta_u \omega) & = \Lambda_{t-u}(k_{v-u} \theta_u \omega) = \Lambda_{t-u}(k_{v-u}\theta_{u-s}\theta_s \omega) \\ 
 &= \Lambda_{t-u}(\theta_{u-s}k_{v-s}\theta_s \omega) = \Lambda_{(t-s)-(u-s)}(\theta_{u-s}\omega') 
 \end{align*}
 
\noindent  where $\omega' = k_{v-s}\theta_s \omega \in \Omega_0$. The conclusion is now immediate from (A.10 viii). \end{proof}

 The next lemma will enable us to establish the good measurability of $J_t$ and $J_{t-s}\circ \theta_s$. 
 
 \begin{align} \notag \mathrm{(A.21)} \quad& \mathbf{Lemma}. \ \mathit{ \ For \ each\ } t > 0 \mathit{ \ and \ } \omega \in \Omega, J_t(\omega) = \sup_\tau J_t^\tau(\omega) \mathit{ \ where \ the \ supremum \ is \ over} \\ & \notag \mathit{ all \ rational \ g.p.\ }\omega/t.\end{align}
 
  \setcounter{equation}{21}

 \begin{proof}Let $\mathcal{R}(t)$ denote the collection of all rational g.p.$\omega/t$, and note that $\mathcal{R}(t)$ is countable. Trivially $J_t(\omega) \ge \sup\{J_t^\tau(\omega); \tau \in \mathcal{R}(t)\}$. For the opposite inequality fix $\epsilon >0$ and choose a g.p.$\omega / t$, $\tau = \{ (s_i,t_i); 1 \le \i \le n\}$ such that $J_t^\tau(\omega) \ge J_t(\omega)-\epsilon$. Using Lemma A.20, for each $i$ we may choose $r_i \in ]s_i,t_i[$ with $r_i$ rational and such that \[ \Lambda_{t_i-r_i}(\theta_{r_i}\omega) \ge \Lambda_{t_i-s_i} (\theta_{s_i}\omega) - \epsilon / n.\]
 
\noindent  For each $i \le n$, choose $u_i > t_i$ such that $k_{u_i-s_i}\theta_{s_i} \omega \in \Omega_0$, and then choose a rational $q_i \in ]t_i,r_{i+1}[$ (if $t=n$ choose a rational $q_n \in ]t_n,t[$ unless $t_n=t$ in which case set $q_n =t_n=t$) such that $q_i < u_i$. This is possible since $\tau$ is a g.p.$\omega /t$. By (A.10 ii) $\Lambda_{q_i-r_i}(\theta_{r_i}\omega) \ge \Lambda_{t_i-r_i}(\theta_{r_i}\omega)$, and using (A.10 vi) $k_{u_i -r_i}(\theta_{r_i}\omega) = k_{u_i-r_i}(\theta_{r_i-s_i}\theta_{s_i}\omega) = \theta_{r_i-s_i}k_{u_i-s_i}(\theta_{s_i}\omega) \in \Omega_0.$ \\
 
\noindent  Therefore $\tau' := \{ (r_i, q_i); 1 \le i \le n\} \in \mathcal{R}(t)$, and \begin{align*}J_t^{\tau'}(\omega) & = \sum_{i=1}^n \Lambda_{q_i-r_i}(\theta_{r_i} \omega) \ge \sum_{i=1}^n \Lambda_{t_i -r_i} (\theta_{r_i}\omega) \\
 & \ge \sum_{i=1}^n \left (\Lambda_{t_i -s_i}(\theta_{s_i}\omega)-\epsilon/n \right) \\
 & \ge J_t^\tau(\omega) - \epsilon \ge J_t(\omega) -2 \epsilon. \end{align*}
 
\noindent Hence $\sup\{J_t^\tau(\omega);\tau \in \mathcal{R}(t) \} \ge J_t(\omega)$ establishing (A.21).  \end{proof}
 
 In light of (A.10i) and (A.16) it follows from (A.21) that $J_t \in \mathcal{F}_{t+} ^* \subset \mathcal{F}^*, t>0$. Therefore $J_{t-s}\circ \theta_s \in \theta_s^{-1} \mathcal{F}^* \subset \mathcal{F}_{\ge s}^*$. Moreover identically on $\Omega$, $J_{t+s} = J_t + J_s \circ \theta_t$ for $s,t > 0$ according to (A.19). Since $P^m (\Omega \ \setminus \ \Omega_0) =0$, (A.18) implies that $\Lambda$ and $J$ are indistinguishable on $]0,\infty[$, and hence the three processes $A, \Lambda$, and $J$ are indistinguishable on $]0,\infty[$. \\
 
 Next we modify $J$ to render it right continuous while preserving the good measurability and shift properties of $J$. Clearly $t \rightarrow J_t(\omega)$ is increasing on $]0,\infty[$ for each $\omega \in \Omega$, and since $J$ are $A$ are indistinguishable, for $P^m$ a.e. $\omega$, $t \rightarrow J_t(\omega)$ is finite on $]0,\infty[$. Of course $J \ge 0$. Consequently $J_{t+}(\omega)$ exists for $t>0$ and $0 \le J_{t+}(\omega) \le \infty$ for each $\omega \in \Omega$. Define $J_0 := 0$ and note that (A.19) implies that $J_{(t+s)+}(\omega) = J_t(\omega)+ J_{s+}(\theta_t \omega)$  for $s,t \ge 0$ and $\omega \in \Omega$. In what follows $\omega \in \Omega$ is suppressed in the notation. Define for $t > 0$
 
 \begin{align} K_t & = J_t - \sum_{0 \le s <t} (J_{s+}-J_s) \quad \mathrm{if \ } J_{t+} < \infty \\ 
 K_t & = \infty \notag \quad \mathrm{if \ } J_{t+}=\infty. \end{align}
 
\noindent Since $J_{s+} - J_s \le J_{s+} - J_{s-}$ the sum in (A.22) is dominated by the sum of jumps of $J$ on the interval $[0,t[$, and hence by $J_t \le J_{t+} < \infty$. Therefore $K_t \ge 0.$ It is easy to check that $t \rightarrow K_t$ is right continuous on $]0,\infty[$. We assert that it is also increasing on $]0,\infty[$. Let $K^\epsilon$ be defined as in (A.22) except that the sum is taken only over the finite number of $s \in [0,t[$ for which  $J_{s+} - J_s > \epsilon$. Then $K_t^\epsilon \downarrow K_t$. Fix $0<u<t$. If $J_{t+} = \infty, K_t = \infty$ and so $K_u \le K_t$. If $J_{t+} < \infty$, let $s_1 < s_2 < \dots < s_n$ be the finite number of $s \in [u,t[$ at which $J_{s+} - J_s > \epsilon$. Then $J_u \ge J_{t+} <\infty$ and 
 
 \begin{align} K_t^\epsilon - K_u^\epsilon & = J_t - J_u - \sum_{i=1}^n (J_{s_i+} - J_{s_i}) \\ \notag 
 &= J_{s_1}  - J_u + \sum_{i=1}^{n-1} (J_{s_{i+1}} - J_{s_i+}) + J_t - J_{s_n+} \ge 0. \end{align}
 
\noindent  Therefore $K_t^\epsilon \ge K_u^\epsilon$ and letting $ \epsilon \downarrow 0$ it follows that $K_t \ge K_u$. Hence $t \rightarrow K_t$ is increasing on $]0,\infty[$. If $J_t =\infty$ for all $t>0$, then $K_t = \infty$ for all $t>0$. On the other hand if $J_v < \infty $ for some $v>0$, then for $0< t< v$ \[\sum_{0 \le s < t} (J_{s+}-J_s) \le J_t < \infty. \]
 
 \noindent But as $t \downarrow 0$ this sum decreases to $J_{0+}-J_0 = J_{0+} < \infty$. Hence $K_{0+} =0$ in this case. Thus $K_{0+} =0$ or $K_{0+} = \infty$ according as $J_t = \infty $ for all $ t>0$ or $J_t < \infty$ for some $t>0$. Define $K_0 := K_{0+}$ so that $t \rightarrow K_t$ becomes right continuous on $[0,\infty[$. If $\omega \in \Omega_0$, $J_t (\omega) = \Lambda_t(\omega) < \infty$. Consequently if $\omega \in \Omega_0$, $K_0(\omega) =0$ and $K_t(\omega) \le J_t(\omega)<\infty$. \\
 
 We claim next that \begin{equation} K_{t+s}(\omega) = K_t(\omega) + K_s(\theta_t \omega) \mathrm{ \ for \ } s, t \ge 0, \omega \in \Omega_0. \end{equation}
 
\noindent If $\infty = J_{(s+t)+ } = J_t + J_{s+}(\theta_t)$, then either $J_t = \infty$ or $J_{s+}(\theta_t) = \infty$ and so (A.24) follows immediately from the definition (A.22) of $K$. If $J_{(t+s)+} < \infty$, then both $J_t$ and $J_{s+}(\theta_t) < \infty$. In this case $K_0 = 0$ and so (A.24) holds for $t=0$ and all $s \ge 0$. If both $t>0$ and $s>0$, then using (A.19)
 
 \begin{align*} K_{t+s} & = J_t + J_s(\theta_t) - \sum_{0 \le u < s +t} (J_{u+}-J_u) \\ 
 & = J_t - \sum_{0 \le u < t} (J_{u+} - J_u) + J_s(\theta_t) - \sum_{0 \le u < s} (J_{(u+t)+} - J_{u+t}) \\
 & = K_t + K_s(\theta_t), \end{align*}
 
\noindent and letting $s \downarrow 0$ we obtain (A.24) in all cases. Clearly $K$ inherits the property $K_t \in \mathcal{F}^*_{t+}$ from $J$ and so $K_{t-s} \circ \theta_s \in \mathcal{F}^*_{\ge s}$ for $0 \le s <t$. By (A.18), $J_t(\omega) = \Lambda_t(\omega)$ if $\omega \in \Omega_0$ and since $t \rightarrow \Lambda_t(\omega)$ is right continuous if $\omega \in \Omega_0$ it follows from the definition of $K$ that $K_t(\omega) = J_t(\omega) = \Lambda_t(\omega)$ if $\omega \in \Omega_0$. Therefore the four processes $K, J, \Lambda$ and $A$ are indistinguishable. Thus $K$ has all the properties $L$ is asserted to have in (5.27) expect the last assertion in (5.27 v) in that we have only shown that $K_{t-s} \circ \theta_s \in \mathcal{F}_{\ge s}^*$ rather than $\mathcal{F}_{>s}^*$. The problem is that $t \rightarrow K_t$ is finite on $[0,\infty[$ only $P^m$ a.s., and so $K_{t-s}(\theta_s) = K_t - K_s$ makes sense only $P^m$ a.s. Of course if $t>0$ and $K_t(\omega) < \infty$, then $s \rightarrow K_{t-s}(\theta_s \omega)= K_t(\omega) -K_s(\omega)$ is right continuous on $[0,t[$. Therefore it is necessary to make on last modification to arrive at the desired functional. \\
 
 If $0<s<t$, define $\tilde{K}_{t-s}(\theta_s) = K_t - K_s$ where $\infty - \infty =0$. Since $K_t \ge K_s$, $\tilde{K}_{t-s}(\theta_s) =0$ when $K_s = \infty$ and so $s \rightarrow \tilde{K}_{t-s}(\theta_s)$ is decreasing on $]0,t[$. Therefore we may define for $t>0$ and $\omega \in \Omega$, \begin{equation}L_t(\omega) = \uparrow \lim_{s\downarrow 0} \tilde{K}_{t-s} (\theta_s \omega). \end{equation}
 
 \noindent If $\omega \in \Omega_0$, $\tilde{K}_{t-s}(\theta_s \omega) = K_t(\omega) - K_s(\omega) \downarrow K_t(\omega)$ as $s \downarrow$ because $K_0(\omega)=0$ for $ \omega \in \Omega_0$. Therefore $L_t(\omega) = K_t(\omega)$ if $\omega \in \Omega_0$ and $t>0$, and so $L$ and $A$ are indistinguishable on $]0,\infty[$. If $0 < s < r < t$ and $ \omega \in \Omega$, $\tilde{K}_{t -r -s}(\theta_{r+s}\omega)$ increases as either $r$ or $s$ decreases. Consequently for $\omega \in \Omega$, $r \rightarrow L_{t-r}(\theta_r \omega)$ is decreasing and right continuous on $[0,t[$, the interchange of limits being valid. In what follows $\omega \in \Omega$ is fixed, but suppressed in our notation. It is clear that $L$ inherits the properties $L_t \in \mathcal{F}^*_{t+}$ from $K$ and hence $L_{t-s}\circ\theta_s \in \mathcal{F}^*_{\ge s}$. Note that for fixed $\epsilon >0$, $ t \rightarrow \tilde{K}_{t-s}(\theta_s)$ is increasing on $]s, \infty[$. Hence $t \rightarrow L_t$ is increasing on $]0,\infty[$ and we may complete the definition of $L$ by defining $L_0 := \lim_{t \downarrow 0} L_t$. 
 
 \begin{align} \hspace{-.4in}\mathbf{Lemma} \ & (i) \ t \rightarrow L_t \mathit{\ is \ right \ continuous \ on \ } ]0,\infty[, \\ 
 & \notag (ii) \ L_{t+s} = L_t +L_s \circ \theta_t \mathit{ \ for \ } s,t \ge 0, \\
 & \notag (iii) \ L_{t-s} \circ \theta_t \in \mathcal{F}^*_{>s} \mathit{ \ and \ } L_0 \mathit{ \ is \ either \ zero \ or \ infinity }. \end{align}
 
 \begin{proof} (i) Suppose $0 < r < s < t < u$. We claim in all cases \[ K_u - K_r + K_t - K_s = K_t - K_r + K_u-K_s .\]
 
\noindent If $K_s < \infty$, this is certainly true. If $K_r < K_s = \infty$, then both sides are infinite, while if $K_r = \infty$, both sides are zero. Recall that $\infty - \infty =0$ by convention. Therefore \[ \tilde{K}_{u-r}(\theta_r) +\tilde{K}_{t-s}(\theta_s) = \tilde{K}_{t-r}(\theta_r) + \tilde{K}_{u-s}(\theta_s),\] and letting $r \downarrow 0$ we obtain, \[L_u + \tilde{K}_{t-s}(\theta_s) = L_t + \tilde{K}_{u-s}(\theta_s). \]
 
\noindent Fix $t>0$. If $L_t = \infty$, $L_u = \infty$ for $u \ge t$ and so $u \rightarrow L_u$ is right continuous on $[t, \infty[$.  If $L_t < \infty$, then $\tilde{K}_{t-s}(\theta_s) < \infty$ for $0 < s <t$ since it increases to $L_t < \infty$ as $s \downarrow 0$. Therefore \[ L_u = L_t - \tilde{K}_{t-s}(\theta_s) + \tilde{K}_{u-s}(\theta_s). \]
 
\noindent But $u \rightarrow \tilde{K}_{u-s}(\theta_s) = K_u -K_s$ is right continuous on $[s, \infty[$, if $K_s <\infty$ and equals zero if $K_s = \infty$ for $u >s$. Since $s<t$, $u \rightarrow L_u$ is right continuous on $[t, \infty[$ in this case also. Since $t>0$ is arbitrary, $u \rightarrow L_u$ is right continuous on $]0, \infty[$ and $L_0 = L_{0+}$ by definition, establishing (i). \\
 
 (ii) If $0 < r < u < t$ then $\tilde{K}_{t-r}(\theta_r) = \tilde{K}_{t-u}(\theta_u) + \tilde{K}_{u-r}(\theta_r)$ in all cases as is easily checked. Let $r \downarrow 0$ to obtain $L_t = \tilde{K}_{t-u} (\theta_u) +L_u$. If $0 \le s < t$ let $u \downarrow s$ with $s < u < t$ obtaining $L_t = L_{t-s}(\theta_s) + L_s$. Writing this for $s=0$ and letting $ t \downarrow 0$ gives $L_0 + L_0 = L_0$ and so either $L_0$ is zero or infinite. If $\omega \in \Omega_0$, $L_t(\omega) = K_t(\omega)$ for $t>0$ and letting $t \downarrow 0$ we find that $L_0 (\omega) = K_0(\omega) =0$. In particular $L$ and $K$, and hence $L$ and $A$, are indistinguishable. Finally since $s \rightarrow L_{t-s}(\theta_s)$ is right continuous on $[0,t[$ and $L_{t-s}\circ \theta_s \in \mathcal{F}_{\ge s}^*$ it follows that $L_{t-s}\circ \theta_s \in \mathcal{F}_{>s}^*$. This establishes both (ii) and (iii). $\square$ \\

 It is now clear that $L$ has all of the properties asserted in Theorem 5.27. It remains only to remove the condition (A.11) in order to complete the proof of Theorem 5.27. To this end suppose that $A$ satisfies the hypotheses of (5.27) and as a first step we replace it with an indistinguishable $A$ satisfying (A.2) as before. Since $\zeta ([\Delta]) =0$ (A.2 i) implies that $A_t([\Delta]) =0$ for all $t \ge 0$. Define
 
 \begin{equation} \tilde{A}_t := A_t 1_{ \{ t < \zeta\} } + A_{\zeta-} 1_{\{\zeta \le t \} } \end{equation}
 
 \noindent where $A_{0-} = 0$ by convention. Then $\tilde{A}$ satisfies (A.2i). Using the facts that $t \rightarrow A_t$ is constant of $[ \zeta, \infty [$ and that $A_t([\Delta]) =0$ for all $t$ it is not difficult to check that $A$ also satisfies (A.2 iii) by considering the five cases (i) $t+s < \zeta$, (ii) $s,t< \zeta \le t+s$, (iii) $s < \zeta \le t$, (iv) $t < \zeta \le s$, and (v) $t,s \ge \zeta$ separately. For example, if $s < \zeta \le t$, then $\zeta \circ \theta_t=0$ and so $ s \ge \zeta \circ \theta_t$. Therefore $\tilde{A}_t + \tilde{A}_s \circ \theta_t = A_{\zeta-} + 0=A_{t+s}$. On the other hand if $s+t < \zeta$, then $s < \zeta \circ \theta_t$ and so $\tilde{A}_t + \tilde{A}_s \circ \theta_t = A_t + A_s \circ \theta_t = A_{t+s} = \tilde{A}_{t+s}$ a.s. $P^m$. The remaining cases are treated similarly. Since $\zeta$ is an $(\mathcal{F}_t^0)$ stopping time it is easy to check that $\tilde{A}$ also satisfies (A.2ii).  It is obvious that $\tilde{A}$ satisfies (A.11) Consequently by what has been proved so far, there exists $\tilde{L}$ indistinguishable from $\tilde{A}$ satisfying (iv)-(vi) of (5.27).  \\
 
 Define $U = (A_\zeta-A_{\zeta-})1_{\{\zeta< \infty\}}$. Note that $A_{\zeta} < \infty $ if $\zeta < \infty$. Since $A_0 =0$ and $A_{0-} =0$ by convention $U=0$ on $\{\zeta =0\}$. Clearly $U \in \mathcal{F}^0$. Next define \[ V = \mathrm{ess}\lim_{t \uparrow \zeta} \sup U \circ \theta_t \mathrm{ \ if \ } \zeta >0, V =0 \mathrm{ \ if \ } \zeta =0. \]
 
\noindent Note that $V([\Delta]) =0$ and that $V \circ \theta_s = V$ if $s < \zeta$ while if $s \ge \zeta$, $\zeta \circ \theta_s =0$ and so $V \circ \theta_s =0$. Thus $V \circ \theta_s = V$. Also it follows from [DM IV; 38.1a] that $V \in \mathcal{F}^*$. See the argument used to show that $\Lambda_t \in \mathcal{F}_{t+}^*$ below (A.6). Moreover (A.2i) and (A.2 iii) imply that $U \circ \theta_s = U$ a.s. $P^m$ for each fixed $s$. Hence by the Fubini theorem for $P^m$ a.e. $\omega$, $U(\theta_s \omega) = U(\omega)$ for Lebesgue a.e. $s$. Consequently $P^m(U=V) =0$. Now define $H_t = V 1_{[\zeta, \infty[}(t)$. Since $V \circ \theta_s =V$ and $V=0$ on $\{\zeta =0\}$ it follows that $H$ satisfies conditions (iv), (v), and (vi) of (5.27). Clearly $H_t \in \mathcal{F}^*$. Suppose $t <r$. If $t < \zeta(\omega)$ then $t < \zeta(\omega) \wedge  r = \zeta(k_r \omega)$ and so $H_t(\omega) =0 = H_t(k_r \omega)$. If $\zeta(\omega) \le t$, then $k_r \omega = \omega$. Therefore if $t < r$, $H_t = H_t \circ k_r$ which implies that $H_t \in \cap_{r>t} k_r^{-1} \mathcal{F}^* = \mathcal{F}^*_{t+}$. Also $H_{t-s} \circ \theta_s \in \theta_s^{-1} \mathcal{F}^* = \mathcal{F}^*_{\ge s}$. Using the fact that $\zeta \circ \theta_t = (\zeta - t)^+$, one sees that $ s \rightarrow H_{t-s} \circ \theta_s$ is right continuous on $[0,t[$, and so, in fact, $H_{t-s} \circ \theta_s \in \mathcal{F}_{>s}^*$. Therefore $H$ satisfies the conclusion of Theorem 5.28, and hence so does $L:= \tilde{L} +H$. Since $U =V$ a.s. $P^m$, it follows from the Fubini theorem that a.s. $P^m$, $\tilde{A}_t+H_t = \tilde{A}_t + U 1_{[\zeta,\infty[}(t) = A_t$ for $t \ne \zeta$. But $\tilde{A}$, $H$, and $A$ are right continuous and so $\tilde{A} + H$ and $A$ are indistinguishable. Since $\tilde{A}$ and $\tilde{L}$ are indistinguishable, so are $L$ and $A$. At long last, the proof of Theorem 5.27 is complete.
 
  \end{proof}
  
  Final remark. Although we have formulated and proved Theorem 5.27 so as to be directly applicable to the proof of Theorem 5.28, it should be clear that it is also valid if $A$ is an additive functional in the usual sense of Markov processes. That is, if $A$ satisfies (i) of (5.27), $A_t \in \mathcal{F}_t$ and $A_{t+s} = A_s + A_t \circ \theta_s$ a.s $P^x$ for each $x$ for fixed $(s,t)$, then for each $x$, $A$ is $P^x$-indistinguishable from a process $L$ as in the conclusion of (5.27). The only manner the structure of $P^m$ was involved was in showing that if $F: \Omega \rightarrow \overline{\mathbb{R}}^+$, then $P^m(F) =0$ implies $P^m(F \circ \theta_t) =0$. See (A.1) and the end of the proof of (A.5), for example. But if $P^x(F) =0$ for all $x$, then $P^x(F\circ\theta_t) = P(P^{X(t)})(F) =0$. 


\begin{center}\textbf{Standard Reference Books}\end{center}

[BG] Blumethal, R.M. and Getoor, R.K.: \textit{Markov Processes and Potential Theory.} Academic Press, New York, 1968. Dover reprint, 2007. \\

[DM] Dellacherie, C. and Meyer, P.-A: \textit{Probabilit\'e et Potentiel}. Ch. I-IV (1975) English translation 1978, Ch. V-VIII (1980); Ch IX-XI (1983); Ch. XII-XVI (1987). Herman, Paris.\\

[DMM] Dellacherie, C., Maisonneuve, and B., Meyer, P.-A. \textit{Probabilit\'e et Potentiel. Chapitres XVII-XXIV}. Hermann, Paris, 1992.  \\

[G] Getoor, R.K.: \textit{Excessive Measures}. Birkha\"user, Boston, 1990. \\

[S] Sharpe, M.: \textit{General Theory of Markov Processes}. Academic Press, Boston, 1988. \\

\begin{center}\textbf{General References }\end{center} 

\vspace{.1in}

[A73] Azema, J: \textit{Th\'eorie g\'en\'erale des processus et retournement du temps.} Ann. Sci. \'Ecole Norm. Sup. \textbf{6} (1973) 459-519. \\

[CG79] Chung, K.L. and Glover, J: Left continuous moderate Markov processes. \textit{Z. Wahrscheinlichkeitstheorie verw. Geb.} \textbf{49} (1979), 237-248. \\

[CW05] Chung, K.L. and Walsh, J.B.: \textit{Markov processes, Brownian motion, and time symmetry}, (2nd edition). Springer, New York, 2005. \\

[D88] Dellacherie, C.: Autour des ensembles semi-polaires. In \textit{Seminar on Stochastic Processes,} 1987. Birkh\"auser Boston, Boston, MA, 1988, pp. 65-92 \\

[DG85] Dynkin, E.B. and Getoor, R.K.: \textit{Additive Functionals and Entrance Laws} J. Functional Anal. \textbf{62} (1985), 221-265. \\

[FM86] Fitzsimmons, P.J. and Masionneuve, B.: \textit{Excessive measures and Markov processes with random birth and death}. Probab. Theory Relat. Fields \textbf{72} (1986) 319-336. \\

[F87] Fitzsimmons, P.J.: \textit{Homogeneous random measures and a weak order for the excessive measures of a Markov process.} Trans. Amer. Math. Soc. \textbf{303} (1987) 431-478. \\

[FO84] Folland, G.B.: \textit{Real Analysis Modern Techniques and Their Applications}. Wiley-Interscience, New York, 1984. \\

[G75a] Getoor, R.K.: \textit{Markov processes: Ray processes and right processes.} Lecture Notes in Mathematics, \textbf{440}. Springer-Verlag, Berlin-New York, 1975. \\

[G75b] Getoor, R.K.: \textit{On the construction of kernels.} Lecture Notes in Mathematics, \textbf{465}. Springer-Verlag, Berlin, 1975. \\

[G87] Getoor, R.K.: \textit{Measures that are translation invariant in one coordinate}. Seminar on Stochastic Processes, 1986. Birkh\"auser, Boston, 1987. \\

[H66] Hunt, G.A.: \textit{Martingales et Processus de Markov.} Dunod, Paris, 1966. \\

[J78] Jeulin, Th.: \textit{Compactification de Martin d'un processus droit.} Z. Wahrscheinlichkeitstheorie verw. Geb. \textbf{42} (1978), 229-260. \\

[M71] Meyer, P.-A: \textit{Le retournement du temps, d'apr\`es Chung et Walsh}. Lecture Notes in Math. \textbf{191}, Springer-Verlag, Berlin. 1971. \\

[M74] Meyer, P.-A: \textit{Ensembles al\'eatoire markoviens homog\'enes, II}. Lecture Notes in Math. \textbf{381}, Springer-Verlag, Berlin, 1974. \\

[W72] Walsh, J.B.: \textit{The perfection of multiplicative functionals.} Lecture Notes in Math. 258, Springer-Verlag, Berlin, 1972.

\end{document}